\documentclass[letterpaper,11pt]{amsart}

\usepackage{graphicx,color,epsfig}
\usepackage{amsfonts,amsmath,amssymb,amsthm}
\usepackage{hyperref}
\usepackage{comment}
\usepackage[normalem]{ulem}
\usepackage{mathtools}

\newtheorem{thm}{Theorem}[section]
\newtheorem{prop}[thm]{Proposition}
\newtheorem{lem}[thm]{Lemma}
\newtheorem{cor}[thm]{Corollary}

\newtheorem{asm}{Assumption}

\theoremstyle{remark}
\newtheorem{rem}[thm]{Remark}

\theoremstyle{definition}
\newtheorem{defn}{Definition}

\newcommand{\ra}{\rightarrow}
\newcommand{\Ra}{\Rightarrow}

\newcommand{\Q}{\mathbb Q}     
\newcommand{\R}{\mathbb R}     
\newcommand{\Z}{\mathbb Z}     

\renewcommand{\a}{\alpha}
\renewcommand{\b}{\beta}

\renewcommand{\d}{\delta}
\renewcommand{\th}{\theta}
\renewcommand{\k}{\kappa}

\newcommand{\e}{\varepsilon}
\newcommand{\del}{\partial}

\renewcommand{\l}{\lambda}
\renewcommand{\L}{\Lambda}

\newcommand{\s}{\sigma}

\newcommand{\bigo}{\mathcal{O}}

\newcommand{\fl}[1]{\lfloor #1 \rfloor}  
\newcommand{\ind}[1]{ \mathbf{1}_{ \{ #1 \} } } 

\newcommand{\w}{\omega}              
\renewcommand{\P}{\mathbb{P}}        
\newcommand{\E}{\mathbb{E}}          
\newcommand{\vp}{\mathrm{v}_0}       

\newcommand{\tw}{\tilde\omega}

\newcommand{\be}{\begin{equation}}
\newcommand{\ee}{\end{equation}}
\DeclareMathOperator{\Var}{Var}   \DeclareMathOperator{\Cov}{Cov}  
\DeclareMathOperator{\supp}{supp}

\def\Pv{\mathbf{P}}  \def\Ev{\mathbf{E}}  \def\Varv{\mathbf{Var}} \def\Covv{\mathbf{Cov}} 

\title[Hydrodynamic limits for independent RWRE]{Hydrodynamic limit for a system of independent, sub-ballistic random walks in a common random environment}

\author{Milton Jara}
\address{Milton Jara\\IMPA\\Estrada Dona Castorina, 110\\Horto, Rio de Janeiro\\Brazil}
\email{mjara@impa.br}

\author{Jonathon Peterson}
\address{Jonathon Peterson\\Purdue University\\Department of Mathematics\\150 N University Street\\West Lafayette, IN  47907\\USA}
\email{peterson@math.purdue.edu}
\urladdr{http://www.math.purdue.edu/~peterson}
\thanks{J. Peterson was partially supported by NSA grant H98230-13-1-0266.}


\date{\today}

\begin{document}

\begin{abstract}
We consider a system of independent random walks in a common random environment. Previously, a hydrodynamic limit for the system of RWRE was proved under the assumption that the random walks were transient with positive speed \cite{pRWRESystem}. In this paper we instead consider the case where the random walks are transient but with a sublinear speed of the order $n^\kappa$ for some $\kappa \in (0,1)$
and prove a quenched hydrodynamic limit for the system of random walks with time scaled by $n^{1/\k}$ and space scaled by $n$. The most interesting feature of the hydrodynamic limit is that the influence of the environment does not average out under the hydrodynamic scaling; that is, the asymptotic particle density depends on the specific environment chosen. 
The hydrodynamic limit for the system of RWRE is obtained by first proving a hydrodynamic limit for a system of independent particles in a directed trap environment. 
\end{abstract}

\maketitle

\section{Introduction}

\subsection{Overview and significance of the main results}

In this section we give an informal discussion of the main results of this article in terms of a simple model. More rigorous definitions and a more general setting can be found in subsequent sections.

Let us consider the following simple model of a random walk in random environment. We have two types of unfair coins, $A$ and $B$. Coins of type $A$ have heads probability $p >\frac{1}{2}$ and coins of type $B$ have heads probability $q < \frac{1}{2}$. Initially we associate to each site of the lattice $\mathbb{Z}$ a coin of one of the types $A$ and $B$ with equal probability. Then we run a simple random walk in the following way. When the walk is at site $x$, it flips the coin associated to $x$. It moves to the right if the coin gets heads and to the left if the coin gets tails. Let us call $(X_n)_{n \geq 0}$ the Markov chain obtained in this way. It will be useful to introduce the parameters $r = \frac{1-p}{p}$ and $s=\frac{1-q}{q}$. In \cite{sRWRE}, it is proved that $(X_n)_{n \geq 1}$ is transient to the right as soon as $p+q >1$ (that is, $rs<1$), and that the asymptotic speed of the chain is equal to zero if $r+s \geq 2$, see \eqref{XnLLN} below. 
When the walk is transient to the right with asymptotic speed zero (i.e., $rs<1$ and $r+s\geq 2$), there exists a parameter $\kappa \in (0,1)$ such that $ r^\kappa + s^\kappa =2$. If $\frac{\log r}{\log s}$ is irrational (the so-called {\em non-lattice condition}), the scaling limit
\[
\lim_{n \to \infty} \frac{X_{n^{1/\kappa}t}}{n}
\]
exists in the weak quenched sense, see \cite{dgWQL,estzWQL,psWQLTn,psWQLXn}. The scaling limit is a process we call {\em directed trap process}. In one sentence, the directed trap process is a Markov process in $\mathbb{R}$ generated by the (random) operator $\frac{d}{d \s}$, where $\s$ is a realization of a $\kappa$-stable subordinator. The random subordinator $\s$ can be understood as the weak scaling limit of the random environment. Our aim in this article is to answer the following question: what can we say about a system of independent random walks following the evolution described above on a common random environment? We will prove that the {\em hydrodynamic limit} of this system of particles is given by the PDE 
\begin{equation}\label{hdlequation}
\frac{\partial u}{\partial t} = -\frac{d u}{d \s },
\end{equation}
where the right side of \eqref{hdlequation} represents the operator $\frac{d}{d\s}$ applied to the function $x\mapsto u(t,x)$ for $t$ fixed. 
One of the striking features of this hydrodynamic limit is the dependence of the limiting equation on the random environment (that is, the operator $\frac{d}{d\s}$ appearing in the PDE is random). This features is not new in the literature; similar results have been proved in \cite{jltTrapHydro,fjlEPRC}, and our result is a natural follow-up. However, one important difference in terms of hydrodynamic limits with \cite{jltTrapHydro,fjlEPRC} is that our limiting equation is {\em hyperbolic}, while the equations in \cite{jltTrapHydro,fjlEPRC} are {\em diffusive}. This is not evident from the scaling exponents of the equations, but a closer look at the proofs of the hydrodynamic limits in \cite{jltTrapHydro,fjlEPRC} reveal the use of tools commonly used in the proof of diffusive hydrodynamic limits. Somehow this is summarized by the self-adjointness of the operators $\frac{d}{dx} \frac{d}{d\s}$, $\frac{d}{d\s} \frac{d}{dx}$ with respect to suitable Hilbert spaces.
Since the limiting equation in the current paper is hyperbolic, the techniques from the previous papers do not seem applicable and we instead prove a hydrodynamic limit by a series of couplings with progressively simpler systems of particles. 

In proving the hydrodynamic limit of the system of independent particles, we need to establish various properties of the limiting equation \eqref{hdlequation} which are interesting in their own right. The point is that it is not
necessarily clear
in which sense the probability density function associated to the directed trap process satisfies the proposed hydrodynamic equation. This is a delicate point which is of independent interest and we devote a considerable part of this article to answer this question. The same problem is already present in \cite{jltTrapHydro,fjlEPRC} and the answer in these three cases (ours and theirs) are different.

\subsection{Transient random walks in random environment}

The main object of study in this paper is a system of independent one-dimensional random walks in a common random environment. 
First we recall the standard model of one-dimensional random walks in a random environment. 
In this model an \emph{environment} is a sequence $\w = \{\w_x\}_{x \in \Z} \in [0,1]^\Z =: \Omega$. 
For a fixed environment $\w$ and any $x \in \Z$, a random walk in the environment $\w$ started at $x$ is a Markov chain $(X_n)_{n\geq 0}$ with distribution $P_\w^x$ defined by 
\[
 P_\w^x(X_0 = x) = 1, \quad\text{and}\quad P_\w^x( X_{n+1}=z \, | \, X_n=y ) = 
\begin{cases}
 \w_y & \text{if } z=y+1 \\
 1-\w_y &  \text{if } z=y-1 \\
0 &\text{otherwise}. 
\end{cases}
\]
A RWRE is constructed by first choosing an environment $\w$ according to some probability measure $P$ on $(\Omega,\mathcal{F})$ (where $\mathcal{F}$ is the natural Borel $\s$-field) and then generating a random walk in the environment $\w$ as above. 
Oftentimes we will be interested in the case $X_0 = 0$ and so we will use the notation $P_\w$ instead of $P_\w^0$. 
The measure $P_\w$ of the random walk conditioned on the environment $\w$ is called the \emph{quenched} law of the RWRE. By averaging over the law $P$ on environments one obtains the \emph{averaged} law $\P(\cdot) = E_P[ P_\w(\cdot) ]$.  

A system of independent RWRE can be constructed as follows. First, let $\w$ be chosen according to the fixed measure on environments $P$. Then, for the fixed environment $\w$ we can let $\{(X^{x,j}_n)_{n\geq 0} \}_{x\in\Z,j\geq 1}$ be an independent family of random walks in the environment $\w$ such that $X^{x,j}_0=x$ for every $x\in \Z$ and $j\geq 1$. In a slight abuse of notation we will let $P_\w$ be the quenched joint distribution of all of the random walks. 
Also, at times we will be interested in the path of a single random walk started from $x$ and so we will use $X_n^x$ instead of $X_n^{x,1}$ in these cases. 

While the above construction enables us to start infintely many random walks at every point, we will be interested in the case where there are finitely many random walks started at every $x\in \Z$. 
We will let $\chi_0(x)$ denote the number of random walks initially started at location $x\in\Z$ and will refer to $\chi_0 = \{\chi_0(x)\}_{x\in\Z}$ as the \emph{initial configuration}. 
If we only follow the paths of these particles then
the configuration after $n$ steps $\chi_n = \{\chi_n(x)\}_{x\in\Z}$ is given by
\[
 \chi_n(x) = \sum_{y\in\Z} \sum_{j\leq \chi_0(y)} \ind{X_n^{y,j} = x}. 
\]

For a fixed environment $\w$, if the initial configuration of particles $\chi_0$ has distribution $\mu \in \mathcal{M}_1(\Z_+^\Z)$, where $\mathcal{M}_1(\Z_+^\Z)$ is the space of probability measures on $\Z_+^\Z$, then we will denote the quenched law of the system of particles by $P_\w^\mu$. In much of what follows below we will be interested in cases where the measure on initial configurations depends on the environment $\w$. That is, we will allow for $\mu = \mu(\w)$ to be a measurable function from $\Omega$ to the space $\mathcal{M}_1(\Z_+^\Z)$ equipped with the topology of weak convergence of probability measures.   
For such initial configurations depending on the environment it makes sense to define the averaged measure on the system of random walks by $\P^\mu(\cdot) = E_P[P_\w^{\mu(\w)}(\cdot)]$. 

The main goal of the current paper will be to prove a hydrodynamic limit theorem for the system of independent RWRE. 
In this paper we will make the following assumptions on the distribution $P$ on the environment $\w$. 
\begin{asm}\label{asmiid}
 The distribution $P$ on the environment is such that $\w = \{\w_x\}_{x\in\Z}$ is an i.i.d.\ sequence. 
\end{asm}
Many of the properties of RWRE can be stated in terms of the distribution of the statistic $\rho_x:=\frac{1-\w_x}{\w_x}$ of the environment. 
To this end, our second assumption is the following. 
\begin{asm}\label{asmt}
 $E_P[\log \rho_0] < 0$. 
\end{asm}
It is known that under Assumptions \ref{asmiid} and \ref{asmt} the RWRE are transient to the right \cite{sRWRE}. 
In addition, Solomon proved in \cite{sRWRE} the following result on the limiting speed of RWRE satisfying Assumptions \ref{asmiid} and \ref{asmt}. 
\begin{equation}\label{XnLLN}
 \P\left( \lim_{n\ra\infty} \frac{X_n}{n} = \vp \right) = 1, \quad \text{where} \quad \vp =  \begin{cases} \frac{1-E_P[\rho_0]}{1+E_P[\rho_0]} & \text{if } E_P[\rho_0] < 1 \\ 0 & \text{if } E_P[\rho_0] \geq 1. \end{cases} 
\end{equation}
In \cite{pRWRESystem}, the following hydrodynamic limit was proved under the additional assumption that the RWRE are ballistic 
(that is, $\vp>0$ or equivalently $E_P[\rho_0] < 1$).
\begin{thm}[Theorem 1.4 in \cite{pRWRESystem}]
 Let Assumptions \ref{asmiid} and \ref{asmt} hold, and additionally assume that $E_P[\rho_0] < 1$ so that the limiting speed $\vp>0$. 
Let $\mathcal{C}_0$ be the set of continuous functions on $\R$ with compact support. 
If there is a bounded function $u_0(x)$ and a sequence of initial configurations $\{\chi_0^n\}_{n\geq 1}$ such that 
\begin{equation}\label{oldhdl1}
 \lim_{n\ra\infty} \frac{1}{n} \sum_{x \in \Z} \chi_0^n(x) \phi(\tfrac{x}{n}) = \int_\R u_0(x) \phi(x) \, dx, \quad  \forall \phi \in \mathcal{C}_0, 
\end{equation}
then for any fixed $t>0$
\begin{equation}\label{oldhdl2}
  \lim_{n\ra\infty} \frac{1}{n} \sum_{x \in \Z} \chi_{nt}^n(x) \phi(\tfrac{x}{n}) = \int_\R u_0(x-t\vp) \phi(x) \, dx, \quad  \forall \phi \in \mathcal{C}_0. 
\end{equation}
The limits \eqref{oldhdl1} are either both almost sure limits under the averaged measure or limits in probability under the quenched measure (for almost every environment $\w$). 
\end{thm}

In this paper we will be interested in studying systems of RWRE that are transient but with a sublinear speed (i.e., $\vp= 0$). As noted in \eqref{XnLLN} above, RWRE that are transient to the right have $\vp=0$ if and only if $E_P[ \rho_0 ] \geq 1$, but for our results we will need to assume the following slightly stronger assumption. 
\begin{asm}\label{asmk}
There exists $\k \in (0,1)$ such that $E_P[ \rho_0^\k ] = 1$.
\end{asm}
We will also need the following technical assumptions. 
\begin{asm}\label{asmnl}
 $\log \rho_0$ is a non-lattice random variable under the distribution $P$ on environments, and $E_P[\rho_0^\k \log \rho_0] < \infty$ where $\k$ is the parameter from Assumption \ref{asmk}. 
\end{asm}

\begin{rem}
 For RWRE that are transient to the right (i.e., $E_P[\log \rho_0] < 0$) the parameter $\k>0$ defined by $E_P[\rho_0^\k] = 1$ plays a major role in many of the known results (note that in some of these results the parameter $\k\geq 1$). For instance, $\k$ determines the rates of decay of both the averaged and quenched large deviation slowdown probabilities \cite{dpzTE,gzQSubexp}, and in the limiting distributions proved in \cite{kksStable} both the scaling and the type of limiting distribution are determined by the parameter $\k$. 
For instance, Assumptions \ref{asmiid}--\ref{asmnl} imply that 
\[
 \lim_{n\ra\infty} \P\left( \frac{X_n}{n^\kappa} \leq x \right) = 1- L_\kappa(x^{-1/\k}), \quad \forall x>0, 
\]
where $L_\kappa$ is the distribution function for a totally skewed to the right $\k$-stable random variable. 
\end{rem}

\begin{rem}
The technical conditions in Assumption \ref{asmnl} are needed to obtain certain precise tail asymptotics that are needed for our results. These technical conditions were also needed for the averaged limiting distributions of transient RWRE in \cite{kksStable} and also more recently for the results on the weak quenched limiting distributions in \cite{dgWQL,estzWQL,psWQLTn,psWQLXn}. 
\end{rem}


\subsection{Hydrodynamic limits for systems of RWRE}

For any environment $\w$, define 
\begin{equation}\label{gdef}
 g_\w(x) = \frac{1}{\w_x}\left( 1 + \rho_{x+1} + \rho_{x+1}\rho_{x+2} + \rho_{x+1}\rho_{x+2}\rho_{x+3}+\cdots\right). 
\end{equation}
If the distribution on environments satisfies Assumptions \ref{asmiid} and \ref{asmt}, then it can be shown that $g_\w(x)  = E_\w\left[ \sum_{n\geq 0} \ind{X^x_n = x} \right]$ for $P$-a.e.\ environment $\w$ - that is, $g_\w(x)$ is the expected number of times the random walk started at $x$ visits $x$ (including the visit at time 0). 
The functions $g_\w$ are useful for constructing stationary distributions for the systems of independent RWRE in the environment $\w$. 
To this end, for any $\a>0$ let $\bar\mu_\a(\w) \in \mathcal{M}_1(\Z_+^\Z)$ be defined by 
\[
\bar\mu_\a(\w) = \bigotimes_{x \in \Z} \text{Poisson}( \a g_\w(x) ). 
\]
It was shown in \cite{pRWRESystem} (see also \cite{clEPRE}) that $\bar\mu_\a(\w)$ is a stationary distribution for the system of RWRE for any $\a>0$. That is, 
\[
 P_\w^{\bar\mu_\a(\w)}( \eta_n \in \cdot ) = P_\w^{\bar\mu_\a(\w)}( \eta_0 \in \cdot) = \bar\mu_\a(\w)(\cdot), \quad \forall n\geq 0,
\]
for $P$-a.e.\ environment $\w$. 

The hydrodynamic limit proved in this paper will describe the behavior of the system of particles when the initial configurations are what may be considered ``locally stationary''. 
For a continuous function $u:\R \ra [0,\infty)$ and $\w\in \Omega$ and $n\geq 1$ fixed let $\mu_u^{n}(\w) \in \mathcal{M}_1(\Z_+^\Z)$ be the measure on configurations given by 
\begin{equation}\label{rwlocstat}
 \mu_u^{n}(\w) = \bigotimes_{x\in\Z} \text{Poisson}\left( u(\tfrac{x}{n}) g_\w(x)  \right).
\end{equation}
We call these configurations locally stationary because the distribution of the configuration in a small neighborhood of $x \approx yn$ is approximately the same under $\mu_u^{n}(\w)$ and $\bar\mu_{u(y)}(\w)$. 

If $E_P[\rho_0] < 1$ (equivalently $\k > 1$), then it follows from Assumption \ref{asmiid} that $E_P[g_\w(x)] = \frac{1+E_P[\rho_0]}{1-E_P[\rho_0]} = \vp^{-1} < \infty$. 
From this it is easy to see that if $\chi_0^n$ is a sequence of initial configurations with $\chi_0^n \sim \mu_u^n(\w)$, then the configurations $\chi_0^n$ satisfy the condition \eqref{oldhdl1} for the hydrodynamic limit when $E_P[\rho_0] < 1$ (if $\k>1$ then the limit holds in probability with respect to the quenched measure; if $\k>2$ then it can be shown that the limit holds almost surely with respect to the averaged measure). 
Informally, this says that by scaling space by $n$ and giving each particle a mass of $n^{-1}$, the empirical distribution of particles converges to the deterministic measure $\vp^{-1} u(x) \, dx$. 

If $\k \leq 1$ then $E_P[g_\w(x)] = \infty$, and so one can no longer hope for condition \eqref{oldhdl1} to hold when the initial configurations have distribution $\mu_u^n(\w)$. 
However, it will follow from our main results below that when $\k \in (0,1)$, 
\begin{equation}\label{weakiclim}
\lim_{n\ra\infty} \P^{\mu_u^n}\left( \frac{1}{n^{1/\k}} \sum_x \chi_0(x) \phi(x/n) \in \cdot \right) = \Pv\left( \int_\R u(x) \phi(x) \, \s_W(dx) \in \cdot \right), \quad \forall \phi\in \mathcal{C}_0, 
\end{equation}
where $\s_W=\{\s_W(x)\}_{x\in \R}$ is a two-sided $\k$-stable subordinator with distribution $\Pv$.
\begin{rem}
In a slight abuse of notation, we will use $\s_W(x)$ to denote a non-decreasing c\`adl\`ag function of $x$ with $\s_W(0) = 0$ and will use $\s_W(dx)$ to denote the corresponding measure on $\R$ given by $\s_W((a,b]) = \s_W(b)-\s_W(a)$. 
We explain briefly the notation $\s_W$ used here and in the rest of the paper. Under the measure $\Pv$, $W$ is non-homogeneous Poisson point process on $\R\times (0,\infty)$ with intensity measure $\l y^{-\k-1} \, dx\,dy$ for some $\l>0$ and then the measure $\s_W(dx)$ is defined by $\s_W(A) = \iint_{A\times (0,\infty)} y \, W(dx \, dy)$. 
\end{rem}
\begin{rem}
Note that \eqref{weakiclim} differs from \eqref{oldhdl1} in two respects. First of all, the necessary scaling gives mass $n^{-1/\k}$ to each particle, and secondly the limiting scaled empirical measure is a \emph{random} measure $u(x) \, \s_W(dx)$ instead of a deterministic measure. 
\end{rem}
\begin{rem} 
A proof of \eqref{weakiclim} can be obtained with a little bit of work using results from \cite{dgWQL}. However, we will not give this argument here since this will also follow from the proof of our main result below where we will show a similar convergence for the empirical measure of the system of random walks as they evolve. 
\end{rem}

Our main result in this paper is a hydrodynamic limit for systems of RWRE when $\k \in (0,1)$. 
\begin{thm}\label{th:RWREhdl}
 Let Assumptions \ref{asmiid}--\ref{asmnl} hold with $\k \in (0,1)$, and let $u$ be a nonnegative, continuous function with compact support on $\R$. 
Then, for
any function $\phi(t,x)$ on $\R_+ \times \R$ that is continuous with compact support, 
\[
  \lim_{n\ra\infty} \P^{\mu_u^n}\left( \frac{1}{n^{1/\k}} \int \sum_x \chi_{tn^{1/\k}}(x) \phi(t, \tfrac{x}{n}) \, dt \in \cdot \right) = \Pv\left( \iint u_W(t,x) \phi(t,x) \, \s_W(dx)\, dt \in \cdot \right),
\]
where $\s_W$ is a two-sided $\k$-stable subordinator with distribution $\Pv$ and the function $u_W(t,x)$ 
satisfies $u_W(0,\cdot) \equiv u(\cdot)$ and 
\begin{equation}\label{PDEdtdW}
\begin{split}
 - \frac{\del}{\del t} u_W(t,x) &= 
\begin{cases}
 \lim_{h \ra 0^-} \frac{u_W(t,x+h)-u_W(t,x)}{ \s_W(x+h)-\s_W(x) } &  \text{if } t = 0 \\
 \lim_{h \ra 0} \frac{u_W(t,x+h)-u_W(t,x)}{ \s_W(x+h)-\s_W(x) } &  \text{if } t > 0, 
\end{cases}
,\qquad \forall x \in \mathcal{J}_W, 
\end{split}
\end{equation}
where $\mathcal{J}_W = \{x: \, \s_W(x) - \s_W(x-) > 0 \}$ is the set of the locations where $\s_W$ has a jump. 
\end{thm}

\begin{figure}
 \begin{tabular}{ccc}
 \includegraphics[scale=0.35]{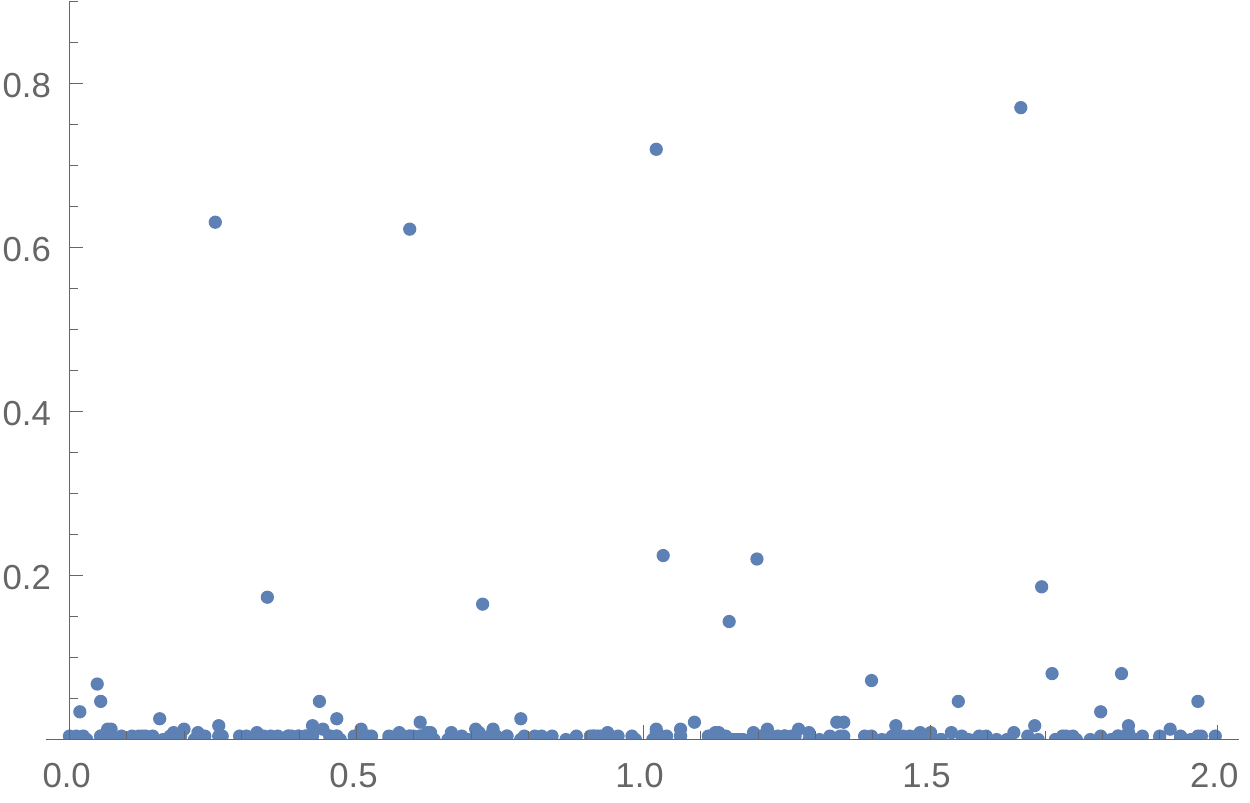} 
& \includegraphics[scale=0.35]{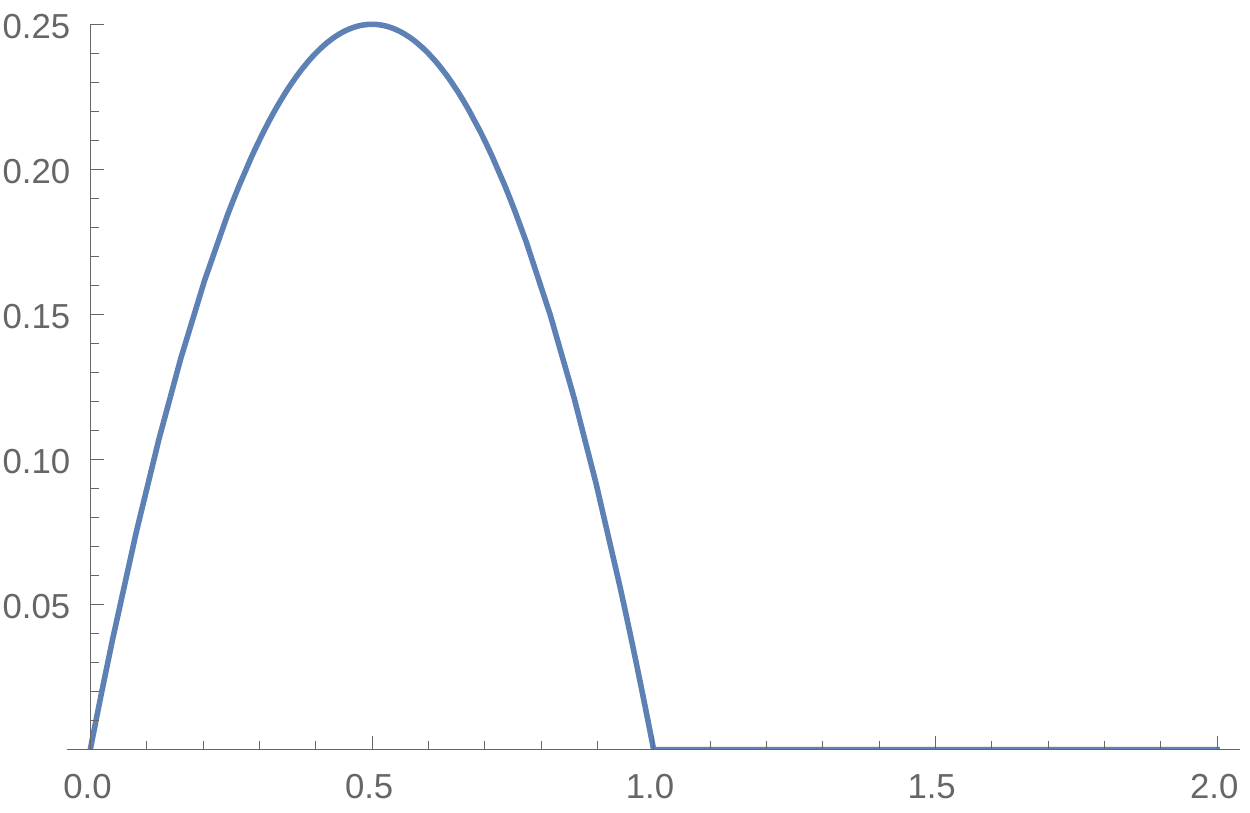}
& \includegraphics[scale=0.35]{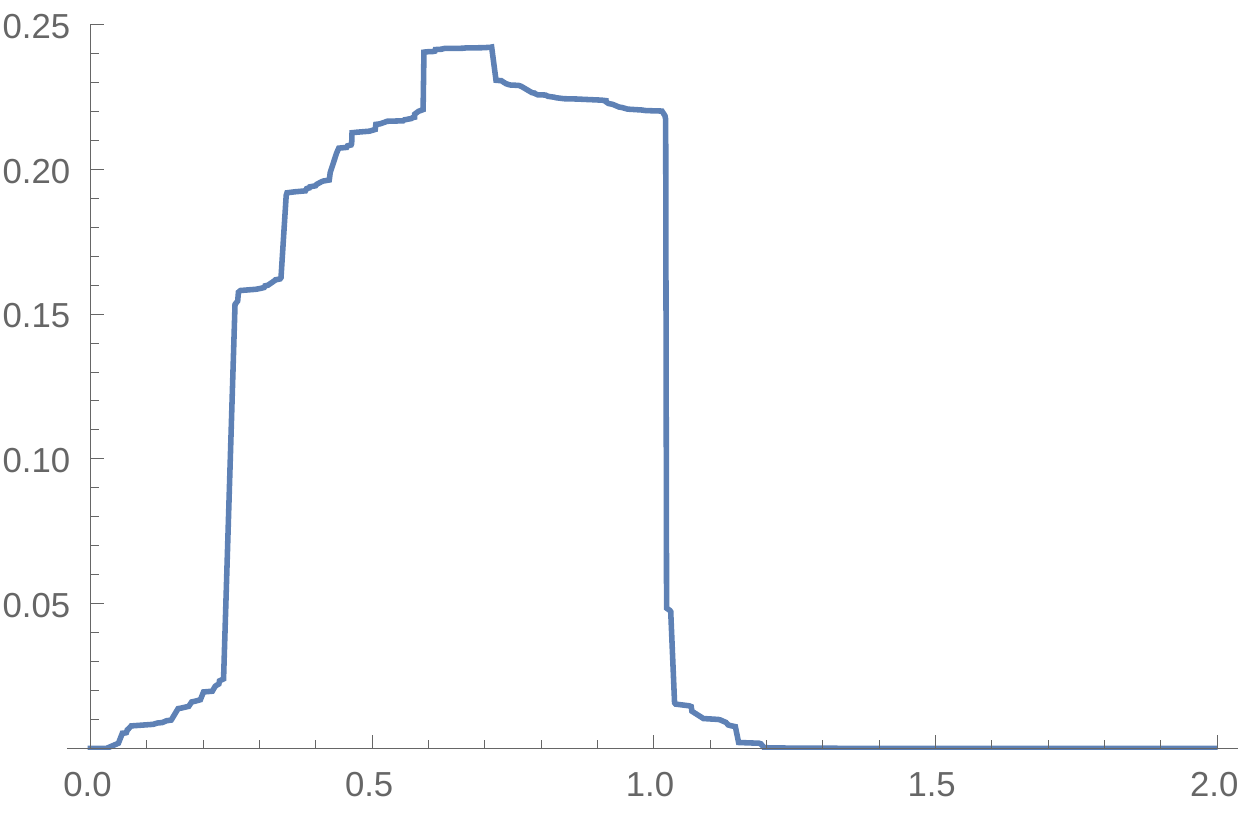}\\
  \includegraphics[scale=0.35]{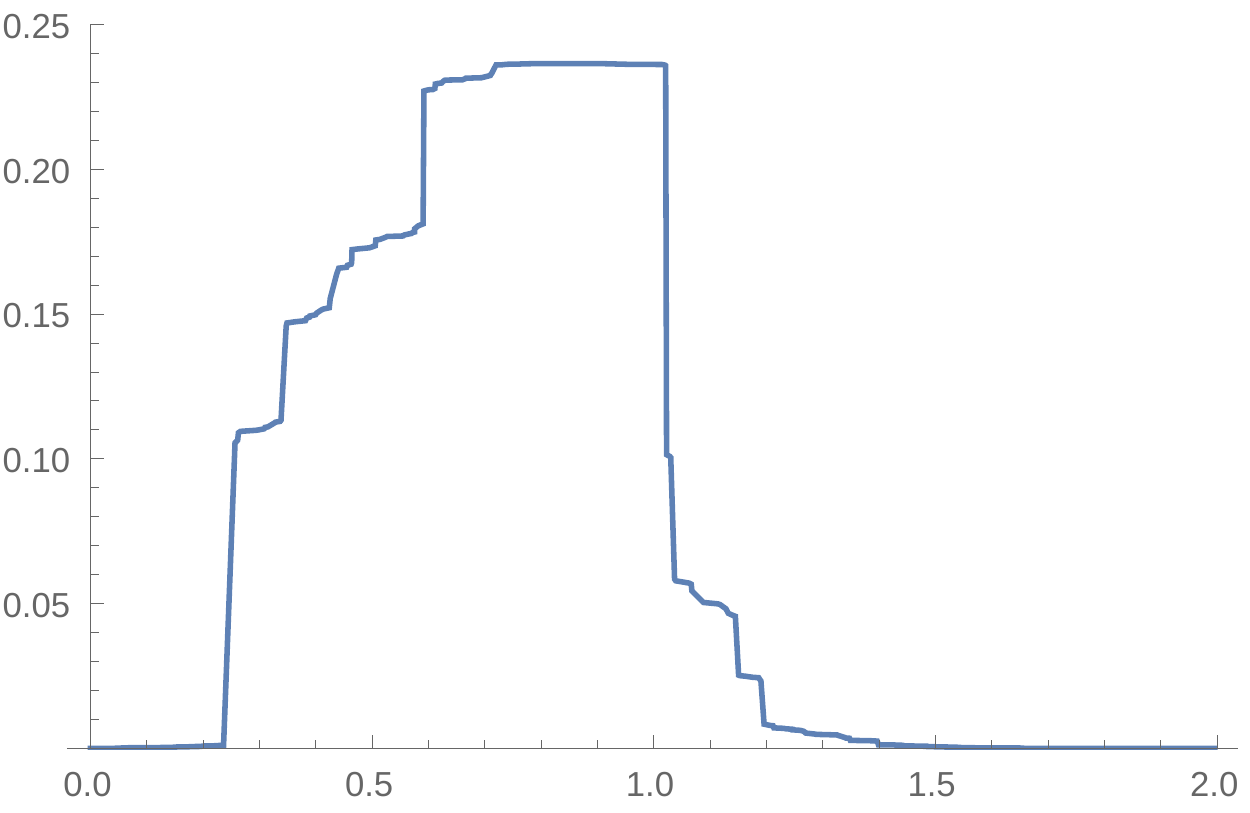}
& \includegraphics[scale=0.35]{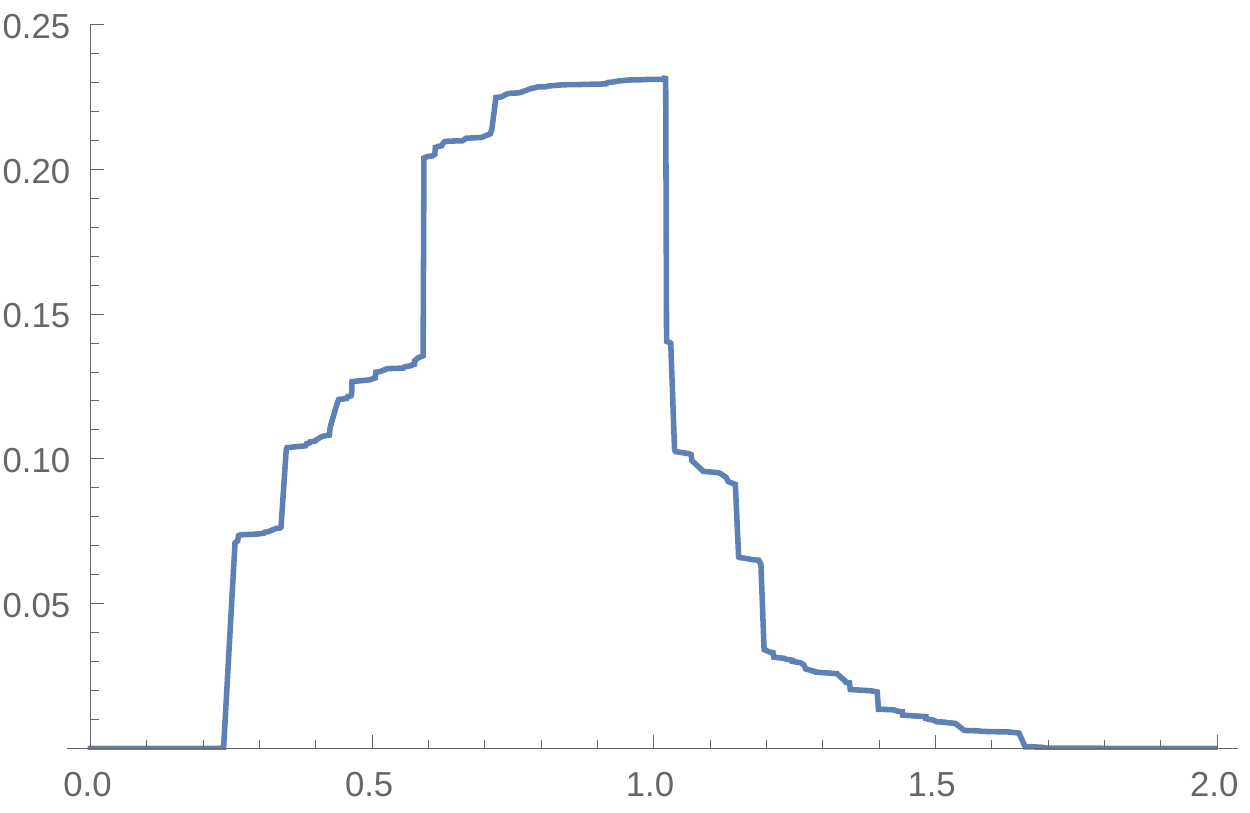}
& \includegraphics[scale=0.35]{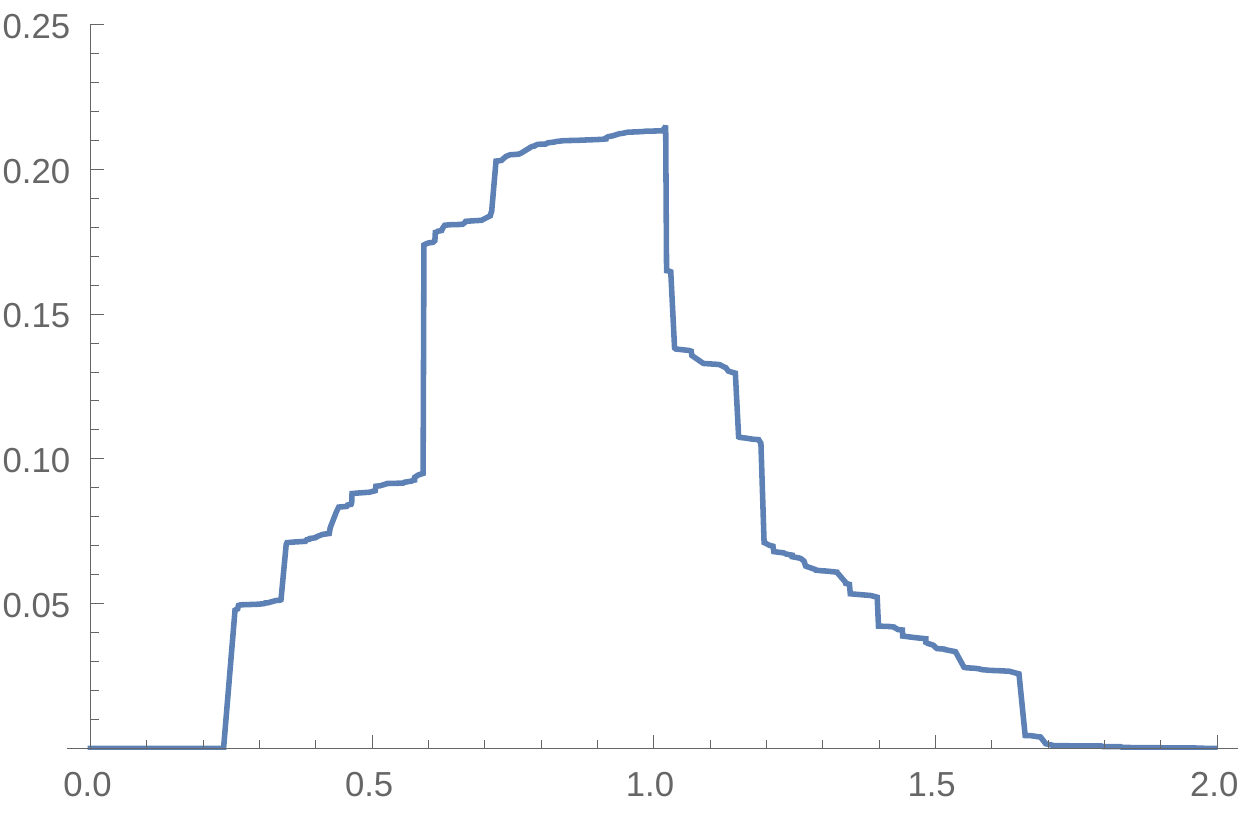}
 \end{tabular}
\caption{A simulation of the function $u_W(t,x)$ which appears in statement of Theorem \ref{th:RWREhdl} (and also in Theorem \ref{th:dthdl} below). 
The top left frame is a Poisson point process $W$ on $\R \times (0,\infty)$ with intensity measure $ y^{-\k-1} \, dy \, dx$ with $\k=0.7$. The succeeding frames show the function $u_W(t,x)$ at times $t=0,0.25,0.5,0.75,$ and $1.0$, respectively, when the initial configurations are given by the function $u(x) = x(1-x)\ind{x \in [0,1]}$.
For this simulation, a truncated point process was simulated on $\R \times [0.001,\infty)$ and since this point process has only finitely many atoms a corresponding approximation of $u_W$ can then be calculated by solving a finite system of ordinary differential equations related to the differentiability properties of $u_W$ in \eqref{PDEdtdW}.  
}\label{fig:uW}
\end{figure}

\begin{rem}
Both $u_W$ and $\s_W$ are defined in Section \ref{sec:dt} in terms of a point process $W = \sum_k \delta_{(x_k,y_k)}$ on $(0,\infty) \times \R$ (in the proof of Theorem \ref{th:RWREhdl} $W$ will be a random Poisson point process with intensity measure $\l y^{-\k-1} \, dx \, dy$).  
The function $\s_W$ in Theorem \ref{th:RWREhdl} is then identified with the atomic measure $\sum_k y_k \delta_{x_k}$ on $\R$ with support $\mathcal{J}_W = \{x_k\}_k$, and the function $u_W$ is defined probabilistically in \eqref{uWdef} using a stochastic process which evolves on $\mathcal{J}_W$. 
This probabilistic formulation defines $u_W(t,x)$ for all $x\in\R$ in such a way that it is c\`adl\`ag in $x$ with jumps only at points $x \in \{x_k\}_k = \mathcal{J}_W$.  Moreover, using this probabilistic formulation, it is shown in Section \ref{sec:uW} that the size of the jump $u_W(t,x_k) - u_W(t,x_k-)$ is proportional to $y_k = \s_W(x_k)-\s_W(x_k-)$. A simulation of the function $u_W(t,x)$ demonstrating these properties is shown in Figure \ref{fig:uW}.
\end{rem}

\begin{rem}
 The hydrodynamic limit in Theorem \ref{th:RWREhdl} is somewhat non-standard in that the limiting empirical measure of the particles is described by a \emph{random} measure $u_W(t,x) \, \s_W(dx)$. As will be seen from the proof of the theorem, the reason for this is that the effect of the environment $\w$ does not ``average out'' in the hydrodynamic scaling. 
We note that somewhat similar results in which the environment survives in the hydrodynamic limit were obtained previously for systems of independent particles in the Bouchaud trap model \cite{jltTrapHydro} and also for the exclusion process with random conductances \cite{fjlEPRC}. 
\end{rem}

\begin{rem}\label{rem:uWPDE}
Hydrodynamic limits of particle systems are typically described as solutions of some fixed PDE. 
However, the differentiability properties of the function $u_W$ given in \eqref{PDEdtdW} suggest that the empirical configuration of particles is asymptotically described by the solution to a \emph{random} PDE of the form \eqref{hdlequation}; the randomness in the PDE comes from the Poisson point process $W$ and is analogous to the randomness of the environment $\w$ in which the random walkers are moving in. 
Note that since $u_W(0,\cdot) = u(\cdot)$ is continuous and $\s_W(\cdot)$ is discontinuous at points in $\mathcal{J}_W$ the right side of \eqref{PDEdtdW} equals 0 in the case $t=0$. We state \eqref{PDEdtdW} using a limit instead in this case to make the connection with the PDE \eqref{hdlequation} more clear.

We suspect that the differentiability properties in \eqref{PDEdtdW}, together with the initial condition $u_W(0,\cdot) \equiv u(\cdot)$ uniquely characterizes the function $u_W(t,x)$. However, instead of characterizing the function $u_W$ as a solution to a PDE of some sort we will define $u_W$ probabilistically in Section \ref{sec:dt} using what we will call a \emph{directed trap process}. The differentiability properties in \eqref{PDEdtdW} will then be proved in Section \ref{sec:uW}. 
Finally, in Appendix \ref{app:uWunique} we show that if we also assume that the function $u(x)$ is of bounded variation then the function $u_W(t,x)$ is in fact the unique solution to a random PDE which depends on the point process $W$.

\end{rem}

We give now a brief outline of how we will prove Theorem \ref{th:RWREhdl}. 
Several recent works \cite{psWQLTn,psWQLXn,estzWQL,dgWQL} have shown that the behavior of RWRE under Assumptions \ref{asmiid}--\ref{asmnl} are similar to the behavior of a much simpler process which we will call a directed trap process. 
We define the directed trap process in Section \ref{sec:dt} and state an analog of Theorem \ref{th:RWREhdl} for systems of independent particles in a common trap environment. 
Due to the simple nature of the directed trap processes, one would expect the proof of the hydrodynamic limit to follow easily. However, a difficulty arises in that the natural limiting directed trap process has traps that are spatially dense in $\R$. Thus, a bit of work is required to properly define the function $u_W$ and to prove the corresponding differentiability properties like in \eqref{PDEdtdW}.
After proving the differentiability properties of the function $u_W$ in Section \ref{sec:uW} we then prove the hydrodynamic limit for directed traps (Theorem \ref{th:dthdl}) in Section \ref{sec:dthdl}. 
In Section \ref{sec:trapst} we recall some of the tools that were introduced in \cite{psWQLTn,psWQLXn} for coupling a RWRE with an associated directed trap process. These techniques are then adapted in Section \ref{sec:couple1} to couple entire systems of independent RWRE in a common environment with an associated system of independent directed trap particles in a common trap environment. 
This comes close to proving the hydrodynamic limit as stated in Theorem \ref{th:RWREhdl},
 but the natural coupling of the RWRE system with a directed trap system leads to a system of RWRE that is different from those in Theorem \ref{th:RWREhdl}. In particular, the natural coupling requires both a change in the distribution $P$ on environments and initial configurations that are different from the locally stationary configurations $\mu_u^n(\w)$ defined in \eqref{rwlocstat}. 
Proving that these two changes to the system of RWRE do not affect the hydrodynamic limit requires significant technical effort and is accomplished in Section \ref{sec:couple2}. 

\subsection{Notation and technical details}

We close the introduction by introducing some notation and discussing some technical details that will be used throughout the remainder of the paper.
Throughout this subsection, we will use $\Psi$ to denote a generic Polish space. 

\subsubsection{Functions}
The set of functions $f:\Psi \ra \R$ that are continuous with compact support will be denoted by $\mathcal{C}_0(\Psi)$, and $\mathcal{C}_0^+(\Psi) \subset \mathcal{C}_0(\Psi)$ will denote the subset of such functions that are also non-negative. In the case that $\Psi = \R$ we will write $\mathcal{C}_0$ and $\mathcal{C}_0^+$ instead of $\mathcal{C}_0(\R)$ and $\mathcal{C}_0^+(\R)$, respectively. 

If $\Psi$ has a metric $d_{\Psi}(x,y)$ and $f:\Psi \ra \R$ is a real-valued function on $\Psi$, then we will use the notation 
\begin{equation}\label{Deltadef}
\Delta(f;\d) = \sup_{d_{\Psi}(x,y) \leq \d } |f(x) - f(y)| 
\end{equation}
for the modulus of continuity of $f$.
The supremum in the definition of $\Delta(f;\d)$ is of course restricted to $x,y$ in the domain of the function $f$. 
Note that we do not assume that the function $f$ is continuous in defining the modulus of continuity. However, recall that if $f$ is uniformly continuous (for instance, if $f \in \mathcal{C}_0$) then $\Delta(f;\delta) \ra 0$ as $\d\ra 0$.

\subsubsection{Measures and point processes}
The space of non-negative Radon measures on $\Psi$ will be denoted by $\mathcal{M}_+(\Psi)$ (recall that Radon measures are finite on compact subsets). We will equip $\mathcal{M}_+(\Psi)$ with the usual vague topology.
The reader is referred to \cite{rEVRVPP} for more details on the vague topology, but we simply recall here that a sequence of Radon measures $\mu_n$ converges to $\mu \in \mathcal{M}_+(\Psi)$ if and only if $\lim_{n\ra\infty} \langle f, \mu_n \rangle = \langle f, \mu \rangle$ for all $f \in \mathcal{C}_0(\Psi)$, where here and throughout the paper we will use the notation
\[
 \langle f, \mu \rangle = \int_\Psi f(x) \mu(dx)
\]
to denote integration of $f$ with respect to the measure $\mu$. 
Also, we note that the vague topology is compatible with a metric that makes $\mathcal{M}_+(\Psi)$ a Polish space. 

If $\Psi$ is also locally compact, we will use $\mathcal{M}_p(\Psi) \subset \mathcal{M}_+(\Psi)$ to denote the subset of $\mathcal{M}_+$ consisting of point processes on $\Psi$; that is purely atomic measures of the form $\sum_i \delta_{x_i}$, where $\delta_{x_i}$ denotes the Dirac-delta measure at the point $x_i \in \Psi$. 
Since point processes are Radon measures, there are only finitely many atoms in each compact subset of $\Psi$, and since we have assumed that $\Psi$ is a locally compact Polish space we can conclude that any point process only has countably many atoms. Therefore, we may (and will often) enumerate the atoms of the point process by the set of integers $\Z$.  
The only point processes that will appear in the remainder of the paper will be on the space $\Psi = \R \times (0,\infty]$. Thus, for the remainder of the paper we will simply write $\mathcal{M}_p$ in place of $\mathcal{M}_p(\R \times (0,\infty])$. Note that $[-L,L]\times [\e,\infty]$ is a compact subset of $\R\times (0,\infty]$ and thus any point process $M \in \mathcal{M}_p$ will have only finitely many atoms in $[-L,L] \times [\e,\infty]$. 

\subsubsection{Path spaces}
If $I \subset \R$ is a connected subset of $\R$, a real-valued function function $x:I \ra \R$ is called \emph{c\`adl\`ag} if it is right continuous with left limits.
We will use the notation $x(t-) = \lim_{s \ra t^-} x(s)$ for the value of the left limit at $t \in I$. 
The collection of all c\`adl\`ag paths $x:I \ra \Psi$ will be denoted $D_I$. 

There are many topologies that can be placed on the path space $D_I$ (see \cite{wSPL}), and in this paper we will at times use two different topologies: the Skorohod $J_1$-topology\footnote{There are several different Skorohod topologies, but the Skorohod $J_1$-topology is the most commonly used and in fact is often just called the Skorohod topology.} and the uniform topology. 
When we wish to indicate that a particular topology is being used we will use the notation $D_I^J$ or $D_I^U$ to denote $D_I$ equipped with the Skorohod $J_1$-topology or the uniform topology, respectively. 
We refer the reader to \cite{bCOPM} or \cite{wSPL} for the details of the Skorohod $J_1$-topology and the uniform topology, and will simply recall here some important properties of these topologies that we will use. 
\begin{itemize}
 \item The Skorohod $J_1$-topology is compatible with a metric that makes $D_I^J$ into a Polish space. 
 \item The uniform topology is stronger than the Skorohod $J_1$-topology. However, if a sequence of paths $x_n \ra x$ in the space $D_I^{J}$ and the limiting path $x$ is continuous, then it follows that $x_n \ra x$ in $D_I^U$ as well. 
\end{itemize}

\subsubsection{Asymptotic notation}
For sequences $\{f_n\}_n$ and $\{g_n\}_n$ of real numbers with $g_n>0$ for all $n$, we will use the following notation for comparing the asymptotics of these sequences as $n\ra\infty$. 
\begin{itemize}
 \item $f_n = \bigo(g_n)$ if $\limsup_{n\ra\infty} |f_n|/g_n < \infty$. 
 \item $f_n = o(g_n)$ if $\lim_{n\ra\infty} f_n/g_n = 0$. 
 \item $f_n \sim g_n$ if $\lim_{n\ra\infty} f_n/g_n = 1.$
\end{itemize}

\noindent\textbf{Acknowledgement.} The authors are grateful to an anonymous referee for a very thorough and careful reading of the paper, and for the helpful suggestion in simplyfing the proof of Lemma \ref{lem:uWprops}. 

\section{Directed traps and systems of independent directed traps}\label{sec:dt}

The proof of the hydrodynamic limit in Theorem \ref{th:RWREhdl} will be obtained by comparing the system of RWRE with a system of particles in a directed trap environment. In this subsection we will introduce the model of directed traps and state a hydrodynamic limit for systems of independent particles in directed traps. 

\subsection{Directed traps}
We begin by describing the model of directed traps. Recall that we are using the notation $\mathcal{M}_p$ for the space of point processes on $\R \times (0,\infty]$. 
A ``trap environment'' for a directed trap process is an element $W = \sum_k \d_{(x_k,y_k)} \in \mathcal{M}_p$. 
For $W$ fixed, we then wish to construct a directed trap process $Z_W$ with the following dynamics. When located at $x_k$ it stays there for an Exp($1/y_k$) amount of time before moving to the next trap to the right. 
Of course, a complication with this informal description arises if the set $\{x_k\}_k$ of spatial trap locations is dense in $\R$, for then there is no ``next'' trap to the right. 
To account for this, we now define the following subset of $\mathcal{M}_p$ for which we can make the above construction rigorous (including cases where the locations of the traps are dense in $\R$). 

\begin{defn}\label{def:T}
 The subset $\mathcal{T} \subset \mathcal{M}_p$ consists of all 
point processes $W$ with the following properties: 
\begin{enumerate}
 \item $W(\R \times \{\infty\}) = 0$.
 \item $\sup_x W(\{x\} \times (0,\infty]) = 1$. 
 \item There exists an $\e_0>0$ such that $W((-\infty,-L]\times [\e_0,\infty)) = W([L,\infty)\times [\e_0,\infty)) = \infty$ for all $L<\infty$. 
 \item For any fixed $\e>0$ and $L<\infty$, 
\[
 \iint\limits_{[-L,L]\times (0,\e)} y \, W(dx \, dy) = \sum_k y_k \ind{|x_k|\leq L, \, y_k < \e} < \infty. 
\]
\end{enumerate}
Point processes $W \in \mathcal{T}$ will be referred to as \textbf{trap environments}.
\end{defn}


\begin{rem}
We briefly explain the need for these three conditions in the definition of $\mathcal{T}$. The first condition states that there are no ``infinite traps'' and the second condition is that there is only one trap at each spatial location. 
The final two conditions relate to the time it will take the process to cross an interval. 
The third condition ensures that there are infinitely many traps larger than a fixed size in both spatial directions; thus the process will not be able to travel an infinite distance in a finite amount of time. 
Finally, since there can be infinitely many ``small traps'' in an interval $[-L,L]$, the fourth condition will ensure that these traps are small enough so that the total time spent by the process $Z_W$ in crossing the interval $[-L,L]$ is finite. 
\end{rem}

\begin{rem}
Because there are (countably) infinitely many atoms $(x_k,y_k)$ in a trap environment $W \in \mathcal{T}$, we can enumerate the trap environments by the index set $\Z$. However, it should not necessarily be assumed that the atoms are ordered with respect to this indicing. That is, we cannot necessarily assume that $k<\ell$ implies that $x_k < x_\ell$ (although this will be the case for some trap environments we will consider later). 
\end{rem}


\begin{defn}
 If $W \in \mathcal{T}$ is a trap environment, then we will define 
\[
 \mathcal{J}_W = \{x \in \R: \, W(\{x\}\times (0,\infty)) > 0\}.
\]
We will refer to $\mathcal{J}_W$ as the set of \textbf{trap locations} for $W$.  
\end{defn}
\begin{defn}
  If $W \in \mathcal{T}$ is a trap environment, then we will define the measure $\s_W \in \mathcal{M}_+(\R)$ by 
\begin{equation}\label{sWdef}
 \s_W(A) = \iint_{A \times (0,\infty)} y \, W(dx \, dy). 
\end{equation}
We will refer to $\s_W$ as the \textbf{trap measure} for $W$. 
\end{defn}
\begin{rem}
 The fact that $\s_W$ is a Radon measure for $W \in \mathcal{T}$ follows from the first and last properties in Definition \ref{def:T}. 
\end{rem}

We now show how to construct a directed trap process $Z_W$ for $W \in \mathcal{T}$. 
Let $\{\zeta_k\}_{k\in \Z}$ be an i.i.d.\ family of Exp(1) random variables, and let 
$\tau_W \in \mathcal{M}_+(\R)$ be the random measure on $\R$ given by 
\[
 \tau_W = \sum_{k \in \Z} y_k \zeta_k \d_{x_k}. 
\]
\begin{rem}
The construction of $\tau_W$ above does not necessarily imply that $\tau_W \in \mathcal{M}_+$, but we claim that for almost every instance of the exponential random variables $\{\zeta_k\}$ the construction does indeed give a Radon measure.  
 To see this, note that $\Ev[\sum_{k \in \Z} y_k \zeta_k \ind{|x_k|\leq L}] = \sum_k y_k \ind{|x_k|\leq L} = \s_W([-L,L]) < \infty$ for any $L<\infty$.
Therefore, it follows that with probability one $\tau_W([-L,L]) < \infty$ for all $L<\infty$. That is, $\tau_W$ is indeed a non-negative Radon measure on $\R$. 
Therefore, we can define $\tau_W$ as above on a set of full probability and arbitrarily set $\tau_W$ to be Lebesgue measure on the null set where the above construction does not give a Radon measure. 
\end{rem}

We are now ready to define the directed trap process $Z_W$. 
The process $Z_W$ will move strictly to the right with speed dictated by the measure $\tau_W$ in that $\tau_W(A)$ will be the total amount of time that the process $Z_W$ spends in the subset $A$. 
To make this precise, for any $x\in \R$ and $t\geq 0$ we will define
\begin{equation}\label{ZWdef}
 Z_W(t;x) = \sup\{ x': \, \tau_W([x,x')) \leq t \}. 
\end{equation}
If the set $\mathcal{J}_W$, which marks to spatial locations of the traps, is nowhere dense in $\R$ then it is clear that $\{Z_W(t;x)\}_{t\geq 0}$ is a Markov process on the set $\mathcal{J}_W$ which evolves in the manner described at the beginning of the section (i.e., waits at $x_k$ for an Exp($1/y_k$) amount of time). 
On the other hand, if $\mathcal{J}_W$ has limit points in $\R$ then it is more complicated to prove that the construction \eqref{ZWdef} gives a Markov process. 
However, it is easy to see that this construction implies the following fact which would correspond to an application of the strong Markov property.
\begin{equation}\label{ZWsm}
  Z_W(t;x) = Z_W(t + \tau_W([z,x)); z), \quad \forall z<x.
\end{equation}
Additionally, the following properties of the process $Z_W$ follow easily from the definition in \eqref{ZWdef}. 
\begin{itemize}
 \item For any fixed $x \in \R$, the process $t\mapsto Z_W(t;x)$ is non-decreasing and right continuous. 
 \item The process $Z_W(t;x)$ starts at the ``first trap'' greater than or equal to $x$. That is, 
\[
 Z_W(0;x) = \inf\{ y \geq x: \, y \in \mathcal{J}_W \}. 
\]
\item The construction in \eqref{ZWdef} gives a natural coupling so that $Z_W(t;x) \leq Z_W(t;y)$ whenever $x \leq y$. Moreover, for any fixed $t > 0$, $Z_W(t;x)$ is non-decreasing and left continuous in $x$. 
\end{itemize}

An important special class of trap environments will be those trap environments with trap locations that are dense in $\R$. 
\begin{defn}
 The subset $\mathcal{T}'$ is the collection of all trap environments $W \in \mathcal{T}$ with the property that
$\mathcal{J}_W$ is dense in $\R$. 
\end{defn}
\noindent
The following additional properties are true for $Z_W$ when $W \in \mathcal{T}'$. 
\begin{itemize}
 \item If $W \in \mathcal{T}'$, then $Z_W(0;x) = x$ for all $x\in \R$. 
 \item If $W \in \mathcal{T}'$, then $Z_W(t;x)$ is almost surely continuous in $t$ for any fixed $x$. This follows from the fact that $\tau_W([x,y))$ is almost surely strictly increasing in $y$ for any fixed $x\in \R$. 
\end{itemize}

In addition to the directed trap process $Z_W$ defined above, we will also need to define an analogous process that moves to the left instead. 
That is, 
\[
 Z_W^*(t;x) = \inf\{ x': \tau_W((x',x]) \leq t \}. 
\]
Of course similar properties that were stated above for $Z_W$ are also true for $Z_W^*$.
Similarly to \eqref{ZWsm}, we have the following strong Markov-like property. 
\begin{equation}\label{ZWssm}
  Z_W^*(t;x) = Z_W^*(t + \tau_W((x,z]); z), \quad \forall x < z.
\end{equation}
Other important properties which are slightly different from those for $Z_W$ are the following. 
\begin{itemize}
 \item For $x \in \R$ fixed, $Z_W^*(t;x)$ is non-increasing and right continuous in $t$. 
 \item $Z_W^*(0;x) = \sup\{y\leq x: y \in \mathcal{J}_W \}$. 
 \item For $t > 0$ fixed, $Z_W^*(t;x)$ is non-decreasing and right continuous in $x$. 
\end{itemize}
The importance of the left-directed trap process $Z_W^*$ is that we will use it to define the function $u_W$ that appears in the statement of Theorem \ref{th:RWREhdl} and also in the hydrodynamic limit for directed traps below. This function is defined by 
\begin{equation}\label{uWdef}
 u_W(t,x) = \Ev\left[ u\left(Z_W^*(t;x)\right) \right]. 
\end{equation}
Clearly the above properties of $Z_W^*$ show that $u_W(0,x) = u(x)$ whenever $W \in \mathcal{T}'$.
We will prove that $u_W(t,x)$ satisfies the differentiability properties of \eqref{PDEdtdW} if $W \in \mathcal{T}'$ in  Section \ref{sec:uW} below. 

\subsection{Hydrodynamic Limits for Directed Traps}
The main goal of the current paper is to study systems of independent random walks in a common random environment $\w$. 
In a similar manner, we will study systems of independent directed trap particles in a common trap environment $W$. 
We will use notation that is similar to that which was used to define the systems of independent RWRE. 

Let $W \in \mathcal{T}$ be a fixed trap environment, and fix an enumeration $\{(x_k,y_k)\}$ of the atoms of $W$. 
Given this enumeration of the traps let $\{ (Z^{k,j}_W(t) )_{t\geq 0} \}_{k \in \Z, j\geq 1}$ be an independent family of directed trap processes in the trap environment $W$ such that 
$Z^{k,j}_W(\cdot) \overset{\text{Law}}{=} Z_W(\cdot\,;x_k)$ for all $k \in \Z$, $j\geq 1$. 
 As with the systems of RWRE, we will only start finitely many particles at each $x_k$, and we will use $\eta_t^W$ to denote the configuration of particles at any time $t\geq 0$. 
That is, we will fix a (random) initial configuration $\{ \eta_0^W(x_k) \}_{k\in \Z}$ according to a distribution that we will specify later. 
Then the configuration $\eta_t^W$ at any later time $t>0$ will be given by 
\[
 \eta_t^W(x_k) = \sum_{\ell \in \Z} \sum_{j=1}^{\eta_0^W(x_\ell)} \ind{Z_W^{\ell,j}(t) = x_k }, \quad k\in\Z, \, t\geq 0.
\]

We will be interested in proving a hydrodynamic limit for systems of independent directed trap processes in a common trap environment. To obtain a limit, however, instead of re-scaling a fixed trap environment we will instead assume that we have a sequence of trap environments $W_n$ that converge to a trap environment $W$. 
Moreover, we will assume that the limiting trap environment $W$ is dense in $\R$. 

\begin{asm}\label{asmtrapc}
 $\{W_n\}_{n\geq 1} \subset \mathcal{T}$ is a sequence of point processes such that $W_n \ra W \in \mathcal{T}'$ as $n\ra\infty$, where the convergence is with respect to the vague topology on $\mathcal{M}_p(\R \times (0,\infty])$.  
\end{asm}
\begin{rem}
 We will denote the atoms of $W_n$ by $\{(x_k^n,y_k^n)\}_{k\in \Z}$ to distinguish them from the atoms $(x_k,y_k)$ of  $W$. That is, $W_n = \sum_{k\in\Z} \d_{(x_k^n,y_k^n)}$ and $W= \sum_{k \in \Z} \d_{(x_k,y_k)}$. 
\end{rem}

Since $W_n \in \mathcal{T}$ for every $n\geq 1$, the fourth condition in the definition of $\mathcal{T}$ implies that 
\[
 \lim_{\e\ra 0}  \iint\limits_{[-L,L]\times (0,\e)} y \, W_n(dx \, dy) = 0, \quad \forall L<\infty, \, n\geq 1. 
\]
However, we will need to assume the following uniform control on this convergence. 

\begin{asm}\label{asmtraps}
 The sequence $\{W_n\}_{n\geq 1}$ of trap environments is such that 
\[
 \lim_{\e\ra 0} \limsup_{n\ra\infty}  \iint\limits_{[-L,L]\times (0,\e)} y \, W_n(dx \, dy) = 0, \quad \forall L<\infty. 
\]
\end{asm}

%

For a fixed sequence of trap environments $W_n \ra W$, we will use the notation $\eta_t^n$ to denote the system of particles in the trap environment $W_n$ rather than the more cumbersome $\eta_t^{W_n}$. 
To obtain a hydrodynamic limit for the systems of independent directed trap particles, we will need to make the following assumptions on the sequence of initial configurations. 
\begin{asm}\label{asmtrapic}
There exists a sequence $a_n \ra \infty$ and a function $u \in \mathcal{C}_0^+$ such that for every $n$, the initial configuration 
of the system $\eta^n$ in the trap environment $W_n = \sum_k \d_{(x_k^n,y_k^n)}$ is given by
\begin{equation}\label{trapic}
 \{ \eta_0^n(x_k^n) \}_{k\in \Z} \text{ is product Poisson with }\quad 
\eta_0^n(x_k^n) \sim \textnormal{Poisson}(a_n u(x_k^n) y_k^n). 
\end{equation}
\end{asm}
\begin{rem}
 It can be seen that for any trap environment $W = \sum_k \d_{(x_k,y_k)}$ and any $\a>0$ the configuration that is product Poisson with $\eta(x_k) \sim \text{Poisson}(\a y_k)$ is a stationary configuration for the system of particles in the trap environment $W$. Thus, the initial configurations given in \eqref{trapic} are analogous to the locally stationary initial configurations $\mu_u^n$ for the systems of independent RWRE that were described above. 
\end{rem}
\begin{rem}
 The sequence $a_n$ is needed so that the number of particles in a fixed interval is unbounded as $n \ra\infty$ and allows for a law of large numbers type result to be used. 
When we use Theorem \ref{th:dthdl} to prove results for systems of RWRE the sequence $a_n$ will be $n^{1/\k}$. 
\end{rem}

We are now ready to state our main result for the hydrodynamics of systems of independent directed trap particles. 
\begin{thm}\label{th:dthdl}
Suppose that there are trap environments $W_n \ra W \in \mathcal{T}'$ as in Assumptions \ref{asmtrapc} and \ref{asmtraps}. 
Let the systems of particles $\eta^n$ in the trap environments $W_n$ be constructed on a common probability space $\Pv$ so that the initial conditions given in Assumption \ref{asmtrapic} are satisfied for some $u \in \mathcal{C}_0^+$ and some sequence $a_n \ra \infty$. 
Then, for any $\phi \in \mathcal{C}_0(\R_+ \times \R)$ we have that 
\begin{equation}\label{dthdllim}
 \lim_{n\ra\infty} \frac{1}{a_n} \int \sum_{k\in\Z} \eta_t^n(x_k^n) \phi(t,x_k^n) \, dt 
= \int\int_\R \phi(t,x) u_W(t,x) \, \s_W(dx) \, dt, \quad \text{ in $\Pv$-probability}, 
\end{equation}
where $u_W(t,x)$ is the function defined in \eqref{uWdef}.
\end{thm}
\begin{rem}
 The inner integral on the right in \eqref{dthdllim} is a Lebesgue integral.
That is, using the representation $W=\sum_k \d_{(x_k,y_k)}$ we have
\[
 \int_\R \phi(t,x) u_W(t,x) \, \s_W(dx) = \sum_{k \in \Z} \phi(t,x_k) u_W(t,x_k) y_k. 
\]
\end{rem}

\section{The function \texorpdfstring{$u_W(t,x)$}{}} \label{sec:uW}

The main goal of this section will be to prove differentiability properties of the function $u_W(t,x)$ that 
arises in the statement of the hydrodynamic limits in Theorems \ref{th:RWREhdl} and \ref{th:dthdl}.   
Recall the definition of $u_W(t,x)$ in \eqref{uWdef} for any trap environment $W \in \mathcal{T}$. 
The main goal of this section is to show some differentiability properties of $u_W$, but we begin with a few easy continuity properties. 
\begin{lem}\label{lem:uWcont}
 For any $W \in \mathcal{T}$, the function $u_W(t,x)$ is right continuous with left limits in $x$ for every $t > 0$ and continuous in $t$ for every $x \in \R$. 
\end{lem}
\begin{proof}
As noted above, it follows from the construction of the process $Z_W^*$ that 
\begin{itemize}
 \item $t\mapsto Z_W^*(t;x)$ is right continuous with left limits for any fixed $x\in \R$, and
 \item $x\mapsto Z_W^*(t;x)$ is right continuous with left limits for any fixed $t > 0$. 
\end{itemize}
From these two facts and the bounded convergence theorem it follows that $u_W(t,x)$ is right continuous with left limits in $t$ for any fixed $x$ and also right continuous with left limits in $x$ for any fixed $t$. 

To show that $u_W(t,x)$ is also left continuous in $t$, first note that \eqref{ZWssm} implies that 
\begin{equation}\label{uWrightx}
\begin{split}
 u_W(t,x+h) &= \Ev\left[ u(Z_W^*(t;x+h)) \ind{\tau_W((x,x+h]) > t} \right] \\
&\qquad + \Ev\left[ u_W(t-\tau_W((x,x+h]), x) \ind{\tau_W((x,x+h]) \leq t} \right].
\end{split}
\end{equation}
Since $u$ (and thus also $u_W$) is bounded and $\tau_W((x,x+h]) \ra 0$ as $h\ra 0^+$, we can thus conclude from the Bounded Convergence Theorem that 
\[
 u_W(t,x) = \lim_{h\ra 0^+} u_W(t,x+h) = u_W(t-,x), 
\]
where the first equality is from the right continuity in $x$ that was proved above. This completes the proof that $u_W(t,x)$ is continuous in $t$. 
\end{proof}
\begin{rem}
 If the trap environment $W \in \mathcal{T}'$, then the fact that $u_W(t,x)$ is continuous in $t$ follows more easily from the fact noted above that $Z_W^*(t;x)$ is continuous in $t$ for any fixed $x\in \R$ when $W \in \mathcal{T}'$. 
\end{rem}

Before proving the differentiability properties of $u_W$, we need the following Lemma which gives a probabilistic formulation for $\lim_{h\ra 0^+} u_W(t,x-h)$. 
\begin{lem}\label{lem:uWc}
 Let $Z_W^\circ(t;x) = Z_W^*(\tau_W(\{x\}) + t;x)$. 
Informally, $t\mapsto Z_W^\circ(t;x)$ is the path of the (left) directed trap process started just after leaving site $x$. 
Then, 
\begin{equation}\label{uWcdef}
 u_W^\circ(t,x) := \Ev[ u(Z_W^\circ(t;x))] =  \lim_{h \ra 0^+} u_W(t,x-h), \quad \forall t\geq 0, \, x \in \R. 
\end{equation}
Moreover, the function $u_W^\circ(t,x)$ is continuous in $t$ for any fixed $x \in \R$. 
\end{lem}
\begin{rem}\label{rem:uWuWc}
 Note that if $x \notin \mathcal{J}_W$, then $\tau_W(\{x\}) = 0$ and so 
 $Z_W^\circ(t;x) = Z_W^*(t;x)$ and $u_W^\circ(t,x) = u_W(t,x)$ for all $t\geq 0$. 
\end{rem}

\begin{proof}
The Markov-like property in \eqref{ZWssm} implies that 
\[
 u_W(t,x-h) = \Ev\left[u\left(Z_W^*(t+\tau_W((x-h,x]);x)\right)\right]. 
\]
The limit in \eqref{uWcdef} then follows from the Bounded Convergence Theorem along with the fact that $\tau_W((x-h,x]) \ra \tau_W(\{x\})$ almost surely as $h \ra 0^+$ and $Z_W^*(\cdot;x)$ is right continuous.  
If $x \notin \mathcal{J}_W$, then it follows from Lemma \ref{lem:uWcont} that $u_W^\circ(t,x)$ is continuous in $t$ since as noted in Remark \ref{rem:uWuWc} above $u_W^\circ(t,x) = u_W(t,x)$ for all $t\geq 0$ in this case.
It remains to show that $u_W^\circ(t,x_k)$ is continuous in $t$ for any fixed $x_k \in \mathcal{J}_W$. 
However, for $x_k \in \mathcal{J}_W$ if we define the trap environment $\widetilde{W}_k = W - \d_{(x_k,y_k)}$ (that is remove the trap at spatial location $x_k$ from the trap environment $W$) then it is easy to see that $Z_W^\circ(t,x_k) = Z_{\widetilde{W}_k}^*(t,x_k)$. Therefore, $u_W^\circ(t,x_k) = u_{\widetilde{W}_k}(t,x_k)$ and it follows from Lemma \ref{lem:uWcont} that $u_W^\circ(t,x_k)$ is continuous in $t$. 
\end{proof}

We are now ready to prove some differentiability properties of $u_W$. 

\begin{prop}\label{prop:duWdt}
For any $W \in \mathcal{T}$, the function $u_W(t,x_k)$ is differentiable in $t$ for any $x_k \in \mathcal{J}_W$ and 
\begin{equation}\label{duWdteq1}
\frac{\del}{\del t} u_W(t,x_k) = - \lim_{h \ra 0^+} \frac{u_W(t,x_k) - u_W(t,x_k-h)}{\s_W((x_k-h,x_k])}, \quad \forall x_k \in \mathcal{J}_W, \, t \geq 0. 
\end{equation}
Moreover, if $W \in \mathcal{T}'$ then we also have
\begin{equation}\label{duWdteq2}
 \frac{\del}{\del t} u_W(t,x_k) = - \lim_{h\ra 0^+}  \frac{u_W(t,x_k+h) - u_W(t,x_k)}{\s_W((x_k,x_k+h])}, \quad \forall x_k \in \mathcal{J}_W, \, t > 0. 
\end{equation}

\end{prop}
\begin{proof}
First of all, by conditioning on the value of $\tau_W(\{x_k\}) \sim \text{Exp}(1/y_k)$ and recalling the definition of $u_W^\circ$ in \eqref{uWcdef} we have that 
\begin{align}
 u_W(t,x_k) 
&= e^{-t/y_k} u(x_k) + \int_0^t \frac{1}{y_k} e^{-s/y_k} u_W^\circ(t-s,x_k) \, ds \nonumber \\
&= e^{-t/y_k} u(x_k) + e^{-t/y_k} \int_0^t \frac{1}{y_k} e^{s/y_k} u_W^\circ(s,x_k) \, ds.  \label{uWibp}
\end{align}
From this representation of $u_W(t,x_k)$, and using the fact that $u_W^\circ(s,x_k)$ is continuous in $s$, it is easy to conclude that $u_W(t,x)$ is differentiable in $t$ with derivative given by 
\begin{equation}\label{duWdteq3}
 \frac{\del}{\del t} u_W(t,x_k) = \frac{- u_W(t,x_k) + u_W^\circ(t,x_k)}{y_k}. 
\end{equation}
The equality in \eqref{duWdteq1} then follows from Lemma \ref{lem:uWc} and the fact that $\s_W((x_k-h,x_k]) \ra y_k$ as $h\ra 0^+$. 

Before giving the proof of \eqref{duWdteq2} when $W \in \mathcal{T}'$, we introduce some notation that will be convenient. For $W$ and $x_k \in \mathcal{J}_W$ fixed, let 
\[
 \tau_h = \tau_W((x_k,x_k+h]), \quad \text{and} \quad \s_h = \s_W((x_k,x_k+h]). 
\]
Since $\tau_h = \sum_{\ell: x_\ell \in (x_k,x_k+h]} y_\ell \zeta_\ell$, with $\zeta_\ell$ i.i.d.\ Exp(1) random variables, it is easy to see that 
\begin{equation}\label{tauhmoments}
 \Ev[\tau_h] = \s_h \quad\text{and}\quad 
\Varv(\tau_h) \leq \s_h^2. 
\end{equation}
Having introduced this notation, we now turn to evaluating the limit on the right side of \eqref{duWdteq2}. 
To this end, we first note that \eqref{uWrightx} implies  
\begin{equation}\label{uWdiscrete}
\begin{split}
\frac{u_W(t,x_k+h) - u_W(t,x_k)}{\s_h} 
& = \frac{\Ev\left[ u(Z_W^*(t;x_k+h)) \ind{\tau_{h} > t} \right]}{\s_h} - \frac{\Pv( \tau_{h} > t )u_W(t,x_k)}{\s_h}  \\
&\qquad - \Ev\left[ \frac{u_W(t,x_k) - u_W(t-\tau_{h},x_k)}{\s_h}  \ind{\tau_h \leq t} \right]. 
\end{split}
\end{equation}
Since $u$ and $u_W$ are uniformly bounded above, to show the first two terms on the right in \eqref{uWdiscrete} vanish as $h\ra 0^+$ it is sufficient to note that 
\[
 \limsup_{h\ra 0^+} \frac{\Pv( \tau_{h} > t )}{ \s_h } \leq \limsup_{h\ra 0^+} \frac{\Pv( \left| \tau_{h} - \s_h \right| > t - \s_h )}{ \s_h } 
\leq \limsup_{h\ra 0^+} \frac{ \Varv( \tau_h )}{\s_h (t-\s_h)^2} 
= 0, 
\]
where in the last equality we used that $\Varv( \tau_h ) \leq \s_h^2$ by \eqref{tauhmoments} and that $\s_h \ra 0$ as $h\ra 0^+$. 
It remains to show that the last term on the right in \eqref{uWdiscrete} tends to $- \frac{\del}{\del t} u_W(t,x_k)$ as $h\ra 0^+$. 
To this end, first note that \eqref{duWdteq3} and the mean value theorem imply that 
\[
 \left| \frac{u_W(t,x_k) - u_W(t-a,x_k)}{a} \right| \leq \frac{2 \|u\|_\infty}{y_k}, \quad \text{for all } a \in [0,t). 
\]
Since $\Ev[\tau_h] = \s_h$, we can conclude that 
\begin{align*}
 & \left| \Ev\left[ \frac{u_W(t,x_k) - u_W(t-\tau_{h},x_k)}{\s_h} \ind{\tau_h \leq t} \right] - \frac{\del}{\del t} u_W(t,x_k) \right| \\
&= \left| \Ev\left[ \frac{u_W(t,x_k) - u_W(t-\tau_{h},x_k)}{\tau_h} \frac{\tau_h}{\s_h} \ind{\tau_h \leq t} \right] - \frac{\del}{\del t} u_W(t,x_k) \Ev\left[ \frac{\tau_h}{\s_h} \right] \right| \\
&\leq \Ev\left[ \left|  \frac{u_W(t,x_k) - u_W(t-\tau_{h},x_k)}{\tau_h} - \frac{\del}{\del t} u_W(t,x_k) \right| \frac{\tau_h}{\s_h} \ind{\tau_h \leq \s_h^{1/2}} \right] 
+ \frac{4 \|u\|_\infty}{y_k} \frac{ \Ev\left[\tau_h \ind{\tau_h > \s_h^{1/2}} \right]}{\s_h} \\
&\leq \sup_{0<\e \leq \s_h^{1/2}} \left|  \frac{u_W(t,x_k) - u_W(t-\e,x_k)}{\e} - \frac{\del}{\del t} u_W(t,x_k) \right| 
+ \frac{4 \|u\|_\infty \s_h^{1/2}}{y_k },
\end{align*}
where in the last inequality we used that $ \Ev[\tau_h \ind{\tau_h > \s_h^{1/2}} ] \leq \frac{ \Ev[ \tau_h^2 ] }{\s_h^{1/2}} \leq 2 \s_h^{3/2}$.
Since $\s_h \ra 0$ as $h\ra 0^+$, both terms in the last line above vanish as $h\ra 0^+$. 
Recalling \eqref{uWdiscrete}, this completes the proof of \eqref{duWdteq2}. 
\end{proof}

\section{Proof of the hydrodynamic limit for directed traps}\label{sec:dthdl}

In this section we will give the proof of the hydrodynamic limit for the systems of independent directed trap particles as stated in Theorem \ref{th:dthdl}. 
As a first step toward the proof of the theorem, we prove the following lemma which helps explain the appearance of the function $u_W(t,x)$. 

\begin{lem}\label{lem:dual}
 If $W = \sum_k \d_{(x_k,y_k)} \in \mathcal{T}$ is a trap environment and $\eta_t^W$ is a system of independent trap particles with initial configuration that is product Poisson with $\Ev[\eta_0^W(x_k) ] = u(x_k) y_k$, then for any $t> 0$ the configuration $\{\eta_t^W(x_k)\}_k$ is also product Poisson but with 
\[
 \Ev[\eta_t^W(x_k)] = u_W(t,x_k) y_k. 
\]
\end{lem}
\begin{rem}\label{rem:dual}
 Note that under the assumptions of Theorem \ref{th:dthdl} we have that 
\[
 \Ev[\eta_t^n(x_k^n)] = a_n u_{W_n}(t,x_k^n) y_k^n.
\]
In this application we have replaced $W$ by $W_n$ and $u(x)$ by $a_n u(x)$. 
\end{rem}
\begin{proof}
 First of all, since the initial configuration is product Poisson and the particles all move independently then it is a simple consequence of the thinning and superposition properties of Poisson random variables that the configuration at any fixed later time is also product Poisson, so it only remains to prove the formula for $\Ev\left[ \eta_t^W(x_k) \right]$. 

To this end, we first claim that 
the following time reversal property holds. 
\begin{equation}\label{duality}
 y_\ell \Pv(Z_W(t;x_\ell) = x_k ) = y_k \Pv(Z_W^*(t;x_k) = x_\ell ), \quad \text{ if } x_\ell \leq x_k. 
\end{equation}
Indeed, it is easy to see from the construction of the processes $Z_W$ and $Z_W^*$ that for any $s \in [0,t]$, 
\begin{align*}
 \Pv\left(Z_W(t;x_\ell) = x_k \bigl| \, \tau_W((x_\ell,x_k)) = s \right)
&= \int_0^{t-s} \frac{1}{y_\ell} e^{-u/y_\ell} e^{-(t-s-u)/y_k} \, du,
\end{align*}
and
\begin{align*}
\Pv\left(Z_W^*(t;x_k) = x_\ell \bigl| \, \tau_W((x_\ell,x_k)) = s \right)
&= \int_0^{t-s} \frac{1}{y_k} e^{-u/y_k} e^{-(t-s-u)/y_\ell} \, du, \\
&= \int_0^{t-s} \frac{1}{y_k} e^{-(t-s-v)/y_k} e^{-v/y_\ell} \, dv, 
\end{align*}
where the last equality follows from the substitution $v=t-s-u$. 
Comparing these formulas and then averaging over all possible values of $\tau_W((x_\ell,x_k))$ we obtain \eqref{duality}.

Next, since $\eta_t^W(x_k) = \sum_{\ell: x_\ell \leq x_k} \sum_{j=1}^{\eta_0^W(x_\ell)} \ind{ Z_W^j(t;x_\ell) = x_k}$, by conditioning on the initial configuration it is easy to see that 
\[
\Ev[\eta_t^W(x_k)] = \sum_{\ell: x_\ell \leq x_k} \Ev[\eta_0^W(x_\ell)] \Pv(Z_W(t;x_\ell) = x_k) = \sum_{\ell: x_\ell \leq x_k} u(x_\ell) y_\ell \Pv(Z_W(t,x_\ell) = x_k)
\]
Applying \eqref{duality} we then obtain that 
\[
 \Ev[\eta_t^W(x_k)] = \sum_{\ell: x_\ell \leq x_k} u(x_\ell) y_k \Pv(Z_W^*(t;x_k) = x_\ell) = y_k \Ev[ u(Z_W^*(t;x_k)) ],
\]
which, recalling the definition of the function $u_W$, is the claimed formula for the mean. 
\end{proof}

Next we introduce some notation that we will use in the proof of Theorem \ref{th:dthdl}. 
First of all, let  
\[
 \pi_t^n = \frac{1}{a_n} \sum_{k \in \Z} \eta_t^n(x_k^n) \d_{x_k^n}. 
\]
denote the re-scaled empirical measure of the particle configuration $\eta_t^n$. 
With this notation, the limit \eqref{dthdllim} in the statement of Theorem \ref{th:dthdl} becomes
\begin{equation}\label{dthdllim2}
 \lim_{n\ra\infty} \int_{\R_+} \langle \phi_t, \pi_t^n \rangle \, dt = \int_{\R^+} \int_\R \phi(t,x) u_W(t,x) \s_W(dx) \, dt,  \quad \text{in $\Pv$-probability,}
\end{equation}
where on the left we use $\phi_t$ for the function $\phi(t,\cdot)$ for any $t\geq 0$ fixed, and where $\langle f, \mu \rangle$ denotes integration of a function $f$ with respect to a measure $\mu$. 
One of the complicating factors in proving the limit in \eqref{dthdllim2} is that since the limiting trap environment $W$ is dense in $\R$ there are infinitely many traps in every interval in the limit. 
We will get around this difficulty by a standard truncation of the trap environment and then by controlling the error induced by this truncation. 
To this end, for $\e>0$ let
\[
 W_n^{(\e)} = \sum_{k \in \Z} \d_{(x_k^n,y_k^n)} \ind{y_k^n \geq \e}
\quad\text{and}\quad
 W^{(\e)} = \sum_{k \in \Z} \d_{(x_k,y_k)} \ind{y_k \geq \e}. 
\]
We can then expand the measure $\Pv$ to also include systems of directed trap particles $\eta_t^{n,\e}$ in the truncated environments $W_n^{(\e)}$. 
The initial configurations for the systems in the truncated environments will again be product Poisson and with $\eta_0^{n,\e}(x_k^n) \sim \text{Poisson}(a_n u(x_k^n) y_k^n)$ whenever the atom $(x_k^n,y_k^n)$ of $W_n$ has $y_k^n \geq \e$, and the corresponding empirical measure of the particle system will be denoted $\pi_t^{n,\e}$. 
We will use the truncated trap environments to prove Theorem \ref{th:dthdl} by first proving a hydrodynamic limit for the systems in the truncated trap environments. That is, we will show that 
\begin{equation}\label{tsystemconv}
 \lim_{n\ra\infty} \int_{\R_+} \langle \phi_t ,\, \pi_t^{n,\e} \rangle \, dt = \int_{\R_+} \int_\R \phi(t,x) u_{W^{(\e)}}(t,x) \s_{W^{(\e)}}(dx)   \, dt, \quad\text{in $\Pv$-probability}, 
\end{equation} 
for arbitrarily small $\e>0$. 
In order to conclude that the corresponding limit \eqref{dthdllim2} holds for the original sequence of trap environments we then need to show that the errors introduced by the truncation of the trap environments are small. That is, we will need to show that
\begin{equation}\label{WeWconv}
\lim_{\e\ra 0} \sup_{t\leq T} \left|  \int_\R \phi(t,x) u_{W^{(\e)}}(t,x) \s_{W^{(\e)}}(dx) - \int_\R \phi(t,x) u_{W}(t,x) \s_{W}(dx) \right| = 0, \quad \forall T<\infty,
\end{equation}
and that
\begin{equation}\label{truncerror}
 \lim_{\e\ra 0} \limsup_{n\ra\infty} \Pv\left( \sup_{t\leq T} | \langle \phi_t , \pi_t^n \rangle - \langle \phi_t, \pi_t^{n,\e}\rangle | \geq \d \right) = 0, \quad \forall \d>0. 
\end{equation}

In the proofs of \eqref{tsystemconv}--\eqref{truncerror}, we will use that both $u$ and $\phi$ have compact support. Thus, for the remainder of this section we will fix constants $L,T<\infty$ such that $\supp u \subset [-L,L]$ and $\supp \phi \subset [-L,L] \times [0,T]$ (clearly the limits in \eqref{WeWconv} and \eqref{truncerror} need only be proved for this choice of $T$). 
Moreover, since $W$ has only countably infinitely many atoms we may choose the constant $L$ so that $W(\{-L,L\} \times (0,\infty)) = 0$ (that is, there are no traps located at $\pm L$ in $W$). 

\subsection{Convergence for the truncated systems}
In this subsection we will prove \eqref{tsystemconv}.
The proof of \eqref{tsystemconv} will follow from the following two lemmas. 
\begin{lem}\label{uWnuW}
For almost every $\e>0$, 
 \[
\lim_{n\ra\infty} \Ev\left[ \int_{\R_+} \langle \phi_t, \pi_t^{n,\e} \rangle \, dt \right] 
= \int_{\R_+} \int_\R \phi(t,x) u_{W^{(\e)}}(t,x) \s_{W^{(\e)}}(dx) \, dt. 
\]
\end{lem}

\begin{lem}\label{varlem}
For almost every $\e>0$, 
\[
 \lim_{n\ra\infty} \Varv\left( \int_{\R_+} \langle \phi_t, \pi_t^{n,\e} \rangle \, dt  \right) = 0. 
\]
\end{lem}

\begin{proof}[Proof of Lemma \ref{uWnuW}]
First of all, let $\e>0$ be such that  $W(\R \times \{\e\} ) = 0$ (since $W$ has countably many atoms, this is true for all but countably many $\e>0$). 
Therefore, by our above choice of $L$, the point process $W$ does not have any atoms on the boundary of $[-L,L]\times [\e,\infty]$ and so the vague convergence of $W_n$ to $W$ implies that for all $n$ large enough the point process $W_n$ has the same number of atoms in $[-L,L] \times [\e,\infty]$ as $W$ does and that the locations of these atoms converge to the locations of the respective atoms in $W$ (see \cite[Proposition 3.13]{rEVRVPP}). That is, letting $N_{\e,L} = W([-L,L]\times [\e,\infty])$ we can enumerate the atoms of $W_n$ and $W$ so that 
\begin{equation}\label{WneWerep}
 W_n^{(\e)} = \sum_{k=1}^{N_{\e,L}} \d_{(x_k^n,y_k^n)}, \quad  W^{(\e)} = \sum_{k=1}^{N_{\e,L}} \d_{(x_k,y_k)},
\quad \text{and}\quad \lim_{n\ra\infty} (x_k^n,y_k^n) = (x_k,y_k), \quad \forall k\leq N_{\e,L}. 
\end{equation}
Moreover, we can choose the enumeration so that the traps are ordered spatially. That is
\[
 1\leq k < \ell \leq N_{\e,L} \quad \Ra \quad x_k^n < x_\ell^n \text{ and } x_k < x_\ell. 
\]

Now, Lemma \ref{lem:dual} (see also Remark \ref{rem:dual}) implies that 
\begin{align*}
  \Ev\left[ \int_{\R_+} \langle \phi_t, \pi_t^{n,\e} \rangle \, dt \right] 
= \int_{\R_+} \Ev\left[ \langle \phi_t, \pi_t^{n,\e} \rangle \right] \, dt 
&= \int_{\R_+} \frac{1}{a_n} \sum_{k:\, y_k^n \geq \e} \phi(t,x_k^n) \Ev[ \eta_t^{n,\e}(x_k^n) ] \, dt \\
&= \int_0^T \sum_{k=1}^{N_{\e,L}} \phi(t,x_k^n) u_{W_n^{(\e)}}(t,x_k^n) y_k^n \, dt.
\end{align*}
Also, note that 
\[
 \int_{\R_+} \int_\R \phi(t,x) u_{W^{(\e)}}(t,x) \, \s_{W^{(\e)}}(dx) \, dt
= \int_0^T \sum_{k=1}^{N_{\e,L}} \phi(t,x_k) u_{W^{(\e)}}(t,x_k) y_k \, dt.
\]
Since $\phi$ is continuous, $x_k^n\ra x_k$, and $y_k^n \ra y_k$, we need only to show that 
\begin{equation}\label{uWnlim}
 \lim_{n\ra\infty} \sup_{t\leq T} \left|u_{W_n^{(\e)}}(t,x_k^n) - u_{W^{(\e)}}(t,x_k)\right| = 0, \quad\forall k\leq N_{\e,L}.  
\end{equation}
For convenience of notation, let $v_k^n(t) = u_{W_n^{(\e)}}(t,x_k^n)$ and $v_k(t) = u_{W^{(\e)}}(t,x_k)$. 
Then, it follows from Proposition \ref{prop:duWdt} that the families of functions $\{v_k^n\}_{k\leq N_{\e,L}}$ and  $\{v_k\}_{k\leq N_{\e,L}}$ are the solutions to the following systems of linear differential equations. 
\[
 \begin{cases}
  \frac{\del}{\del t} v_k^n(t) = \frac{v_{k-1}^n(t) - v_k^n(t)}{y_k^n} & k\leq N_{\e,L} \\
  v_k^n(0) = u(x_k^n) & k\leq N_{\e,L},
 \end{cases}
\quad\text{and}\quad
 \begin{cases}
  \frac{\del}{\del t} v_k(t) = \frac{v_{k-1}(t) - v_k(t)}{y_k} & k\leq N_{\e,L} \\
  v_k(0) = u(x_k) & k\leq N_{\e,L},
 \end{cases}
\]
where for convenience of notation we let $v_0^n(t) \equiv v_0(t) \equiv 0$. Note that here we used the fact that we have ordered the indices so that the traps are in increasing order and that the traps in the truncated environment are all separated so that $u_W^\circ(t,x_k^n) = u_W(t,x_{k-1}^n)$.  
Since $(x_k^n,y_k^n) \ra (x_k,y_k)$ and the function $u$ is continuous, the coefficients and initial conditions of the system
on the left converge to those for the system on the right. 
Thus, $v_k^n$ converges uniformly to $v_k$ as $n\ra\infty$ for each $k\leq N_{\e,L}$, which is exactly what was to be shown in \eqref{uWnlim}. 
\end{proof}

\begin{proof}[Proof of Lemma \ref{varlem}]
 It follows from Lemma \ref{lem:dual} that 
\begin{align*}
\Varv\left(  \langle \phi_t , \pi_t^{n,\e} \rangle \right)
&= \Varv\left( \frac{1}{a_n} \sum_{k:\, y_k^n \geq \e} \eta_t^{n,\e} (x_k^n) \phi(t,x_k^n) \right) \\
&= \frac{1}{a_n^2}  \sum_{k:\, y_k^n \geq \e} \Varv( \eta_t^{n,\e} (x_k^n) ) \phi(t,x_k^n)^2 \\
&= \frac{1}{a_n}  \sum_{k:\, y_k^n \geq \e} u_{W_n^{(\e)}}(t,x_k^n) y_k^n \phi(t,x_k^n)^2 
\leq \frac{\|u\|_\infty \|\phi\|_\infty^2}{a_n} \s_{W_n^{(\e)}}([-L,L]). 
\end{align*}
Note that the last term on the right does not depend on $t$ and vanishes as $n\ra\infty$. 
Since
\begin{align*}
 \Varv\left(  \int_{\R_+} \langle \phi_t , \pi_t^{n,\e} \rangle \, dt \right) 
&= \Varv\left(  \int_0^T \langle \phi_t , \pi_t^{n,\e} \rangle \, dt \right) \\
&= \int_0^T \int_0^T \Covv\left( \langle \phi_s , \pi_s^{n,\e} \rangle, \, \langle \phi_t , \pi_t^{n,\e} \rangle \right) \, ds \, dt \\
&\leq T^2 \left\{ \sup_{t\leq T} \Varv\left(  \langle \phi_t , \pi_t^{n,\e} \rangle \right) \right\},
\end{align*}
this completes the proof of the lemma. 
\end{proof}

\subsection{The error from truncating the limiting process}
Here we will prove \eqref{WeWconv}.
To this end, first note that 
\begin{align*}
&\left|  \int_\R \phi(t,x) u_{W^{(\e)}}(t,x) \s_{W^{(\e)}}(dx) - \int_\R \phi(t,x) u_{W}(t,x) \s_{W}(dx) \right|\\
&\quad= \left| \sum_k y_k \phi(t,x_k) \left\{ u_{W^{(\e)}}(t,x_k)  \ind{y_k \geq \e} - u_W(t,x_k) \right\} \right| \\
&\quad\leq \sum_k y_k |\phi(t,x_k)| \left|  u_{W^{(\e)}}(t,x_k)   - u_W(t,x_k) \right| \ind{y_k \geq \e} + \| u \|_\infty \sum_k y_k |\phi(t,x_k)| \ind{y_k < \e} \\
&\quad\leq \|\phi\|_\infty \sum_{k: |x_k|\leq L} y_k \left|  u_{W^{(\e)}}(t,x_k)   - u_W(t,x_k) \right| \ind{y_k \geq \e} + \| u \|_\infty \|\phi\|_\infty \sum_{k: |x_k|\leq L} y_k \ind{y_k < \e}
\end{align*}
Note that the last term on the right doesn't depend on $t$ and vanishes as $\e \ra 0$. Therefore, we need only to show that 
\[
 \lim_{\e \ra 0} \sup_{t} \sup_{k: |x_k|\leq L} \left|  u_{W^{(\e)}}(t,x_k)   - u_W(t,x_k) \right|  = 0. 
\]
To compare $u_W^{(\e)}(t,x_k)$ and $u_W(t,x_k)$, we will couple 
the left-directed trap processes $Z_{W^{(\e)}}^*(t;x_k)$ and $Z_W^*(t;x_k)$ by using the same exponential random variables $\zeta_k$ to generate the holding times at the traps with $y_k \geq \e$. 
Using this coupling and the fact that $\supp u \subset [-L,L]$ we have (recalling the notation $\Delta$ for the modulus of continuity from \eqref{Deltadef}) that 
\begin{align}
 \left| u\left(Z_{W^{(\e)}}^*(t;x_k)\right) - u\left( Z_W^*(t;x_k) \right) \right| 
&\leq \Delta\left( u ; \sup_{t\leq \tau_{W^{(\e)}}((-L,x_k])} \left| Z_{W^{(\e)}}^*(t;x_k) - Z_W^*(t;x_k) \right| \right) \nonumber \\
&\leq \Delta\left( u ; \sup_{t\leq \tau_{W^{(\e)}}((-L,L])} \left| Z_{W^{(\e)}}^*(t;L) - Z_W^*(t;L) \right| \right). \label{uWuWebound}
\end{align}
Now, let
\[
 \tau_W^{(L,\e)} = \tau_W((-L,L]) - \tau_{W^{(\e)}}((-L,L]) = \sum_{k \in \Z} y_k \zeta_k \ind{x_k \in (-L,L], \, y_k < \e}, 
\]
so that $\tau_W^{(L,\e)}$ is how much less time it takes the process $Z_{W^{(\e)}}^*$ to cross from $L$ to $-L$ than it takes the process $Z_W^*$ to do so. 
With the above coupling of $Z_W^*$ and $Z_{W^{(\e)}}^*$ it is then clear that 
\[
 Z_W^*(t + \tau_W^{(L,\e)}; L) \leq Z_{W^{(\e)}}^*(t;L) \leq Z_W^*(t;L), \quad \forall t \leq \tau_{W^{(\e)}}((-L,L]). 
\]
In particular, this implies that the supremum in \eqref{uWuWebound} can be bounded by the maximum distance the process $Z_W^*(\cdot\,;L)$ travels during a time interval of length $\tau_W^{(L,\e)}$ before reaching $-L$. 
That is, using the notation 
\[
 Z_W^{(*,L)}(\cdot) = \max \left\{ Z_W^*(\cdot\, ;L), -L \right\},
\]
we have that 
\begin{equation}\label{uWuWebound2}
 \left| u\left(Z_{W^{(\e)}}^*(t;x_k)\right) - u\left( Z_W^*(t;x_k) \right) \right| 
\leq \Delta\left(u;  \Delta\left(Z_W^{(*,L)}; \tau_W^{(L,\e)} \right)  \right). 
\end{equation}
Note that this bound is uniform over $t$ and $|x_k|\leq L$. 
Therefore, it follows from \eqref{uWuWebound2} and the definition of the function $u_W$ that 
\begin{equation}\label{uWuWebound3}
 \sup_{t} \sup_{k: |x_k|\leq L} \left|  u_{W^{(\e)}}(t,x_k)   - u_W(t,x_k) \right| 
\leq \Ev\left[  \Delta\left(u;  \Delta\left(Z_W^{(*,L)}; \tau_W^{(L,\e)} \right)  \right) \right]. 
\end{equation}
Since $\tau_W^{(L,\e)} \ra 0$ almost surely and since the process $Z_W^{(*,L)}$ is almost surely continuous (here we are using that $W \in \mathcal{T}'$), decreasing, and bounded below we can conclude that
\[
 \lim_{\e\ra 0} \Delta\left(Z_W^{(*,L)}; \tau_W^{(L,\e)} \right) = 0, \quad \Pv\text{-a.s.}
\]
Then, since $\Delta(u;\d)$ is bounded and vanishes as $\d\ra 0$, we can conclude by the bounded convergence theorem that the right side of \eqref{uWuWebound3} vanishes as $\e \ra 0$. 
This finishes the proof of \eqref{WeWconv}.

\subsection{The error from the truncated systems}
Here we will prove \eqref{truncerror}.
First of all, we describe how we will couple the system $\eta_t^n$ and $\eta_t^{n,\e}$. 
Clearly we can couple the initial conditions by letting 
\[
 \eta_0^{n,\e}(x_k^n) = \begin{cases} \eta_0^n(x_k^n) & \text{if } y_k^n \geq \e \\ 0 & \text{if } y_k^n < \e. \end{cases}
\]
We will prove \eqref{truncerror} by showing that the directed trap processes $Z_{W_n}^{k,j}$ and $Z_{W_n^{(\e)}}^{k,j}$ started at the locations with $y_k\geq \e$ can be coupled so that the differences are typically small, and then by showing that the number of particles in $\eta_0^n$ that start at a trap with $y_k^n < \e$ do not contribute much to $\langle \phi_t, \pi_t^n\rangle$ if $\e>0$ is sufficiently small. 
As was done for the coupling of $Z_W^*$ with $Z_{W^{(\e)}}^*$ in the previous section, we can couple $Z_{W_n}^{k,j}$ and $Z_{W_n^{(\e)}}^{k,j}$ by using the same exponential random variables $\zeta_k$ to generate the waiting times at the traps with $y_k^n \geq \e$. 
With this coupling, then similarly to the proof of \eqref{uWuWebound3} in the previous section we can show that
\[
\max_{k: \, y_k^n \geq \e, \, |x_k^n|\leq L} \Ev\left[ \sup_t \left| \phi\left(t,Z_{W_n}^{k,j}(t)\right) - \phi\left(t,Z_{W_n^{(\e)}}^{k,j}(t)\right) \right| \right]
\leq \Ev\left[  \Delta\left(\phi;  \Delta\left(Z_{W_n}^{(L)}; \tau_{W_n}^{(L,\e)} \right)  \right) \right],
\]
where we use the notation 
\[
 Z_{W_n}^{(L)}(\cdot) = \min\{ Z_{W_n}(\cdot\,;-L),\, L\} \quad \text{and}\quad  \tau_{W_n}^{(L,\e)} = \tau_{W_n} ([-L,L)) - \tau_{W_n^{(\e)}}([-L,L)).
\]
Therefore, with this coupling of the two systems we have that 
\begin{align}
&\Pv\left( \sup_t | \langle \phi_t, \pi_t^n \rangle - \langle \phi_t, \pi_t^{n,\e}\rangle | \geq \d \right) \nonumber \\
&\quad \leq \Pv\left( \frac{1}{a_n} \sum_{k: \, y_k^n<\e} \eta_0^n(x_k^n) \geq \frac{\d}{2\|\phi\|_\infty} \right) \nonumber \\
&\quad \qquad + \Pv\left( \sum_{k: \, y_k^n \geq \e} \sum_{j=1}^{\eta_0^n(x_k^n)}  \sup_t \left| \phi\left(t,Z_{W_n}^{k,j}(t)\right) - \phi\left(t,Z_{W_n^{(\e)}}^{k,j}(t)\right) \right| \geq \frac{\d a_n}{2}  \right) \nonumber \\
&\quad \leq  \frac{2 \|\phi\|_\infty}{\d} \sum_{k: \, y_k^n < \e} u(x_k^n) y_k^n
+ \frac{2}{\d a_n} \sum_{k: \, y_k^n \geq \e} \Ev\left[\eta_0^n(x_k^n) \right] \Ev\left[  \Delta\left(\phi;  \Delta\left(Z_{W_n}^{(L)}; \tau_{W_n}^{(L,\e)} \right)  \right) \right] \nonumber \\
&\quad \leq  \frac{2 \|\phi\|_\infty}{\d} \sum_{k: \, y_k^n < \e} u(x_k^n) y_k^n
+ \frac{2 }{\d} \Ev\left[  \Delta\left(\phi;  \Delta\left(Z_{W_n}^{(L)}; \tau_{W_n}^{(L,\e)} \right)  \right) \right] \sum_{k: \, y_k^n \geq \e} u(x_k^n) y_k^n \label{truncerror2}
\end{align}
Due to Assumption \ref{asmtraps} and the fact that $u$ is bounded and has compact support, the first term on the right will vanish as $n\ra\infty $ and then $\e \ra 0$. For the second term on the right, note that Assumptions \ref{asmtrapc} and \ref{asmtraps} imply that 
\[
 \lim_{\e \ra 0} \limsup_{n\ra\infty} \sum_{k: \, y_k^n \geq \e} u(x_k^n) y_k^n = \int_\R u(x) \s_W(dx) < \infty. 
\]
Therefore, the final sum on the right in \eqref{truncerror2} is uniformly bounded in $\e$ and $n$, and so to finish the proof of \eqref{truncerror} we need only to show that 
\[
 \lim_{\e\ra 0} \limsup_{n\ra 0} \Ev\left[  \Delta\left(\phi;  \Delta\left(Z_{W_n}^{(L)}; \tau_{W_n}^{(L,\e)} \right)  \right) \right] = 0.
\]
To this end, since $\d\mapsto \Delta(\phi;\d)$ is uniformly bounded and vanishes as $\d \ra 0$, it is enough to show that
\[
 \lim_{\e\ra 0} \limsup_{n\ra\infty} \Pv\left( \Delta\left(Z_{W_n}^{(L)}; \tau_{W_n}^{(L,\e)}\right) > \d' \right) = 0, \quad \forall \d'>0. 
\]
To prove this, note that for any $\e' > 0$, 
\begin{equation}\label{truncerror3}
\Pv\left( \Delta\left(Z_{W_n}^{(L)}; \tau_{W_n}^{(L,\e)}\right) > \d' \right) 
\leq \Pv\left(  \tau_{W_n}^{(L,\e)} \geq \e' \right) + \Pv\left( \Delta(Z_{W_n}^{(L)}; \e' ) > \d' \right)
\end{equation}
Since 
$\Ev[ \tau_{W_n}^{(L,\e)} ] = \sum_k y_k^n \ind{x_k^n \in [-L,L), \, y_k^n < \e}$, 
it follows from Assumption \ref{asmtraps} that the first term on the right in \eqref{truncerror3} vanishes if we first let $n\ra\infty$ and then $\e\ra 0$ for any fixed $\e'>0$. 
For the last term on the right in \eqref{truncerror3}
we will need the following lemma whose proof we postpone for the moment. 
\begin{lem}\label{ZWnconv}
 If Assumptions \ref{asmtrapc} and \ref{asmtraps} hold, then $Z_{W_n}^{(L)}(\cdot) = \min\{ Z_{W_n}(\cdot\, ;-L), \, L\}$ converges in distribution to $Z_W^{(L)}(\cdot) = \min\{ Z_{W}(\cdot\,;-L), \, L\}$ 
in the space $D^U_{\R_+}$
\end{lem}
Since $x \mapsto \Delta(x,\e')$ is a continuous mapping from $D_{\R_+}^U$ to $\R$, it follows from Lemma \ref{ZWnconv} that 
\[
 \lim_{n\ra\infty} \Pv\left( \Delta(Z_{W_n}^{(L)}; \e' ) > \d' \right) =  \Pv\left( \Delta( Z_W^{(L)}; \e') > \d' \right), 
\]
and since $Z_W^{(L)}(\cdot)$ is almost surely continuous, the right side can thus be made arbitrarily small by taking $\e' \ra 0$.
This completes the proof of \eqref{truncerror}, pending the proof of Lemma \ref{ZWnconv}. 

\begin{proof}[Proof of Lemma \ref{ZWnconv}]
We claim that it is enough to show that 
\begin{equation}\label{tauWnconv}
 \tau_{W_n}([-L,\cdot)) \Longrightarrow \tau_W([-L,\cdot)) \quad\text{on the space } D_{[-L,L]}^J.
\end{equation}
To see this, recall that the process $Z_{W_n}(\cdot;-L)$ was constructed by a time/space inversion of the process $\tau_{W_n}([-L,\cdot))$, and note that this time/space inversion functional is continuous at paths that are strictly increasing (see \cite[Corollary 13.6.4]{wSPL}). Since the limiting trap environment $W \in \mathcal{T}'$, the process $x\mapsto \tau_W([-L,x))$ is strictly increasing and then it follows from \cite[Corollary 13.6.4]{wSPL} and \eqref{tauWnconv} that $Z_{W_n}^{(L)}(\cdot)$ converges in distribution to $Z_W^{(L)}(\cdot)$.
Since the limiting process $Z_W^{(L)}(\cdot)$ is continuous, this convergence can be taken with respect to the uniform topology. 

To prove \eqref{tauWnconv}, first recall that for any $\e>0$ such that $W(\R \times \{\e\}) = 0$ (which is true for all but countably many $\e>0$) we can represent the truncated trap environments $W_n^{(\e)}$ and $W^{(\e)}$ as in \eqref{WneWerep}. 
From this representation it is easy to see that for any such $\e>0$ the process $\tau_{W_n^{(\e)}}([-L,\cdot))$ converges in distribution to the process $\tau_{W^{(\e)}}([-L,\cdot))$ in the space $D_{[-L,L]}^{J}$. 
Since this is true for arbitrarily small $\e>0$, \eqref{tauWnconv} will follow if we can show that 
\begin{equation}\label{tauWeerror}
  \lim_{\e \ra 0} \sup_{x\in [-L,L]} \left|\tau_{W^{(\e)}}([-L,x)) - \tau_{W}([-L,x)) \right| = 0, \quad \text{in $\Pv$-probability}, 
\end{equation}
and 
\begin{equation}\label{tauWneerror}
\lim_{\e\ra 0} \limsup_{n\ra\infty} \Pv\left( \sup_{x\in [-L,L]} \left|\tau_{W_n^{(\e)}}([-L,x)) - \tau_{W_n}([-L,x)) \right| \geq \d \right) = 0, \quad \forall \d>0. 
\end{equation}
In both \eqref{tauWeerror} and \eqref{tauWneerror} we will couple $\tau_{W^{(\e)}}$ with $\tau_W$ and $\tau_{W_n^{(\e)}}$ with $\tau_{W_n}$ by using the same exponential random variables $\zeta_k$ at traps that are in both environments. In this way the supremum over $x \in [-L,L]$ in both \eqref{tauWeerror} and \eqref{tauWneerror} is clearly achieved at $x=L$. From this coupling, it is clear that $\tau_{W^{(\e)}}([-L,L))$ increases to $\tau_W([-L,L))$ as $\e \ra 0$, $\Pv$-a.s., and thus \eqref{tauWeerror} holds. 
To see that \eqref{tauWneerror} holds, note that 
\[
  \Pv\left( \left|\tau_{W_n^{(\e)}}([-L,L)) - \tau_{W_n}([-L,L)) \right| \geq \d \right)
\leq \frac{1}{\d} \sum_{k} y_k \ind{x_k^n \in [-L,L), \, y_k^n < \e}.
\]
Since Assumption \ref{asmtraps} implies that this last sum vanishes as first $n\ra\infty$ and then $\e \ra 0$, this implies \eqref{tauWneerror}.  
\end{proof}

\section{The trap structure in RWRE}\label{sec:trapst}

In this section we review how a trapping structure can be identified in the environment $\w$ for a RWRE. 
This trapping structure will then later be used to construct a system of independent particles in a directed trap environment that can be effectively coupled with the system of independent RWRE. 
Several slightly different approaches have been used recently to identify trapping structures within RWRE \cite{dgWQL,estzWQL}, but we will follow for the most part the approach and terminology developed in \cite{pzSL1,p1LSL2,psWQLTn,psWQLXn} by the second author of the present paper.  

We begin the identification of the trap structure of the environment by recalling the notion of the \emph{potential} of the environment, first introduced by Sinai in \cite{sRRWRE}. 
For any fixed environment $\w$, we can define the potential $V_\w:\Z \ra \R$ of the environment by 
\begin{align*}
 V_\w(x) = 
\begin{cases}
 \sum_{i=0}^{x-1} \log \rho_i &  \text{if } x \geq 1 \\
 0 &  \text{if } x = 0\\
-\sum_{i=x}^{-1} \log \rho_i &  \text{if } x \leq -1. 
\end{cases}
\end{align*}
Next, we will define a doubly-infinite sequence $\{\nu_k\}_{k \in \Z}$ that we will refer to as the \emph{ladder locations} of the environment. 
These will be defined by 
\begin{align}
 \nu_0 &= \sup \{x \leq 0 : \, V_\w(y) > V_\w(x), \, \forall y < x \}, && \nonumber \\
\nu_{k} &= \sup \{x < \nu_{k+1}: \, V_\w(y) > V_\w(x), \, \forall y < x \}, &&  \forall k \leq -1 \label{nukdef}\\
\nu_k &= \inf \{ x > \nu_{k-1}: V_\w(x) < V_\w(\nu_{k-1}) \}, && \forall k \geq 1. \nonumber 
\end{align}

The ladder locations of the environment serve to identify ``traps'' for the RWRE in the following manner. Since $E_P[\log \rho_0] < 0$ by Assumption \ref{asmt}, the potential $V_\w$ of the environment is generally decreasing. However, there may be atypical long intervals of the environment where the potential is instead increasing. Since a decreasing (or increasing) potential indicates a local drift of the random walk to the right (or left), the atypical long intervals where the potential is increasing act as barriers or traps that the random walk must overcome. The ladder locations $\nu_k$ are constructed so that any interval where the potential is increasing must lie between two consecutive ladder locations $\nu_k$ and $\nu_{k+1}$. 

In order to prove the effective coupling of the RWRE with the related directed trap process, it is most convenient to make a slight change in the law on environments that we are using. 
Let
\begin{equation}\label{blockdef}
 B_k = (\w_{\nu_k}, \w_{\nu_k+1}, \ldots, \w_{\nu_{k+1}-1}), \qquad k \in \Z, 
\end{equation}
be the ``block'' of the environment on the interval $[\nu_k,\nu_{k+1})$. 
It follows from the definition of the ladder locations and Assumption \ref{asmiid} that $\{B_k\}_{k\in\Z}$ is an independent sequence and that for any $k\neq 0$ the block $B_k$ has the same distribution as 
\begin{equation}\label{blockdist}
 \tilde{B}_0 = (\w_0,\w_1,\ldots,\w_{\tilde{\nu}-1}),\quad \text{where}\quad \tilde{\nu} = \inf\{x>0: V_\w(x) < V_\w(0) \}.
\end{equation}
However, under the i.i.d.\ measure $P$ on environments the block $B_0$ containing the origin has a different distribution than all the other blocks (this is an instance of the ``inspection paradox''). 
We can make all the blocks between ladder locations have the same distribution by changing the distribution $P$ on environments to the distribution $Q$ on environments defined by 
\[
 Q( \w \in \cdot ) = P( \w \in \cdot \, | \, V_\w(y) > 0, \, \forall y < 0 ) = P(\w \in \cdot \, | \, \nu_0 = 0 ). 
\]
The measure $Q$ is stationary under shifts of the environment by the ladder locations in the following sense: 
if $\theta$ is the natural left shift operator on environments so that 
$(\theta^y \w)_x = \w_{y+x}$ for any $y,x \in \Z$, then
$\w$ and $\theta^{\nu_k}\w$ have the same distribution under $Q$ for any $k$. 
The environment $\w$ is no longer i.i.d.\ under the measure $Q$, but instead the blocks $B_k$ are i.i.d., all with the same distribution as $\tilde{B}_0$ as defined in \eqref{blockdist} (note that the distribution of $\tilde{B}_0$ is the same under the measures $P$ and $Q$ since the conditioning in the definiton of $Q$ only changes the environment to the left of the origin). 

\begin{rem}
 The definition of the ladder locations given in \eqref{nukdef} is slightly different from the one used in \cite{pzSL1,p1LSL2,psWQLTn,psWQLXn}. However, the distribution $Q$ is the same as in those previous papers and under the measure $Q$ the ladder locations are the same in this paper as in the previous papers. 
\end{rem}

\subsection{Coupling RWRE and directed traps}

The coupling of the RWRE with a directed trap process is obtained through the hitting times of the processes. For a RWRE $X_n$ let the hitting times be defined by $T_x = \inf\{n\geq 0: X_n = x\}$ for any $x \in \Z$. 
Since we expect the intervals between ladder locations to serve as traps for the RWRE, to each such interval $[\nu_k,\nu_{k+1})$ we identify the (quenched) expected crossing time as
\begin{equation}\label{bkdef}
 \b_k = \b_k(\w) = E_\w^{\nu_k}[ T_{\nu_{k+1}} ], \quad \forall k \in \Z. 
\end{equation}
Given an environment $\w$, we can thus use the parameters $\{\b_k(\w) \}_{k\in\Z}$ and the ladder locations $\{\nu_k \}_{k\in \Z}$ to define a trap environment. 
That is, we will define the point process $\mathfrak{B} = \mathfrak{B}(\w)$ by 
\begin{equation}\label{Btrapenv}
 \mathfrak{B} = \sum_{k \in \Z} \d_{(\nu_k, \b_k)}. 
\end{equation}
Techniques were developed in \cite{psWQLTn,psWQLXn} for coupling a random walk in the environment $\w$ with a directed trap process in the trap environment $\mathfrak{B}$ (we will explain this more fully below). 
In this way, understanding the probabilistic structure of the trap environment $\mathfrak{B}$ is useful for analyzing the behavior of the RWRE. 
To this end, the following tail estimates on the trap structure which were proved in \cite[Theorem 1.4 and Lemma 2.2]{pzSL1} are useful. 
\begin{equation}\label{bnutails}
 Q(\b_1 > x) \sim C_1 x^{-\k}, \quad \text{and} \quad Q( \nu_1 > x ) \leq C_2 e^{-C_3 x},
\end{equation}
for some constants $C_1, C_2, C_3>0$.
It follows from the tail estimate on $\nu_1$ that $\bar\nu:= E_Q[\nu_1] < \infty$. 
Moreover, since the sequence $\{\nu_{k+1}-\nu_k\}_{k\in\Z}$ is i.i.d.\ under the measure $Q$, we can conclude that 
\begin{equation}\label{unifnuk}
 \lim_{n\ra\infty} \frac{\sup_{|k| \leq n} \left| \nu_k - k \bar\nu\right|}{n} = 0, \quad \text{ in $Q$-probability.}
\end{equation}

The sequence $\{\b_k\}_{k\in\Z}$ is stationary and ergodic under $Q$ but not independent. 
Nonetheless, the sequence is close enough to independent that the following result on the limiting structure of the trap environment was obtained in 
\cite{psWQLXn}.

\begin{lem}[Lemma 5.1 in \cite{psWQLXn}]\label{lem:barWnconv}
Let $\overline{W}_n= \overline{W}_n(\w)$ be the point process on $\R \times (0,\infty]$ given  by 
\[
 \overline{W}_n  = \sum_k \d_{(\frac{k \bar\nu}{n}, \frac{\b_k}{n^{1/\k}})}. 
\]
There exists a constant $\l>0$ such that under the measure $Q$ on environments, $\overline{W}_n$ converges in distribution to a Poisson point process with intensity measure $\l y^{-\k-1} \, dx \, dy$. 
\end{lem}
\begin{rem}
Note that if the $\b_k$ were in fact independent, the conclusion of Lemma \ref{lem:barWnconv} would follow from the tail decay of $\b_1$ given in \eqref{bnutails}. 
Also, the convergence in \cite{psWQLXn} was actually proved for the point process with atoms at $(k/n,\b_k/n^{1/\k})$ instead of  $(k \bar\nu/n,\b_k/n^{1/\k})$, but of course this only changes the value of the constant $\l>0$ in the intensity measure of the limiting Poisson point process. 
\end{rem}

The Poissonian limit of the trap structure in Lemma \ref{lem:barWnconv} implies the following corollaries that we will use throughout the remainder of the paper. The proofs of these corollaries will be given in Appendix \ref{app:Wnweakconv}. 
\begin{cor}\label{cor:Wnweakconv}
For any environment $\w$, let the rescaled trap environment $W_n$ be given by 
\begin{equation}\label{Wndef}
 W_n = \sum_{k \in \Z} \d_{(\frac{\nu_k}{n \vphantom{n^{1/\kappa}}}, \frac{\b_k}{n^{1/\k}})}. 
\end{equation}
There exists a constant $\l>0$ such that under the measure $Q$ the pair $(W_n,\s_{W_n})$ converges in distribution on the space $\mathcal{M}_p \times D_\R^J$ to $(W,\s_W)$ where $W$ is a Poisson point process on $\R \times (0,\infty]$ with intensity measure $\l y^{-\k-1} \, dx \, dy$ and $\s_W$ is the corresponding trap measure defined in \eqref{sWdef}. 
\end{cor}


\begin{cor}\label{cor:bstable}
 Under the measure $Q$ on environments, $n^{-1/\k} \sum_{|k|\leq n} \b_k$ converges in distribution to a non-negative $\k$-stable random variable. 
\end{cor}

\begin{cor}\label{cor:tBstable}
Under the averaged measure $E_Q[ P_{\w}(\cdot)]$ for the directed trap process $Z_{\mathfrak{B}}$, the 
rescaled crossing times $n^{-1/\k} \tau_{\mathfrak{B}}([0,\nu_n))$ 
converge in distribution 
to a non-negative $\k$-stable random variable, and the rescaled directed trap process $\{t\mapsto n^{-1} Z_{\mathfrak{B}}(t n^{1/\k};0)\}$ converges in distribution to the inverse of a $\k$-stable subordinator
 on the space $D_{\R_+}^U$. 
\end{cor}

Having given the Poissonian limit of the trap environment, we now turn to a review of the coupling of the random walk in the environment $\w$ with the directed trap process in the trap environment $\mathfrak{B}$.
We will give a brief overview of the nature of this coupling and refer the reader to \cite{psWQLTn,psWQLXn} for more details. 
First of all, we expand the measure $P_\w$ to contain an i.i.d.\ sequence of Exp(1) random variables $\{\zeta_k\}_{k \in \Z}$. These exponential random variables are used to generate the holding times of the directed trap process at each trap location. That is, the directed trap process $Z_{\mathfrak{B}}$ waits at location $\nu_{k}$ for time $\beta_k \zeta_k$ before jumping to location $\nu_{k+1}$. 
The coupling of this process with a random walk in the environment $\w$ is then obtained by coupling each crossing time $T_{\nu_{k+1}} - T_{\nu_k}$ between successive ladder locations with the time $\beta_k \zeta_k$ that it takes the directed trap process to cross the same distance. 
Without giving the details of this coupling procedure, we simply note that this coupling is done so that the sequence of the coupled hitting times $\{\left(T_{\nu_{k+1}} - T_{\nu_k}, \b_k \zeta_k\right)\}_{k \geq 0}$ is independent under the measure $P_\w$. 
Moreover, in \cite[Lemma 4.4]{psWQLXn} it was shown that this coupling can be done so that 
\begin{equation}\label{htcouple}
 \lim_{n\ra\infty} E_Q\left[ P_\w\left( \sup_{k\leq An} |T_{\nu_k} - \tau_{\mathfrak{B}}([0,\nu_k))| \geq \d n^{1/\k} \right) \right] = 0, \quad \forall \d>0, \, A<\infty. 
\end{equation}
(Since under the Assumptions \ref{asmiid}--\ref{asmnl} the rescaled hitting times $n^{-1/\k} T_n$ converge in distribution with respect to the averaged measure $\P$ \cite{kksStable}, the above coupling is useful for comparing the asymptotic distributions of the hitting times for the two processes.)
We note that the coupling as constructed only couples the hitting times of the two processes. The path of the random walk is then constructed by first determining the crossing times $T_{\nu_{k+1}} - T_{\nu_k}$ via this coupling and then by sampling the paths of the walk 
$\{X_i, \, T_{\nu_k} \leq i \leq T_{\nu_{k+1}} \}$
between hitting times of successive ladder locations
with respect to the quenched measure conditioned on the values of the crossing times. 
Our next result shows that this coupling procedure yields the following comparison of the locations of the random walk and the directed trap process. 

\begin{prop}\label{XZcouple}
 For any environment $\w$ we can expand the quenched distribution $P_\w$ to give a coupling of the random walk $\{X_n\}_{n\geq 0}$ with a directed trap process $\{Z_{\mathfrak{B}}(t;0)\}_{t\geq 0}$ in the trap environment $\mathfrak{B}$ in such a way that 
\[
 \lim_{n\ra\infty} E_Q\left[ P_\w\left( \sup_{t\leq T} \left| X_{t n^{1/\k}} - Z_{\mathfrak{B}}(t n^{1/\k};0) \right| \geq \e n \right) \right] = 0, \quad \forall T<\infty, \, \e>0. 
\]
\end{prop}
\begin{proof}
For convenience of notation, in the proof of the proposition we will denote the directed trap process $Z_{\mathfrak{B}}(t;0)$ by $Z_{\mathfrak{B}}(t)$ instead. 
Let $X_n^* = \max_{i\leq n} X_i$ be the running maximum of the random walk. 
It follows from \cite[Lemma 6.1]{psWQLXn} that 
$\sup_{t\leq T} n^{-1} ( X_{tn^{1/\k}}^* - X^{\vphantom{*}}_{t n^{1/\k}} )$ 
converges to $0$ in $\P$-probability\footnote{In fact, the proof of Lemma 6.1 in \cite{psWQLXn} can be used to show that the convergence is $\P$-a.s.}. 
Since the measure $Q$ on environments is obtained by conditioning the measure $P$ on an event of positive probability we can conclude that 
$\sup_{t\leq T} n^{-1} ( X^*_{tn^{1/\kappa}} -X^{\vphantom{*}}_{tn^{1/\kappa}} )$ 
converges to $0$ in probabability with respect to the averaged measure $E_Q[P_\w(\cdot)]$ as well. Therefore, it is enough to show that
\[
 \lim_{n\ra\infty} E_Q\left[ P_\w\left( \sup_{t\leq T} \left| X_{t n^{1/\k}}^* - Z_{\mathfrak{B}}(t n^{1/\k}) \right| \geq \e n \right) \right] = 0, \quad \forall T<\infty, \, \e>0. 
\]

We now couple the random walk with the directed trap process according the the procedure outlined above prior to the statement of Proposition \ref{XZcouple}. 
 To use the control of the hitting times in \eqref{htcouple} to obtain control on $|X_{tn^{1/\k}}^* - Z_\mathfrak{B}(tn^{1/\k})|$, note that 
\begin{equation}\label{spacetimecouple}
\begin{split}
& \left\{ \sup_{k\leq An} |T_{\nu_k} - \tau_{\mathfrak{B}}([0,\nu_k))| < \d n^{1/\k} \right\} 
\cap \left\{ \max_{k\leq An} (\nu_{k} - \nu_{k-1}) \leq \frac{\e n}{2} \right\} \\
& \subset \left\{ Z_\mathfrak{B}((t-\d)n^{1/\k}) \leq X_{t n^{1/\k}}^* \leq  Z_\mathfrak{B}((t+\d)n^{1/\k}) + \frac{\e n}{2}, \text{ for all } t\leq \frac{\tau_\mathfrak{B}([0,\nu_{\fl{An}}))}{n^{1/\k}} - \d \right\}.
\end{split}
\end{equation}
To see this, first of all note that $t\leq \frac{\tau_\mathfrak{B}([0,\nu_{\fl{An}}))}{n^{1/\k}} - \d $ implies that $Z_\mathfrak{B}((t+\d)n^{1/\k}) = \nu_k$ for some $k < \fl{An}$. Since the process $Z_\mathfrak{B}$ is non-decreasing this implies that $\tau_\mathfrak{B}([0,\nu_{k+1})) > (t+\d)n^{1/\k}$, and then the control on the hitting times in the first event in \eqref{spacetimecouple} implies that $T_{\nu_{k+1}} > t n^{1/\k}$, or equivalently that 
\[
 X_{t n^{1/\k}}^* < \nu_{k+1} \leq Z_\mathfrak{B}((t+\d)n^{1/\k}) + (\nu_{k+1}-\nu_k).
\]
This proves the upper bound on $X_{t n^{1/\k}}^*$ needed in the event on the right side of \eqref{spacetimecouple}. The corresponding lower bound on $X_{tn^{1/\k}}^*$ is proved similarly. 

Using \eqref{spacetimecouple} we can conclude that
\begin{align*}
 & E_Q\left[ P_\w\left( \sup_{t\leq T} \left| X_{t n^{1/\k}}^* - Z_\mathfrak{B}(t n^{1/\k}) \right| \geq \e n \right)  \right] \\
&\leq  E_Q\left[ P_\w\left( \sup_{k\leq An} \left| T_{\nu_k} - \tau_\mathfrak{B}([0,\nu_k)) \right| \geq \d n^{1/\k} \right) \right]  +  E_Q\left[ P_\w\left( \tau_\mathfrak{B}([0,\nu_{\fl{An}})) \leq (T+\d) n^{1/\k} \right) \right] \\
&\quad +  E_Q\left[ P_\w\left( \sup_{t\leq T} Z_\mathfrak{B}((t+\d)n^{1/\k}) - Z_\mathfrak{B}((t-\d)n^{1/\k}) \geq \frac{\e n}{2} \right) \right]+ Q\left( \max_{k\leq An} (\nu_{k} - \nu_{k-1}) > \frac{\e n}{2} \right)
\end{align*}
Equation \eqref{htcouple} shows that the first term on the right vanishes as $n\ra\infty$ for any fixed $A<\infty$ and $\d>0$, and 
since the $\nu_{k+1}-\nu_k$ are i.i.d.\ with exponential tails the last term on the right vanishes for any fixed $\e>0$. 
The remaining two terms are handled by Corollary \ref{cor:tBstable}. 
The second term vanishes as we first take $n\ra\infty$ and then let $A\ra\infty$ for any $\d>0$ fixed,
and since  
the process $t\mapsto n^{-1} Z_\mathfrak{B}(tn^{1/\k})$ converges in distribution on $D_{\R_+}^U$ to the inverse of a $\k$-stable subordinator (which is a continuous, non-decreasing process), the third probability on the right vanishes as $n\ra\infty$ and then $\d\ra 0$. 
\end{proof}

\section{Coupling the RWRE system with a directed trap system}\label{sec:couple1}

In this section we will use the couplings of a RWRE with a directed trap trap process to give a coupling estimate on systems of independent RWRE and systems of independent particles in a directed trap environment. 
Unfortunately this will not quite be enough to give us the desired hydrodynamic limit for systems of independent RWRE as stated in Theorem \ref{th:RWREhdl} since the systems of independent RWRE considered in this section differ in two ways from those in the statement of Theorem \ref{th:RWREhdl}. 
\begin{itemize}
 \item The environment $\w$ will be chosen according to the distribution $Q$ instead of the original distribution $P$ on environments. 
 \item The initial configurations of particles $\chi_0$ will be different from the locally stationary initial configurations in Theorem \ref{th:RWREhdl}. In particular, in this section we will use initial configurations where all the particles start at some ladder location $\nu_k$ of the environment. 
\end{itemize}
These two difficulties will be resolved in the following section. 

To describe the system of directed trap processes that we will couple with systems of RWRE we first need to introduce some new notation for the law of the directed trap processes with certain initial configurations. If $W = \sum_k \d_{(x_k,y_k)} \in \mathcal{T}$ is a trap environment and $u:\R \ra (0,\infty)$ is a continuous function, we will let $P_W^u$ denote the law of the system $\eta_t^W$ of independent directed trap processes in the trap environment $W$ with an initial configuration that is product Poisson with $\eta_0^W(x_k) \sim \text{Poisson}( u(x_k) y_k )$. 
We will be interested in systems of directed trap processes in the trap environment $\mathfrak{B} = \sum_k \d_{(\nu_k, \b_k)}$ with initial conditions that are product Poisson with $\eta_0^{\mathfrak{B}}(\nu_k) \sim \text{Poisson}(u(\nu_k/n) \b_k )$. 
If we let $u_n(\cdot) = u(\cdot/n)$ denote a rescaled modification of the function $u$, then $P_{\mathfrak{B}}^{u_n}$ is the law of such a system of directed trap processes. 
The following Proposition shows how Theorem \ref{th:dthdl} can be used to give a hydrodynamic limit for these systems of independent directed trap processes.


\begin{prop}
There exists a constant $\l>0$ such that for any $u \in \mathcal{C}_0^+$ and $\phi \in \mathcal{C}_0(\R_+ \times \R)$,
 \begin{align*}
  \lim_{n\ra\infty}  & E_Q\left[ P_{\mathfrak{B}}^{u_n} \left( \frac{1}{n^{1/\k}} \int \sum_{k \in \Z} \eta_{t n^{1/\k}}^{\mathfrak{B}} (\nu_k) \phi(t,\nu_k/n) dt  \in \cdot \right) \right] \\
&\qquad\qquad =  \Pv\left( \iint u_W(t,x) \phi(t,x)  \, \s_W(dx) \,dt\in \cdot \right), 
 \end{align*}
where $u_n(\cdot) = u(\cdot/n)$ and under the measure $\Pv$, $W$ is a Poisson point process on $\R \times (0,\infty)$ with intensity measure $\l y^{-\k-1} \, dx \, dy$. 
\end{prop}
\begin{proof}
 Note that in the hydrodynamic limit we are trying to prove, we are rescaling space by $n$ and time by $n^{1/\k}$. This rescaling can instead be incorporated into the trap environment. In particular, if we let
 $W_n$ be the rescaled version of $\mathfrak{B}$ as defined in \eqref{Wndef}
then it is easy to see that 
\begin{align*}
&  P_{\mathfrak{B}}^{u_n} \left( \frac{1}{n^{1/\k}} \int \sum_{k \in \Z} \eta_{t n^{1/\k}}^{\mathfrak{B}} (\nu_k) \phi(t,\nu_k/n) \, dt \in \cdot \right)\\
&\qquad = P_{W_n}^{n^{1/\k} u} \left( \frac{1}{n^{1/\k}} \int \sum_{k \in \Z} \eta_t^{W_n}(\nu_k/n) \phi(t,\nu_k/n) \, dt \in \cdot \right). 
\end{align*}
The probabilities on the right side are in the right format to apply Theorem \ref{th:dthdl} with $a_n = n^{1/\k}$, but unfortunately the point processes $W_n$ do not converge almost surely to a fixed $W \in \mathcal{M}_p$. 
However, a consequence of Corollary \ref{cor:Wnweakconv} is that
there exists a probability space with measure $\Pv$ containing a sequence of point processes $\tilde{W}_n$ such that $\tilde{W}_n$ has the same distribution as $W_n$ for every $n\geq 1$, but for which $(\tilde{W}_n,\s_{\tilde{W}_n})$ converges $\Pv$-almost surely to a random pair $(W,\s_W)$ where $W$ is a Poisson point process with intensity measure $\l y^{-\k-1} \, dx \, dy$. 
It follows
that with probability one the sequence $\tilde{W}_n$ satisfies Assumptions \ref{asmtrapc} and \ref{asmtraps}. 
Thus, applying Theorem \ref{th:dthdl} we can conclude that 
\begin{align*}
& \lim_{n\ra\infty} E_Q\left[  P_{\mathfrak{B}}^{u_n} \left( \frac{1}{n^{1/\k}} \int \sum_{k \in \Z} \eta_{t n^{1/\k}}^{\mathfrak{B}} (\nu_k) \phi(t,\nu_k/n)  \, dt \leq z \right) \right] \\
&\quad = \lim_{n\ra\infty} \Ev \left[ P_{\tilde{W}_n}^{n^{1/\k} u} \left( \frac{1}{n^{1/\k}} \int \sum_{k \in \Z} \eta_t^{\tilde{W}_n}(\nu_k/n) \phi(t,\nu_k/n) \, dt \leq z \right) \right] \\
&\quad = \Pv\left( \iint u_W(t,x) \phi(t,x) \, \s_W(dx) \, dt \leq z \right)
\end{align*}
for any $z \in \R$ where the distribution function on the right is continuous at $z$. 
\end{proof}

Next we wish to couple the system of directed traps $\eta_t^{\mathfrak{B}}$ with a system of independent RWRE. Given an environment $\w$ and a function $u \in \mathcal{C}_0^+$, we will let $\eta_t^{\mathfrak{B}}$ have distribution $P_{\mathfrak{B}}^{u_n}$. That is, the initial configuration $\eta_0^{\mathfrak{B}}$ is product Poisson with $\eta_0^{\mathfrak{B}}(\nu_k) \sim \text{Poisson}(u(\nu_k/n) \b_k)$ for every $k \in \Z$. 
The related system of independent RWRE $\chi_t$ will have an initial configuration that is also product Poisson with 
\begin{equation}\label{rwnukic}
 \chi_0(x) \sim 
 \begin{cases}
  \text{Poisson}(u(\frac{\nu_k}{n}) \b_k) & \text{if } x = \nu_k \text{ for some } k \in \Z\\
  \d_0 & x \notin \{\nu_k\}_{k\in\Z}.
 \end{cases}
\end{equation}
We will denote the law of this initial configuration by $\hat{\mu}_u^n$ so that the system of RWRE has quenched law denoted by $P_\w^{\hat{\mu}_u^n}$. 

Note that we can obviously couple the initial configurations of these two systems so that $\eta_0^{\mathfrak{B}}(\nu_k) = \chi_0(\nu_k)$ for all $k \in \Z$; that is the systems start with the same number of particles at each site. Next, we can couple the evolution of the two systems by pairing each random walk particle with a corresponding directed trap particle at the same starting location. 
Each of these couplings is performed independently and is done according to the method given in \cite{psWQLTn,psWQLXn} which is outlined above prior to Proposition \ref{XZcouple}. 
With this coupling of the systems of particles we can obtain the following result. 

\begin{lem}\label{lem:chietacouple}
 For every environment $\w$ (with corresponding trap environment $\mathfrak{B}$) and every $n\geq 1$, there exists a coupling $P_{\w,\mathfrak{B}}^{\hat{\mu}_u^n,u_n}$ of a system $\chi$ of independent RWRE  in environment $\w$ with a system $\eta^{\mathfrak{B}}$ of directed trap particles  in the trap environment $\mathfrak{B}$ such that $\chi$ and $\eta^{\mathfrak{B}}$ have marginal distributions $P_\w^{\hat\mu_u^n}$ and $P_{\mathfrak{B}}^{u_n}$, respectively. Moreover, this coupling can be constructed so that 
\begin{equation}\label{chietacouple}
 E_Q\left[ P_{\w,\mathfrak{B}}^{\hat{\mu}_u^n,u_n}\left( \sup_{t\leq T} \left| \sum_{x \in \Z} \chi_{tn^{1/\k}}(x)\phi(t,x/n) - \sum_{k \in \Z} \eta_{t n^{1/\k}}^{\mathfrak{B}}(\nu_k) \phi(t,\nu_k/n) \right|  \geq \e n^{1/\k} \right)  \right] = 0, 
\end{equation}
for any $\e>0$, $T<\infty$ and $\phi \in \mathcal{C}_0(\R_+ \times \R)$. 
\end{lem}
\begin{proof}
Since the initial configurations $\chi_0 = \eta_0^{\mathfrak{B}}$ are the same, we can match the $j$-th particles at $\nu_k$ in each system and then compare the systems at any later time by comparing the differences in these particles at this later time. That is, 
\begin{align*}
 \sum_{x \in \Z} \chi_{tn^{1/\k}}(x)\phi(x/n) - \sum_{k \in \Z} \eta_{t n^{1/\k}}^{\mathfrak{B}}(\nu_k) \phi(\nu_k/n) 
= \sum_{k \in \Z} \sum_{j=1}^{\eta_0^{\mathfrak{B}}(\nu_k)} \left\{ \phi\left( \frac{X_{tn^{1/\k}}^{\nu_k,j} }{n} \right) - \phi\left( \frac{Z_{\mathfrak{B}}^{k,j}(tn^{1/\k})}{n} \right) \right\}
\end{align*}
Now, let $\phi \in \mathcal{C}_0(\R_+ \times \R)$ and $\e>0$ be fixed and choose $A<\infty$ (we will let $A \ra \infty$ later). Since $\phi$ is uniformly continuous, there exists a $\d>0$ such that $|\phi(x,t) - \phi(y,t)| < \frac{\e}{2A}$ if $|x-y| < \d$ and $t\leq T$. Therefore, for this choice of $A$ and $\d$ we have that
\begin{align*}
 & P_{\w,\mathfrak{B}}^{\hat{\mu}_u^n,u_n}\left( \sup_{t\leq T} \left| \sum_{x \in \Z} \chi_{tn^{1/\k}}(x)\phi(t,x/n) - \sum_{k \in \Z} \eta_{t n^{1/\k}}^{\mathfrak{B}}(\nu_k) \phi(t,\nu_k/n) \right|  \geq \e n^{1/\k} \right)   \\
 &\leq  P_{\mathfrak{B}}^{u_n}\left( \sum_k \eta_0^{\mathfrak{B}}(\nu_k) >  A n^{1/\k}  \right) \\
 &\qquad  +  P_{\w,\mathfrak{B}}^{\hat{\mu}_u^n,u_n}\left(  \sum_{k \in \Z} \sum_{j=1}^{\eta_0^{\mathfrak{B}}(\nu_k)} \ind{ \sup_{t\leq T} | X_{tn^{1/\k}}^{\nu_k,j} - Z_{\mathfrak{B}}^{k,j}(t n^{1/\k})   | \geq \d n } \geq \frac{\e n^{1/\k} }{ 4 \|\phi\|_\infty}  , \, \sum_k \eta_0^{\mathfrak{B}}(\nu_k) \leq A n^{1/\k} \right) \\
 &\leq P_{\mathfrak{B}}^{u_n}\left( \sum_k \eta_0^{\mathfrak{B}}(\nu_k) >  A n^{1/\k}  \right) 
 + \frac{4 \|\phi\|_\infty A}{\e} \max_{|k|\leq Ln} P_{\w,\mathfrak{B}} \left( \sup_{t\leq T} \left| X_{tn^{1/\k}}^{\nu_k} - Z_{\mathfrak{B}}^{k}(t n^{1/\k}) \right| \geq \d n \right), 
\end{align*}
where $L<\infty$ is such that $\supp u \subset [-L,L]$ so that all particles are started at sites $\nu_k \in [-Ln,Ln]$. 
(The last inequality above is obtained by first conditioning on the initial configuration and applying Chebychev's inequality.)
We wish to show that the above terms are small when we take expectations with respect to the measure $Q$ and then let $n\ra\infty$.
To handle the first term, note that if the environment is such that $\sum_k u(\nu_k/n) \b_k \leq \frac{A n^{1/\k}}{2}$ then 
$\sum_k \eta_0^{\mathfrak{B}}(\nu_k)$ is stochastically dominated by a Poisson($\frac{A n^{1/\k}}{2}$) random variable under the measure $P_{\mathfrak{B}}^{u_n}$. 
From this, we can obtain that 
\begin{align*}
 \limsup_{n\ra\infty} E_Q\left[ P_{\mathfrak{B}}^{u_n}\left( \sum_k \eta_0^{\mathfrak{B}}(\nu_k) >  A n^{1/\k}  \right) \right] 
&\leq \lim_{n\ra\infty} Q\left( \sum_k u\left(\frac{\nu_k}{n}\right) \b_k > \frac{A n^{1/\k}}{2} \right)\\
&= \Pv\left( \int u(x) \s_W(dx) > \frac{A}{2} \right), 
\end{align*}
where the last equality follows from Corollary \ref{cor:Wnweakconv}. 
Since this can be made arbitrarily small by taking $A \ra\infty$, it remains only to show that 
\[
 \lim_{n\ra\infty} E_Q\left[ \max_{|k|\leq Ln} P_{\w,\mathfrak{B}} \left( \sup_{t\leq T} \left| X_{tn^{1/\k}}^{\nu_k} - Z_{\mathfrak{B}}^{k}(t n^{1/\k}) \right| \geq \d n \right) \right] = 0, \qquad \forall L<\infty, \d>0. 
\]
Of course, by the shift invariance of $Q$ with respect to the ladder locations this is equivalent to showing that  
\begin{equation}\label{EXZuc}
\lim_{n\ra\infty} E_Q\left[ \max_{k\in [0,Ln]} P_{\w,\mathfrak{B}} \left( \sup_{t\leq T} \left| X_{tn^{1/\k}}^{\nu_k} - Z_{\mathfrak{B}}^{k}(t n^{1/\k}) \right| \geq \d n \right) \right] = 0, \quad \forall L<\infty, \, \d>0. 
\end{equation}
To control the probabilities inside the expectation above, fix $\d'>0$ and $S<\infty$ and note that for any $k\in[0,Ln]$, 
\begin{align}
& P_{\w,\mathfrak{B}} \left( \sup_{t\leq T} \left| X_{tn^{1/\k}}^{\nu_k} - Z_{\mathfrak{B}}^{k}(t n^{1/\k}) \right| \geq \d n \right) \nonumber \\
&\qquad \leq P_{\w,\mathfrak{B}} \left( \sup_{t\leq T} \left| X_{T_{\nu_k}+tn^{1/\k}}^{0} - Z_{\mathfrak{B}}^{0}(\tau_{\mathfrak{B}}([0,\nu_k)) + t n^{1/\k}) \right| \geq \d n \right) \nonumber \\
&\qquad \leq P_{\w,\mathfrak{B}}\left( \max_{k \in [0,Ln]} \left| T_{\nu_k} - \tau_{\mathfrak{B}}([0,\nu_k)) \right| \geq \d' n^{1/\k} \right) + P_\w(\tau_{\mathfrak{B}}([0,\nu_{Ln})) \geq S n^{1/\k} ) \label{XZuc1}\\
&\qquad\qquad + P_{\w,\mathfrak{B}} \left( \sup_{t\leq T+S+\d'} \left| X_{tn^{1/\k}}^{0} - Z_{\mathfrak{B}}^{0}(t n^{1/\k}) \right| \geq \d n/2 \right) \label{XZuc2} \\
&\qquad\qquad + P_{\w,\mathfrak{B}} \left( \sup_{t \leq T+S}  Z_{\mathfrak{B}}^{0}((t+2\d') n^{1/\k}) - Z_{\mathfrak{B}}^{0}(t n^{1/\k}) \geq \d n/2 \right) \label{XZuc3}
\end{align}
Since all of the probabilities in \eqref{XZuc1}-\eqref{XZuc3} do not depend on $k \in [0,Ln]$, in order to prove \eqref{EXZuc} it will be sufficient to control each of the terms in \eqref{XZuc1}-\eqref{XZuc3} when first taking expectations with respect to the measure $Q$ on environments and then letting $n\ra\infty$. 
That the first term in \eqref{XZuc1} vanishes in this way is the content of \eqref{htcouple} above. 
Proposition \ref{XZcouple} shows that \eqref{XZuc2} vanishes when averaging over $Q$ and then letting $n\ra\infty$ for any $S+T+\d'<\infty$ and $\d>0$. 
Finally, for the second term in \eqref{XZuc1} and the term in \eqref{XZuc3}, note that Corollary \ref{cor:tBstable} implies that under the averaged measure $E_Q[P_\w(\cdot) ]$, the crossing time $n^{-1/\k}\tau_{\mathfrak{B}}([0,\nu_{Ln}))$ converges in distribution to a $\k$-stable random variable $Y_\k$ and $\{t\mapsto n^{-1}Z_{\mathfrak{B}}(t n^{1/\k})\}$ converges in distribution to the inverse of a $\k$-stable subordinator $\{t\mapsto \mathcal{Z}(t)\}$. 
Therefore, we can conclude that
\begin{align*}
& \limsup_{n\ra\infty}  E_Q\left[ \max_{k\in [0,Ln]} P_{\w,\mathfrak{B}} \left( \sup_{t\leq T} \left| X_{tn^{1/\k}}^{\nu_k} - Z_{\mathfrak{B}}^{k}(t n^{1/\k}) \right| \geq \d n \right) \right] \\
&\qquad \leq \Pv\left( Y_\k \geq S \right) + \Pv\left( \sup_{t \leq T+S}  \mathcal{Z}(t+2\d')  - \mathcal{Z}(t) \geq \d/2 \right) 
\end{align*}
The right-hand side vanishes as we first take $\d'\ra 0$ and then let $S\ra\infty$ (note that here we're using that $\mathcal{Z}$ is almost surely a continuous process). 
This completes the proof of \eqref{EXZuc} and thus also the proof of the lemma. 
\end{proof}

\section{Changing the initial configuration and the law on environments}\label{sec:couple2}

The previous section comes close to proving a hydrodynamic limit for the system of RWRE, but as noted above the system of RWRE studied in the previous section uses an initial configuration of particles that is concentrated on the ladder locations only and the environment $\w$ comes from the distribution $Q$ instead of $P$. In this section we remove these two difficulties. We will again use couplings of systems to be able to compare the behavior of two systems of RWRE. 

Some of the analysis in the current section requires a detailed analysis of the random environment. To this end we will introduce notation that will help simplify things. 
Recall that $\rho_x = \frac{1-\w_x}{\w_x}$ for any $x \in \Z$. Many formulas for quenched probabilities or expectations of interest involve sums of products of the $\rho_x$. To this end, we will let 
\begin{equation}\label{PiWRdef}
 \Pi_{i,j} = \prod_{x=i}^j \rho_x, \quad
W_j = \sum_{i=-\infty}^j \Pi_{i,j}, \quad \text{and}\quad R_i = \sum_{j=i}^\infty \Pi_{i,j}. 
\end{equation}
Since we are assuming that the environment $\w$ is i.i.d.\ with $E_P[\log \rho_0] < 0$, the infinite sums $W_j$ and $R_i$ converge almost surely. 
We will also need notation for the partial sums which converge to $W_j$ and $R_i$, and thus we will let
\begin{equation}\label{WRpartialdef}
 W_{\ell,j} = \sum_{i=\ell}^j \Pi_{i,j}, \quad \text{and} \quad R_{i,\ell} = \sum_{j=i}^\ell \Pi_{i,j}.
\end{equation}
Before proceeding to the analysis of the systems of RWRE in the rest of this section, we mention briefly two important places where this notation is useful. First of all, from the definition of the $g_\w(x)$ in \eqref{gdef} it is clear that 
\begin{equation}\label{gform}
 g_\w(x) = \frac{1}{\w_x}(1 + R_{x+1} ) = 1 + R_x + R_{x+1}, \quad \forall x \in \Z. 
\end{equation}
(For the second equality above we used that $\frac{1}{\w_x} = 1 + \rho_x$.)
Secondly, the quenched expectations for hitting times can be derived from the fact that $E_\w^x[T_{x+1}] = 1 + 2 W_x$ for all $x \in \Z$. 
In particular, we will use below that 
\begin{align}
 E_\w T_n = n + 2\sum_{j=0}^{n-1} W_j &= n + 2\sum_{j=0}^{n-1} \left( W_{0,j} + W_{-1}\Pi_{0,j} \right) \nonumber \\
&= n + 2\sum_{j=0}^{n-1} \sum_{i=0}^j \Pi_{i,j} + 2W_{-1}\sum_{j=0}^{n-1}\Pi_{0,j}. \label{EwTnform}
\end{align}

Another statistic of the environment that will be helpful in our analysis below is 
\[
 M_k = \sup \{ \Pi_{\nu_k, j} : \, \nu_k\leq j < \nu_{k+1} \} = \sup \left\{ e^{V_\w(x) - V_\w(\nu_k)}: \, x \in (\nu_k,\nu_{k+1}] \right\}, \qquad \forall k \in \Z. 
\]
That is, $M_k$ measures the maximal increase of the potential $V_\w$ of the environment between the ladder locations $\nu_k$ and $\nu_{k+1}$. 
Because $M_k$ only depends on the environment between successive ladder locations, it follows that the sequence $\{M_k\}_{k\in \Z}$ is i.i.d.\ under the measure $Q$. When the expected crossing time $\b_k$ between ladder locations is large, typically the main contribution comes from $M_k$ so that we can use $M_k$ at times as an i.i.d.\ approximation of the stationary sequence $\b_k$. Moreover, it will be important below that $M_k$ has similar tail decay as $\b_k$. In particular, 
it follows from \cite{iEVQ} that
there is a $C>0$ such that 
\begin{equation}\label{Mtail}
 Q(M_1 > x) \sim C x^{-\k}, \quad \text{as } x \ra\infty. 
\end{equation}

\subsection{Coupling different initial configurations}

We begin by showing that the system with particles started only at the ladder locations as in the previous section can be coupled with a system of particles with locally stationary initial configuration. 
Recall the definitions of the distributions $\mu_u^n$ and $\hat{\mu}_u^n$ on initial configurations in \eqref{rwlocstat} and \eqref{rwnukic} respectively. 
The following proposition allows us to compare the systems of independent RWRE with initial configurations $\mu_u^n$ and $\hat{\mu}_u^n$, respectively.

\begin{prop}\label{pr:Qiccoupling}
 Let Assumptions \ref{asmiid}--\ref{asmnl} hold with $\k \in (0,1)$ fixed, and let $u \in \mathcal{C}_0^+$. 
There exists a coupling $P_\w^{\mu_u^n(\w), \hat{\mu}_u^n(\w)}$ of two systems of particles $\{\chi_n\}_{n\geq 0}$ and $\{\hat\chi_n\}_{n\geq 0}$ such that the marginal distribution of 
$\{\chi_n\}_{n\geq 0}$ is $P_\w^{\mu_u^n(\w)}$ and the marginal distribution of $\{\hat{\chi}_n\}_{n\geq 0}$ is $P_\w^{\hat{\mu}_u^n}$, and such that for any $\phi \in \mathcal{C}_0(\R_+\times \R)$ and $T<\infty$,
\[
 \lim_{n\ra\infty} E_Q\left[ P_\w^{\mu_u^n(\w), \hat{\mu}_u^n(\w)} \left( \sup_{t\leq T} \left| \sum_{x\in \Z} ( \chi_{tn^{1/\k}}(x) - \hat\chi_{tn^{1/\k}}(x) ) \phi(t,x/n) \right| \geq \d n^{1/\k} \right) \right] = 0, \qquad \forall \d>0. 
\]
\end{prop}

The proof of Proposition \ref{pr:Qiccoupling} is most easily accomplished via yet another intermediate coupling. We will consider a system of independent RWRE $\bar{\chi}_n$ that is constructed by taking the initial configuration $\hat\chi_0$ and ``spreading out'' the particles so that there are particles started at every $x \in \Z$. In this way every particle in each system will be matched with a corresponding particle in the other system and the difficulty will be in showing that for most of the matched pairs of particles, the random walks can be coupled so that the distance of particles is not too far apart as the random walks evolve. The next step will be to give a coupling of the system of random walks $\bar\chi_n$ with the system $\chi_n$ that has the locally stationary initial configuration. In this step particles in the two systems will be matched to particles in the other system but starting at the same site $x\in\Z$. In this way, any two matched particles can be perfectly coupled for all time, but 
the 
difficulty arises in that the initial configurations are slightly different and so it needs to be shown that the number of unmatched particles between the two systems is not too large. 

In order to introduce the intermediate system $\bar\chi_n$ and describe the couplings of the systems of RWRE we first need to introduce some notation. 
For any $x,k \in \Z$ with $x < \nu_{k+1}$ let $b_{x,k} = b_{x,k}(\w)$ be defined by 
\[
 b_{x,k} = b_{x,k}(\w) = E_\w^{\nu_\k} \left[ \sum_{n=0}^{T_{\nu_{k+1}}-1} \ind{X_n = x} \right]. 
\]
The utility of the parameters $b_{x,k}$ will be that they will allow us to connect the parameters $\b_k$ and $g_\w(x)$ that are used in the definitions of the initial configurations $\mu_u^n(\w)$ and $\hat{\mu}_u^n(\w)$. In particular, it is easy to see that 
\begin{equation}\label{bxksums}
 \b_k = \sum_{x: \, x < \nu_{k+1}} b_{x,k}
\qquad\text{and}\qquad 
 g_\w(x) = \sum_{k: \, x < \nu_{k+1} } b_{x,k}.
\end{equation}
With this notation we can define the distribution that will be used for the initial configuration of the system $\bar\chi_0$. For $\w\in \Omega$ and $u\in \mathcal{C}_0$ fixed, let 
\[
 \pi_u^n(\w) = \bigotimes_{ x \in \Z } \text{Poisson}\left( \sum_{k:\, x < \nu_{k+1}} b_{x,k} u( \tfrac{\nu_k}{n} ) \right).
\]

Having introduced the necessary notation, we are now ready to approach the proof of Proposition \ref{pr:Qiccoupling}. 
To make the proof easiest to follow, we will state the intermediate couplings as two separate lemmas and then give the proofs of these lemmas. Obviously Proposition \ref{pr:Qiccoupling} follows easily from these two lemmas. 

\begin{lem}\label{lem:rwcouple1}
Let Assumptions \ref{asmiid}--\ref{asmnl} hold with $\k \in (0,1)$ fixed, and let $u \in \mathcal{C}_0^+$.
There exists a coupling $P_\w^{\hat{\mu}_u^n(\w), \pi_u^n(\w)}$ of two systems of particles $\{\hat\chi_n\}_{n\geq 0}$ and $\{\bar\chi_n \}_{n\geq 0}$ 
with marginal distributions $P_\w^{\hat{\mu}_u^n(\w)}$ and $P_\w^{\pi_u^n(\w)}$, respectively, and such that
for any $\phi \in \mathcal{C}_0(\R_+\times \R)$ and $T<\infty$,
\[
 \lim_{n\ra\infty} E_Q\left[ P_\w^{\hat{\mu}_u^n(\w), \pi_u^n(\w)} \left( \sup_{t\leq T} \left| \sum_{x\in \Z} ( \hat\chi_{tn^{1/\k}}(x) - \bar\chi_{tn^{1/\k}}(x) ) \phi(t,x/n) \right| \geq \d n^{1/\k} \right) \right] = 0, \qquad \forall \d>0. 
\]
\end{lem}
\begin{lem}\label{lem:rwcouple2}
Let Assumptions \ref{asmiid}--\ref{asmnl} hold with $\k \in (0,1)$ fixed, and let $u \in \mathcal{C}_0^+$.
There exists a coupling $P_\w^{\mu_u^n(\w), \pi_u^n(\w)}$ of two systems of particles $\{\chi_n\}_{n\geq 0}$ and $\{\bar\chi_n \}_{n\geq 0}$ 
with marginal distributions $P_\w^{\mu_u^n(\w)}$ and $P_\w^{\pi_u^n(\w)}$, respectively, and such that
for any $\phi \in \mathcal{C}_0(\R_+\times \R)$ and $T<\infty$,
\[
 \lim_{n\ra\infty} E_Q\left[ P_\w^{\mu_u^n(\w), \pi_u^n(\w)} \left( \sup_{t\leq T} \left| \sum_{x\in \Z} ( \chi_{tn^{1/\k}}(x) - \bar\chi_{tn^{1/\k}}(x) ) \phi(t,x/n) \right| \geq \d n^{1/\k} \right) \right] = 0, \qquad \forall \d>0. 
\]
\end{lem}

\begin{proof}[Proof of Lemma \ref{lem:rwcouple1}]
 We begin by coupling the initial configurations of particles $\hat\chi_0$ and $\bar\chi_0$. 
Given the environment $\w$, let $\{B_{x,k}^n\}_{(x,k):\, x < \nu_{k+1}}$ be a family of independent Poisson random variables with $B_{x,k}^n \sim \text{Poisson}(b_{x,k} u(\frac{\nu_k}{n}))$. 
Using the first equality in \eqref{bxksums}, it is easy to see that we can construct the initial configurations $\hat\chi_0$ and $\bar\chi_0$ by letting
\[
 \hat\chi_0(\nu_k) = \sum_{x: \, x < \nu_{k+1}} B_{x,k}^n, 
\quad\text{and}\quad
\bar\chi_0(x) = \sum_{k: \, x < \nu_{k+1} } B_{x,k}^n. 
\]
Given this coupling of the initial configurations, for each pair $(x,k)$ with $x<\nu_{k+1}$ we can couple $B_{x,k}^n$ pairs of particles started at $x$ and $\nu_k$. That is, for each such pair $(x,k)$ and any $j\leq B_{x,k}^n$ we will let $\bar X_\cdot^{(x,k),j} $ and $\hat X_\cdot^{(x,k),j} $ be two random walks started at $x$ and $\nu_k$, respectively. 
We will couple these walks in the following simple manner. If $x \leq \nu_k$ then the walk $\bar X_\cdot^{(x,k),j}$ will evolve independently until reaching $\nu_k$, at which point it will trace the path of the walk $\hat X_\cdot^{(x,k),j}$. Conversely, if $\nu_k < x$ then the walk $\hat X_\cdot^{(x,k),j}$ will evolve independently until reaching $x$, at which point it will trace the path of the walk $\bar X_\cdot^{(x,k),j}$. 
Note that it is clear from this coupling procedure that 
\begin{equation}\label{hXbXcouple}
\begin{split}
 \sup_n \left|\hat X_n^{(x,k),j} - \bar X_n^{(x,k),j} \right|
&\leq
\begin{cases}
 \inf\left\{i\geq 0: \bar X_i^{(x,k),j} = \nu_k \right\} & \text{if } x \leq \nu_k \\
 \inf\left\{i\geq 0: \hat X_i^{(x,k),j} = x  \right\} & \text{if }  x > \nu_k
\end{cases}\\
&\overset{\text{Law}}{=}
\begin{cases}
 T_{\nu_k}^x &  \text{if } x \leq \nu_k \\
 T_x^{\nu_k} &  \text{if } x > \nu_k, 
\end{cases}
\end{split}
\end{equation}
where in the last line we use the notation $T^x_y$ for the first hitting time of a site $y \in \Z$ by a RWRE started at $x \in \Z$.

Having constructed the coupling of the systems of RWRE, we have that 
\[
 \sum_{x \in \Z} (\hat\chi_{tn^{1/\k}}(x) - \bar\chi_{tn^{1/\k}}(x) ) \phi(t,\tfrac{x}{n}) = 
\sum_{|k|\leq Ln} \sum_{x<\nu_{k+1}} \sum_{j=1}^{B_{x,k}^n} \left( \phi\left(t,\frac{\hat X_{tn^{1/\k}}^{(x,k),j}}{n} \right) - \phi\left(t,\frac{\bar X_{tn^{1/\k}}^{(x,k),j}}{n} \right) \right),
\]
where on the right side we can restrict the first sum to $|k|\leq Ln$ since $\supp u \subset [-L,L]$ implies that $B_{x,k}^n = 0$ if $|k|>Ln$.
Next, we claim that in the hydrodynamic limit scaling, we can ignore all the coupled pairs $(x,k)$ with $k$ such that $M_k \leq n^{1/\k}/\log n$. To see this, note that it was shown in the proof of Proposition 4 in \cite{estzWQL} that 
\begin{equation} \label{betasmallsum}
 \lim_{n\ra\infty} \frac{1}{n^{1/\k}} \sum_{|k|\leq Ln} \b_k \ind{M_k \leq \frac{n^{1/\k}}{\log n} } = 0, \quad \text{in $Q$-probability}. 
\end{equation}
Since $B_{x,k}^n$ is Poisson with mean at most $\|u\|_\infty b_{x,k}$, it follows easily from \eqref{bxksums} and \eqref{betasmallsum} that
\begin{equation}\label{Bsmallsum}
 E_Q\left[ P_\w^{\hat\mu_u^n(\w),\pi_u^n(\w)} \left( \sum_{|k|\leq Ln, \, M_k \leq \frac{n^{1/\k}}{\log n}} \sum_{x < \nu_{k+1}} B_{x,k}^n \geq \d n^{1/\k} \right) \right] = 0, \quad \forall \d>0. 
\end{equation}
Therefore, to finish the proof of the lemma, it will be enough to show that 
\begin{equation}\label{bigMcoupling}
 \lim_{n\ra\infty} E_Q\left[ P_\w^{\hat\mu_u^n,\pi_u^n} \left( \sum_{\substack{|k|\leq Ln\\M_k > \frac{n^{1/\k}}{\log n}}} \sum_{x < \nu_{k+1}} \sum_{j=1}^{B_{x,k}^n} \sup_{t\leq T} \left|   \phi\left(t,\frac{\hat X_{tn^{1/\k}}^{(x,k),j}}{n} \right) - \phi\left(t,\frac{\bar X_{tn^{1/\k}}^{(x,k),j}}{n} \right)  \right| \geq \d n^{1/\k} \right) \right] = 0, 
\end{equation}
for all $\d>0$. 
We will prove \eqref{bigMcoupling} by showing the following. 
\begin{itemize}
 \item If $M_k \geq \frac{n^{1/\k}}{\log n}$ and $x$ is close enough to $\nu_k$, then we will be able to couple the two walks so that the difference is not very large. 
 \item There are relatively few coupled particles started at pairs $x<\nu_{k+1}$ that are not close enough to have a good coupling. 
\end{itemize}
In order to make this precise, we will need to identify an interval around each $\nu_k$ where the coupling works well. In particular, for all $k\in\Z$ with $M_k\geq \frac{n^{1/\k}}{\log n}$ we will identify intervals $[a_k^n, c_k^n]$ around the respective ladder locations $\nu_k$ with the following properties. 
\begin{equation}\label{acprop1}
 \lim_{n\ra\infty} \frac{1}{n^{1/\k}} \sum_{|k|\leq Ln} \sum_{x < a_k^n} b_{x,k} \ind{M_k > \frac{n^{1/\k}}{\log n}} = 0 , \quad \text{in $Q$-probability}, 
\end{equation}
\begin{equation}\label{acprop2}
 \lim_{n\ra\infty} \frac{1}{n^{1/\k}} \sum_{|k|\leq Ln} \sum_{c_k^n < x < \nu_{k+1}} b_{x,k} \ind{M_k > \frac{n^{1/\k}}{\log n}} = 0 , \quad \text{in $Q$-probability}, 
\end{equation}
and for some $\e>0$, 
\begin{equation}\label{acprop3}
 \lim_{n\ra\infty} Q\left( \exists |k|\leq Ln: M_k \geq \frac{n^{1/\k}}{\log n}, \,  E_\w^{a_k^n}[ T_{c_k^n} ] > n^{1-2\e} \right) = 0. 
\end{equation}
Before defining these intervals properly, we first show how properties \eqref{acprop1}-\eqref{acprop3} allow us to prove \eqref{bigMcoupling}. 
First of all, in a similar manner as \eqref{betasmallsum} was used to prove \eqref{Bsmallsum} we can use \eqref{acprop1} and \eqref{acprop2} to show 
\[
 \lim_{n\ra\infty} E_Q\left[ P_\w^{\hat\mu_u^n,\pi_u^n} \left( \sum_{\substack{|k|\leq Ln \\ M_k > \frac{n^{1/\k}}{\log n}}} \sum_{\substack{x < \nu_{k+1} \\ x \notin [a_k^n,c_k^n]}} B_{x,k}^n \geq \d n^{1/\k}   \right) \right] = 0, \quad \forall \d > 0. 
\]
That is, the total number of pairs of particles corresponding to $(x,k)$ with $x \notin [a_k^n,c_k^n]$ is negligible in the hydrodynamic limit. Thus, to prove \eqref{bigMcoupling} we need only to control the coupling of the walks when $x \in [a_k^n,c_k^n]$ and this will be accomplished using \eqref{acprop3}. 
To this end, let $\{\s_{x,k}^j \}_{x,k,j}$ be a family of independent random variables for $k$ such that $M_k > n^{1/\k}/\log n$, $x \in [a_k^n,c_k^n]$ and $j\leq B_{x,k}^n$ with $\s_{x,k}^j$ having the same distribution as $T^{a_k^n}_{c_k^n}$ for every $j\leq B_{x,k}^n$. Due to \eqref{hXbXcouple} we can couple the $\s_{x,k}^j$ so that $\sup_n |\hat X_n^{(x,k),j} - \bar X_n^{(x,k),j}| \leq \s_{x,k}^j$. 
Therefore, 
\[
\sum_{\substack{|k|\leq Ln\\M_k > \frac{n^{1/\k}}{\log n}}} \sum_{x \in [a_k^n,c_k^n] } \sum_{j=1}^{B_{x,k}^n}
\sup_{t\leq T} \left|   \phi\left(t,\frac{\hat X_{tn^{1/\k}}^{(x,k),j}}{n} \right) - \phi\left(t,\frac{\bar X_{tn^{1/\k}}^{(x,k),j}}{n} \right)  \right|
\leq 
\sum_{\substack{|k|\leq Ln\\M_k > \frac{n^{1/\k}}{\log n}}} \sum_{x \in [a_k^n,c_k^n] } \sum_{j=1}^{B_{x,k}^n}
\Delta\left( \phi; \frac{\s_{x,k}^j}{n}  \right),
\]
from which we can conclude that 
\begin{align}
 & P_\w^{\hat\mu_u^n,\pi_u^n} \left( \sum_{\substack{|k|\leq Ln\\M_k > \frac{n^{1/\k}}{\log n}}} \sum_{x \in [a_k^n,c_k^n] } \sum_{j=1}^{B_{x,k}^n} \sup_{t\leq T} \left|   \phi\left(t,\frac{\hat X_{tn^{1/\k}}^{(x,k),j}}{n} \right) - \phi\left(t,\frac{\bar X_{tn^{1/\k}}^{(x,k),j}}{n} \right)  \right| \geq \d n^{1/\k} \right) \nonumber \\
&\leq P_\w^{\hat\mu_u^n,\pi_u^n} \left( \sum_{\substack{|k|\leq Ln\\M_k > \frac{n^{1/\k}}{\log n}}} \sum_{x \in [a_k^n,c_k^n] } B_{x,k}^n \Delta(\phi;n^{-\e})  \geq  \frac{\d n^{1/\k}}{2} \right) \nonumber \\
&\qquad + P_\w^{\hat\mu_u^n,\pi_u^n} \left( \sum_{\substack{|k|\leq Ln \\ M_k > \frac{n^{1/\k}}{\log n} }} \sum_{x \in [a_k^n,c_k^n]} \sum_{j=1}^{B_{x,k}^n} \ind{\s_{x,k}^j > n^{1-\e} } \geq \frac{\d n^{1/\k}}{4\|\phi\|_\infty} \right) \nonumber \\
&\leq \frac{2 \Delta(\phi;n^{-\e}) \|u\|_\infty}{\d n^{1/\k}} \sum_{|k|\leq Ln} \b_k   + \frac{4\|\phi\|_\infty \|u\|_\infty }{\d n^{1/\k}}\sum_{\substack{|k|\leq Ln \\ M_k > \frac{n^{1/\k}}{\log n} }} \sum_{x \in [a_k^n,c_k^n]} b_{x,k} P_\w\left( T^{a_k^n}_{c_k^n} \geq n^{1-\e} \right). \label{hXbXcouplebound}
\end{align}
For $k$ with $M_k > n^{1/\k}/\log n$, we have that Chebychev's inequality and \eqref{bxksums} imply that
\[
 \sum_{x \in [a_k^n,c_k^n]} b_{x,k} P_\w\left(T^{a_k^n}_{c_k^n} \geq n^{1-\e} \right)
\leq \frac{1}{n^{1-\e}} \sum_{x \in [a_k^n,c_k^n]} b_{x,k} E_\w^{a_k^n}[ T_{c_k^n} ] 
\leq \frac{E_\w^{a_k^n}[ T_{c_k^n} ]}{n^{1-\e}} \b_k. 
\]
Therefore, we can conclude that \eqref{hXbXcouplebound} is bounded above by 
\begin{align*}
\left( \frac{2}{\d n^{1/\k}} \sum_{|k|\leq Ln} \b_k \right) \left\{ \Delta(\phi;n^{-\e}) \|u\|_\infty + \frac{2\|\phi\|_\infty \|u\|_\infty }{n^{1-\e}} \max_{\substack{|k|\leq Ln\\ M_k > \frac{n^{1/\k}}{\log n}}} E_\w^{a_k^n}[T_{c_k^n}] \right\}.
\end{align*}
If we choose $\e>0$ as in \eqref{acprop3}, then Corollary \ref{cor:Wnweakconv} together with \eqref{acprop3} implies that this converges to 0 in $Q$-probability, which in turn implies that \eqref{bigMcoupling} holds. 

It remains now to show that we can choose the intervals $[a_k^n,c_k^n]$ containing $\nu_k$ when $M_k > n^{1/\k}/\log n$ so that \eqref{acprop1}-\eqref{acprop3} hold. 
First, fix a constant $K > 1/(\k \bar\nu E_P[-\log \rho_0] )$, and let $a_k^n = \nu_{k - \lceil K \log n \rceil}$. With this definition of $a_k^n$, property \eqref{acprop1} follows from \cite[Proposition 3]{estzWQL}. 
To define the right endpoint $c_k^n$, fix $\gamma < 1$ and for $k \in \Z$ with $M_k \geq \frac{n^{1/\k}}{\log n} \geq n^{\gamma}$ (for $n$ large enough) define
\[
 c_k^n = \inf\{ j> \nu_k: \Pi_{\nu_k,j} \geq n^\gamma \}. 
\]
To verify \eqref{acprop2} for this choice of $c_k^n$, first note that 
\begin{align}
& Q\left( \sum_{|k|\leq Ln} \sum_{c_k^n < x < \nu_{k+1}} b_{x,k} \ind{M_k > \frac{n^{1/\k}}{\log n}} \geq \d n^{1/\k} \right) \nonumber \\
&\quad \leq Q\left( \sum_{|k|\leq Ln} \ind{M_k >  \frac{n^{1/\k}}{\log n}} \geq \d \log n \right) + Q\left( \exists |k| \leq Ln: \!\!\! \sum_{c_k^n < x < \nu_{k+1}} \!\!\! b_{x,k} \ind{M_k > \frac{n^{1/\k}}{\log n}} \geq \frac{n^{1/\k}}{\log n} \right) \nonumber \\
&\quad \leq \frac{2Ln+1}{\d \log n} Q\left( M_0 > \frac{n^{1/\k}}{\log n} \right) + (2Ln + 1) Q\left(  \sum_{c_0^n < x < \nu_{1}} \!\!\! b_{x,0} \geq \frac{n^{1/\k}}{\log n}, \, M_0 > \frac{n^{1/\k}}{\log n} \right). \label{acprop2ub1} 
\end{align}
The tail asymptotics of $M_0$ in \eqref{Mtail} imply that the first term in \eqref{acprop2ub1} is $\bigo((\log n)^{\k-1}) = o(1)$.
To control the probability in the second term of \eqref{acprop2ub1}, 
let $T_x^+ = \inf\{ n\geq 1: X_n = x\}$ denote the first return time of the RWRE to $x\in \Z$ and note that for $0\leq x < \nu_1$
\begin{align*}
 b_{x,0} = E_\w^x\left[ \sum_{n=0}^{T_{\nu_1}-1} \ind{X_n = x} \right] = \frac{1}{P_\w^x(T_{\nu_1} < T_x^+ )} 
&= \frac{1+R_{x+1,\nu_1-1}}{\w_x} \\
&= 1 + R_{x,\nu_1-1} + R_{x+1,\nu_1-1} \\
&\leq 1 + 2 \nu_1\left( \max_{x\leq j< \nu_1} \Pi_{x,j} \right).
\end{align*}
(The third equality is a standard calculation for hitting probabilities for reversible Markov chains, and can be deduced for instance from \cite[equation (2.1.4)]{zRWRE}.)
Using this upper bound for $b_{x,k}$ we can conclude that (for $n$ large enough)
\begin{align*}
& Q\left(  \sum_{c_0^n < x < \nu_{1}} \!\!\! b_{x,0} \geq \frac{n^{1/\k}}{\log n}, \, M_0 > \frac{n^{1/\k}}{\log n} \right) \\
&\qquad \leq Q(\nu_1 \geq (\log n)^2) + Q\left( \max_{c_0^n < i \leq j < \nu_1} \Pi_{i,j} \geq \frac{n^{1/\k}}{4(\log n)^5}, \, M_0 \geq \frac{n^{1/\k}}{\log n} \right).
\end{align*}
Since $\nu_1$ has exponential tails, in order to show that $\eqref{acprop2ub1}$ vanishes as $n\ra\infty$ we need only to show that the second probability on the right above is $o(n^{-1})$. 
This will be accomplished using some estimates from \cite{p1LSL2}. Let $J = \max\{ j \in [1,\nu_1]: \Pi_{0,j-1} = M_0 \}$ be the (last) location in $(0,\nu_1]$ where the potential achieves its maximum in that interval. Then, define $M^-$ and $M^+$ by 
\[
 M^- = \min\{ \Pi_{i,j}: \, 0 < i \leq j < J \} \wedge 1, \quad \text{and}\quad 
 M^+ = \max\{ \Pi_{i,j}: \, J < i \leq j < \nu \} \vee 1. 
\]
That is $M^-$ controls the amount the potential can decrease before $J$ and $M^+$ controls the amount the potential can increase after $J$. 
With this notation it is easy to see that 
\[
 \Pi_{i,j} \leq 
\begin{cases}
 \frac{M_0}{n^{\gamma} M^-} & \text{if } c_0^n < i \leq J, \, j < \nu_1\\
 M^+ & \text{if } J \leq i \leq j < \nu_1,
\end{cases}
\]
and therefore
\begin{equation}\label{Mpmbounds}
\begin{split}
 Q\left( \max_{c_0^n < i \leq j < \nu_1} \Pi_{i,j} \geq \frac{n^{1/\k}}{4(\log n)^5}, \, M_0 \geq \frac{n^{1/\k}}{\log n} \right)
&\leq Q\left( M^+ \geq \frac{n^{1/\k}}{4(\log n)^5}, \, M_0 \geq \frac{n^{1/\k}}{\log n} \right)\\
&\qquad + Q\left( M^- \leq n^{-\gamma/2}, \, M_0 \geq \frac{n^{1/\k}}{\log n} \right)\\
&\qquad  + Q\left( M_0 \geq \frac{n^{1/\k+\gamma/2}}{4(\log n)^5} \right).
\end{split}
\end{equation}
It follows from \cite[Lemma 4.1]{p1LSL2} that the first two probabilities on the right are $o(n^{-1})$, and the tail decay of $M_0$ in \eqref{Mtail} implies that the third probability is also $o(n^{-1})$. 

We have shown that our choice of $a_k^n$ and $c_k^n$ satisfy \eqref{acprop1} and \eqref{acprop2}, and it only remains to show that \eqref{acprop3} also holds. To this end, it is clearly enough to show 
\begin{equation}\label{acprop3sc}
 Q\left( M_0 \geq \frac{n^{1/\k}}{\log n}, E_\w^{a_0^n}[T_0] \geq \frac{n^{1-2\e}}{2} \right) + Q\left( M_0 \geq \frac{n^{1/\k}}{\log n}, E_\w^{0}[T_{c_0^n}] \geq \frac{n^{1-2\e}}{2} \right) = o(n^{-1}), 
\end{equation}
for some $\e>0$. 
For the first probability in \eqref{acprop3sc}, note that 
\[
 Q\left( M_0 \geq \frac{n^{1/\k}}{\log n}, E_\w^{a_0^n}[T_0] \geq \frac{n^{1-2\e}}{2} \right)
= Q\left( M_0 \geq \frac{n^{1/\k}}{\log n}\right) Q\left(E_\w^{a_0^n}[T_0] \geq \frac{n^{1-2\e}}{2} \right),
\]
since $M_0$ is independent of the environment to the left of $0$ under the measure $Q$. Then, recalling that $a_0^n = \nu_{-\lceil K\log n \rceil}$ we have that 
\[
 Q\left(E_\w^{a_0^n}[T_0] \geq \frac{n^{1-2\e}}{2} \right) 
= Q\left( \sum_{k=-\lceil K\log n\rceil}^{-1} \!\!\! \b_k \geq \frac{n^{1-2\e}}{2} \right)
\leq \lceil K \log n \rceil Q\left( \b_0 \geq \frac{n^{1-2\e}}{2 \lceil K \log n \rceil } \right).
\]
Therefore, the tail decay of $\b_0$ and $M_0$ imply that the first probability in \eqref{acprop3sc} is 
$o(n^{-1})$ as long as $\e < 1/2$. 
To control the second term in \eqref{acprop3sc}, note that \eqref{EwTnform} implies that 
\[
 E_\w[T_{c_0^n}] = c_0^n + 2 \sum_{j=0}^{c_0^n-1} \sum_{i=0}^j \Pi_{i,j} + 2 W_{-1} \sum_{j=0}^{c_0^n-1} \Pi_{0,j} 
\leq \nu_1 + 2 \nu_1^2 n^\gamma + 2 W_{-1} \nu_1 n^\gamma,
\]
where the last inequality follows from the fact that $\Pi_{i,j} \leq \Pi_{0,j} \leq n^\gamma$ for any $0 \leq i \leq j < c_0^n < \nu_1$. 
Thus, since $\nu_1$ and $W_{-1}$ have exponential tails under the measure $Q$ we have that the second probability in \eqref{acprop3sc} is bounded above by 
\[
 Q\left( \nu_1 + 2 \nu_1^2 n^\gamma + 2 W_{-1} \nu_1 n^\gamma \geq \frac{n^{1-2\e}}{2} \right) = o(n^{-1}), \quad \text{if } 0 < \e < \frac{1-\gamma}{2}. 
\]
This completes the proof of \eqref{acprop3sc}, which in turn completes the proof of the lemma. 
\end{proof}

\begin{proof}[Proof of Lemma \ref{lem:rwcouple2}]
The key to the proof of Lemma \ref{lem:rwcouple2} is to show that the initial configurations $\chi_0$ and $\bar\chi_0$ can be coupled under the measure $E_Q[P_\w^{\mu_u^n(\w),\pi_u^n(\w)}(\cdot)]$ so that there are typically much less than $n^{1/\k}$ particles that are not coupled with a particle in the other system. 
That is, since we can perfectly couple for all time the first $\min\{ \chi_0(x), \bar\chi_0(x)\}$ particles at each site $x \in \Z$, we need only to show that 
\begin{equation}\label{bchiuncoupled}
 \lim_{n\ra\infty} E_Q\left[ P_\w^{\mu_u^n(\w),\pi_u^n(\w)} \left( \sum_{x} | \chi_0(x) - \bar\chi_0(x) |  \geq \d n^{1/\k} \right) \right] = 0, \quad \forall \d>0. 
\end{equation}
For any $\th_1,\th_2>0$, it is a standard property of Poisson random variables that one can construct a coupling of random variables $Y_1 \sim \text{Poisson}(\th_1)$ and $Y_2 \sim \text{Poisson}(\th_2)$ so that $|Y_1-Y_2| \sim \text{Poisson}(|\th_1-\th_2|)$. 
Indeed, if $\th_1 < \th_2$ then we can let $Z_1$ and $Z_2$ be independent Poisson random variables with parameters $\th_1$ and $\th_2-\th_1$, respectively, and then let $Y_1 = Z_1$ and $Y_2 = Z_1+Z_2$. 
Therefore,  since $g_\w(x) = \sum_{k: x < \nu_{k+1}} b_{x,k}$
we can construct a coupling of the initial configurations so that $\{|\chi_0(x) - \bar\chi_0(x)|\}_{x\in \Z}$ are independent with 
\[
 | \chi_0(x) - \bar\chi_0(x) | \sim 
\text{Poisson}\left( \left| \sum_{k: x < \nu_{k+1}} b_{x,k} \left( u(\tfrac{x}{n}) - u(\tfrac{\nu_k}{n}) \right) \right| \right). 
\]
From this it follows easily that 
\begin{align}
& P_\w^{\mu_u^n(\w),\pi_u^n(\w)} \left( \sum_{x} | \chi_0(x) - \bar\chi_0(x) |  \geq \d n^{1/\k} \right) \nonumber \\
&\qquad \leq \frac{1}{\d n^{1/\k}} \sum_{k \in \Z} \sum_{x<\nu_{k+1}} b_{x,k} \left| u( \tfrac{x}{n}) - u(\tfrac{\nu_k}{n}) \right| \nonumber \\
&\qquad \leq \frac{1}{\d n^{1/\k}} \sum_{|k|\leq Ln} \sum_{x<\nu_{k+1}} b_{x,k} \Delta\left(u; \tfrac{|x-\nu_k|}{n} \right) + \frac{\|u\|_\infty}{\d n^{1/\k}} \sum_{k > L n} \sum_{x \leq Ln} b_{x,k}, \label{uncoupledub}
\end{align}
where in the second inequality we use $\supp u \subset [-L,L]$. 
To control the second sum in \eqref{uncoupledub}, note that the definition of $b_{x,k}$ implies that 
\begin{equation}\label{bxksumsep}
\begin{split}
 \sum_{k > L n} \sum_{x \leq Ln} b_{x,k} &= E_\w^{\nu_{\fl{Ln}+1}}\left[ \sum_{m=0}^\infty \ind{X_m \leq Ln} \right] \\
& \leq E_\w^{\nu_{\fl{Ln}+1}}\left[ \sum_{m=0}^\infty \ind{X_m \leq \nu_{\fl{Ln}} } \right] 
\overset{\text{Law}}{=} E_\w\left[ \sum_{m=0}^\infty \ind{X_m \leq \nu_{-1}} \right],
\end{split}
\end{equation}
where the last equality indicated equality in law under the measure $Q$. 
Since this last quenched expectation is $Q$-a.s.\ finite, it follows that the second term in \eqref{uncoupledub} converges to $0$ in $Q$-probability. 
We wish to show that the first term in \eqref{uncoupledub} also converges to 0 in $Q$-probability. 
To this end, note that for any $\e \in (0,1)$ the first term in \eqref{uncoupledub} is bounded above by 
\begin{align*}
& 
\frac{\Delta\left(u; n^{-\e} \right) }{\d n^{1/\k}} \sum_{|k|\leq Ln} \sum_{\substack{x<\nu_{k+1}\\|x-\nu_k| \leq n^{1-\e}}} \!\!\! b_{x,k} 
+ \frac{2\|u\|_\infty }{\d n^{1/\k}} \sum_{|k|\leq Ln} \sum_{\substack{x<\nu_{k+1}\\|x-\nu_k| > n^{1-\e}}} \!\!\! b_{x,k} \\
&\qquad \leq \frac{\Delta\left(u; n^{-\e} \right) }{\d n^{1/\k}} \sum_{|k|\leq Ln} \b_k + \frac{2\|u\|_\infty }{\d n^{1/\k}} \sum_{|k|\leq Ln} \sum_{\substack{x<\nu_{k+1}\\|x-\nu_k| > n^{1-\e}}} \!\!\! b_{x,k}.
\end{align*}
Since $\Delta(u;n^{-\e}) \ra 0$, Corollary \ref{cor:Wnweakconv} implies that the first term on the right converges to 0 in $Q$-probability. 
To show that the second term also converges to $0$ in $Q$-probability, we further decompose it as
\begin{align}
 & \frac{2\|u\|_\infty }{\d n^{1/\k}} \sum_{\substack{|k|\leq Ln \\ M_k \leq \frac{n^{1/\k}}{\log n}}} \sum_{\substack{x<\nu_{k+1}\\|x-\nu_k| > n^{1-\e}}} \!\!\! b_{x,k}
+ \frac{2\|u\|_\infty }{\d n^{1/\k}}  \sum_{\substack{|k|\leq Ln \\ M_k > \frac{n^{1/\k}}{\log n}}} \sum_{\substack{x<\nu_{k+1}\\|x-\nu_k| > n^{1-\e}}} \!\!\! b_{x,k} \nonumber \\
&\qquad \leq \frac{2\|u\|_\infty }{\d n^{1/\k}} \sum_{|k|\leq Ln} \b_k \ind{M_k \leq \frac{n^{1/\k}}{\log n}}
+ \frac{2\|u\|_\infty }{\d n^{1/\k}}  \sum_{\substack{|k|\leq Ln \\ M_k > \frac{n^{1/\k}}{\log n}}} \sum_{\substack{x<\nu_{k+1}\\|x-\nu_k| > n^{1-\e}}} \!\!\! b_{x,k} . \label{sumsfar}
\end{align}
The first term in \eqref{sumsfar} converges to 0 in $Q$-probability by \eqref{betasmallsum}. 
For the second term in \eqref{sumsfar}, recall the definition of the intervals $[a_k^n,c_k^n]$ containing $\nu_k$ when $M_k > n^{1/\k}/\log n$ that were given in the proof of Lemma \ref{lem:rwcouple1} above. Since $a_k^n = \nu_{k - \lceil K \log n \rceil}$ and $c_k^n \in (\nu_k,\nu_{k+1})$ we have that 
$|x-\nu_k| > n^{1-\e}$ implies that $x \notin [a_k^n,c_k^n]$ unless
$\nu_{k+1} - \nu_{k-\lceil K \log n \rceil} > n^{1-\e}$. 
For $n$ large enough we have 
\[
 Q(\exists |k|\leq Ln : \, \nu_{k+1} - \nu_{k-\lceil K \log n \rceil} > n^{1-\e} )
\leq 3 L n \, Q\left( \nu_1 > \frac{n^{1-\e}}{2 K \log n } \right),
\]
and since $\nu_1$ has exponential tails the right side vanishes as $n\ra\infty$. 
It follows from this, together with \eqref{acprop1} and \eqref{acprop2}, that the second term in \eqref{sumsfar} converges to 0 in $Q$-probability. 	
Combined with our above estimates, this completes the proof that \eqref{uncoupledub} converges to 0 in $Q$-probability, and this is enough to prove that \eqref{bchiuncoupled} holds. 
\end{proof}

\subsection{Changing the distribution on the environment}

The results proved so far will be enough to prove a hydrodynamic limit of the form in Theorem \ref{th:RWREhdl} for a system of independent RWRE with locally stationary initial configurations, but where the environment $\w$ has distribution $Q$ instead of $P$ as in the statement of Theorem \ref{th:RWREhdl}. In this subsection, we complete the final comparison that will be needed for the proof of Theorem \ref{th:RWREhdl} by giving a coupling of two systems of RWRE in different environments $\w$ and $\tilde{\w}$ where the environments $\w$ and $\tilde{\w}$ have distribution $P$ and $Q$ respectively. 

We begin by giving a coupling of the environments $\w$ and $\tw$. 
Recall the definition of the blocks $B_k$ of the environment between ladder locations from \eqref{blockdef} and that the blocks $\{B_k\}_{k\neq 0}$ are i.i.d.\ under the measure $P$, the blocks $\{B_k\}_{k\in\Z}$ are i.i.d.\ under $Q$, and $B_1$ has the same distribution under both $P$ and $Q$. Therefore, given an environment $\w$ with distribution $P$ we can construct an environment $\tilde\w$ with distribution $Q$ by simply removing the block $B_0$ from the environment $\w$. 
That is, 
\begin{equation}\label{PQcouple}
 \tilde\w_x = 
\begin{cases}
 \w_{\nu_1+x} &  \text{if } x\geq 0 \\
 \w_{\nu_0+x} &  \text{if } x< 0. 
\end{cases}
\end{equation}

Having given this coupling of environments with distribution $P$ and $Q$, respectively, 
we are now ready to state the following Proposition which shows that changing the distribution on environments from $Q$ to $P$ does not affect the hydrodynamic limit if we use the sequence of locally stationary initial configurations $\mu_u^n(\w)$. 

\begin{prop}\label{pr:PQsystemcouple}
 Let Assumptions \ref{asmiid}--\ref{asmnl} hold with $\k \in (0,1)$ fixed, and let $u \in \mathcal{C}_0$.
For any two environments $\w,\tilde\w \in \Omega$ there exists a coupling $P_{\w,\tilde\w}^{\mu_u^n(\w),\mu_u^n(\tilde\w)}$ of two systems of random walks $\{\chi_n\}_{n\geq 0}$ and $\{\tilde\chi_n\}_{n\geq 0}$ with marginal distributions $P_\w^{\mu_u^n(\w)}$ and $P_{\tilde\w}^{\mu_u^n(\tilde\w)}$, respectively, and such that
if $\mathcal{P}$ is the coupling of pairs of environments $(\w,\tilde\w) \in \Omega^2$ as given in \eqref{PQcouple} then 
\[
 \lim_{n\ra\infty} E_{\mathcal{P}} \left[ P_{\w,\tilde\w}^{\mu_u^n(\w),\mu_u^n(\tilde\w)} \left( \sup_{t\leq T} \left| \sum_{k\in \Z} ( \chi_{tn^{1/\k}}(k) - \tilde\chi_{tn^{1/\k}}(k) ) \phi(t,k/n) \right| \geq \d n^{1/\k} \right) \right] = 0, \qquad \forall \d>0,
\]
for any $\phi\in \mathcal{C}_0(\R_+\times \R)$ and $T<\infty$. 
\end{prop}

In preparation for the proof of Proposition \ref{pr:PQsystemcouple}, we will first prove the following Lemma which will be used to show that the distributions $\mu_u^n(\w)$ and $\mu_u^n(\tw)$ on initial configurations are very similar when $\w$ and $\tw$ are coupled as in \eqref{PQcouple}.

\begin{lem}\label{PQcouplediff}
 Under the coupling $(\w,\tilde\w)$ of the measures $P$ and $Q$ given in \eqref{PQcouple}, with probability one, 
\[
 \sum_{x \leq -1} | g_\w(\nu_0+x) - g_{\tilde\w}(x) | < \infty. 
\]
\end{lem}
\begin{proof}
Recall the definitions of the random variables $\Pi_{i,j}$, $W_i$, $R_i$, and $R_{i,\ell}$ as given in \eqref{PiWRdef} and \eqref{WRpartialdef}. 
We will use the notation $\tilde{\Pi}_{i,j}$, $\tilde{W}_i$, $\tilde{R}_i$ and $\tilde{R}_{i,\ell}$ for the random variables as defined in \eqref{PiWRdef} and \eqref{WRpartialdef} but corresponding to the environment $\tw$ instead of $\w$. 
Then, recalling \eqref{gform}, this notation gives that 
\[
g_{\tilde\w}(x) - g_\w(\nu_0+x) = \tilde{R}_x + \tilde{R}_{x+1} - R_{\nu_0 + x} - R_{\nu_0+x+1}. 
\]
Next, note that 
\[
 R_{\nu_0 + x} = R_{\nu_0 + x,\nu_0-1} + \Pi_{\nu_0+x,\nu_0-1} R_{\nu_0}, \quad\text{for } x<0. 
\]
A similar decomposition is true for $\tilde{R}_x$, but using the coupling of $\w$ and $\tw$ given by \eqref{PQcouple} we get that 
\[
 \tilde{R}_x = \tilde{R}_{x,-1} + \tilde{\Pi}_{x,-1}\tilde{R}_0
= R_{\nu_0 + x,\nu_0-1} + \Pi_{\nu_0+x,\nu_0-1} R_{\nu_1}, \quad \text{for } x < 0. 
\]
Therefore, we obtain that 
\[
 g_{\tilde\w}(x) - g_\w(\nu_0+x) = \left( \Pi_{\nu_0+x,\nu_0-1} + \Pi_{\nu_0+x+1,\nu_0-1} \right)\left( R_{\nu_1} - R_{\nu_0} \right).
\]
Note that this shows that the sign of $g_{\tilde\w}(x) - g_\w(\nu_0+x)$ does not depend on $x$ and is the same as the sign of $R_{\nu_1} - R_{\nu_0}$. Therefore, we can conclude that 
\[
 \sum_{x<0} \left| g_{\tilde\w}(x) - g_\w(\nu_0+x) \right| = ( 1 + 2 W_{\nu_0-1} ) \left|R_{\nu_1}-R_{\nu_0}\right|. 
\]
\end{proof}

We now return to the proof of Proposition \ref{pr:PQsystemcouple}. 
\begin{proof}[Proof of Proposition \ref{pr:PQsystemcouple}]

We begin by giving a brief overview of the outline of the proof. 
Recall that in coupling $(\w,\tw)$ defined in \eqref{PQcouple}, the site $x$ in the environment $\tw$ correspond to either $\nu_0+x$ if $x<0$ or $\nu_1+x$ if $x\geq 0$ in the environment $\w$; we will refer to these as \emph{corresponding sites} in the two environments. 
The coupling of the systems of RWRE that we will construct will have two parts to it. First of all, we will create a coupling of the initial configurations 
so that the number of particles started at corresponding sites in the two environments are approximately the same. 
Secondly, we will couple as many of the particles as possible in the system $\tilde\chi$ to particles started at corresponding sites in the system $\chi$, and we will couple the paths of these random walks so that the difference between them is bounded for all time (for each fixed pair of walks the bound will be a finite random variable).


Using the standard coupling of Poisson random variables with different parameters we can construct a coupling of the initial configurations $\chi_0$ and $\tilde\chi_0$ with the required distributions so that 
\begin{align}
  | \tilde\chi_0(x) - \chi_0(\nu_1+x) | &\sim \text{Poisson}\left( g_{\tilde\w}(x)| u_0(\tfrac{x}{n}) - u_0(\tfrac{\nu_1+x}{n}) | \right), &\quad \forall x\geq 0 \label{ictildechicoupleR} \\
 | \tilde\chi_0(x) - \chi_0(\nu_0+x) | &\sim \text{Poisson}\left( | g_{\tilde\w}(x) u_0(\tfrac{x}{n}) - g_\w(\nu_0+x) u_0(\tfrac{\nu_0+x}{n}) | \right), &\quad \forall x < 0, \label{ictildechicoupleL}
\end{align}
and such that all the above Poisson random variables are independent. 
(Note that in \eqref{ictildechicoupleR} we used that $g_{\tilde\w}(x) = g_\w(\nu_1+x)$ for all $x\geq 0$.)
Given these initial configurations, let 
\[
 \gamma(x) = 
\begin{cases}
 \min\{ \chi_0(\nu_0 + x), \, \tilde\chi_0(x) \} & \text{ if } x < 0 \\
 \min\{ \chi_0(\nu_1 + x), \, \tilde\chi_0(x) \} & \text{ if } x \geq 0. 
\end{cases}
\]
That is, $\gamma(x)$ gives the number of particles started at $x$ in the environment $\tw$ that are matched with a particle started at the corresponding site in the environment $\w$ (either $\nu_0 + x$ if $x<0$ or $\nu_1+x$ if $x\geq 0$). 
The Poisson random variables in \eqref{ictildechicoupleR} and \eqref{ictildechicoupleL} then give the unmatched particles at one of the corresponding sites in either $\w$ or $\tw$. 
In addition, since the sites $x \in [\nu_0,\nu_1)$ in the environment $\w$ have no corresponding site in $\tw$ all of the particles in $[\nu_0,\nu_1)$ are also unmatched. 
Therefore, the total number of unmatched particles in either system is 
\[
 \mathfrak{U}_n = \sum_{x<0} | \tilde\chi_0(x) - \chi_0(\nu_0+x) | + \sum_{x\geq 0} | \tilde\chi_0(x) - \chi_0(\nu_1+x) |  + \sum_{x \in [\nu_0,\nu_1)} \chi_0(x),
\]
which is a Poisson random variable with mean that is bounded above by
\begin{align*}
&\sum_{x<0} \left|g_{\tw}(x)u\left(\tfrac{x}{n}\right) - g_\w(\nu_0+x)u\left(\tfrac{\nu_0+x}{n}\right) \right|  + \sum_{x\geq 0} g_{\tw}(x)\left|u\left(\tfrac{x}{n}\right) - u\left(\tfrac{\nu_1+x}{n}\right) \right| + \sum_{x=\nu_0}^{\nu_1-1} g_\w(x) u\left(\tfrac{x}{n}\right) \\
&\quad \leq \sum_{x<0} g_{\tw}(x) \left|u\left(\tfrac{x}{n}\right) - u\left(\tfrac{\nu_0+x}{n}\right) \right| + \sum_{x\geq 0} g_{\tw}(x)\left|u\left(\tfrac{x}{n}\right) - u\left(\tfrac{\nu_1+x}{n}\right) \right| \\
&\quad\qquad + \sum_{x<0} \left|g_{\tw}(x) - g_\w(\nu_0+x)\right|u\left(\tfrac{\nu_0+x}{n}\right) + \sum_{x=\nu_0}^{\nu_1-1} g_\w(x) u\left(\tfrac{x}{n}\right) \\
&\quad \leq \Delta\left(u; \tfrac{\nu_1 - \nu_0}{n} \right) \sum_{|x|\leq Ln} g_{\tilde\w}(x) 
 + \|u\|_\infty \left\{ \sum_{x<0} |g_{\tilde\w}(x) - g_\w(\nu_0+x)| +  \sum_{x =\nu_0}^{\nu_1-1}  g_\w(x) \right\},
\end{align*}
We claim that the number of unmatched particles is negligible on the scale of the hydrodynamic limit. That is, we will show that $n^{-1/\k} \mathfrak{U}_n$ converges to 0 in probability (under the averaged measure). Indeed, for any $\e>0$ and $A<\infty$, 
\begin{equation}\label{unmatchedub}
\begin{split}
 E_{\mathcal{P}}\left[ P_{(\w,\tilde\w)}^{\mu_u^n(\w),\mu_u^n(\tilde\w)} \left( \mathfrak{U}_N \geq \e n^{1/\k} \right) \right] 
&\leq \mathcal{P}\left( \nu_1- \nu_0 \geq \sqrt{n} \right) + \mathcal{P}\left( \sum_{|x| \leq Ln} g_{\tilde\w}(x) \geq A n^{1/\k} \right) \\
&\quad + \mathcal{P}\left( \sum_{x<0} |g_{\tilde\w}(x) - g_\w(\nu_0+x)| + \sum_{x=\nu_0}^{\nu_1-1} g_\w(x)  \geq  n \right) \\
&\quad + \Pv\left(\text{Poisson}\left( \Delta\left(u;\tfrac{1}{\sqrt{n}} \right) A n^{1/\k} + \|u_0\|_\infty n \right) \geq \e n^{1/\k} \right), 
\end{split}
\end{equation}
where in the last line we use the notation Poisson($\mu$) to denote a Poisson random variable with mean $\mu$. 
Since the random variables in the first and third probabilities on the right do not depend on $n$ (and are finite by Lemma \ref{PQcouplediff}), clearly these probabilities vanish as $n\ra\infty$. 
Also, since $\Delta(u;1/\sqrt{n}) \ra 0$ as $n\ra\infty$, the Poisson random variable in the fourth probability has a mean that is $o(n^{1/\k})$. Thus, the fourth probability also vanishes as $n\ra\infty$ for any choice of $A<\infty$ and $\e>0$.
For the second probability on the right above, we claim that 
\begin{equation}\label{gsum}
 \lim_{A\ra\infty} \limsup_{n\ra\infty} Q\left( \sum_{|x| \leq Ln} g_{\w}(x) \geq A n^{1/\k} \right)  = 0.
\end{equation}
(Note that since the coupling measure $\mathcal{P}$ for $(\w,\tw)$ has marginal $Q$ on $\tw$, the probability in \eqref{gsum} is the same as the second probability on the right in \eqref{unmatchedub}.)
This can be proved using results from \cite{dgWQL}, but we will give a proof using the techniques of the current paper instead. 
To this end, first note that 
\begin{align*}
 \frac{1}{n^{1/\k}} \sum_{|x|\leq Ln} g_{\w}(x) &= \frac{1}{n^{1/\k}} \sum_{|x|\leq Ln} \sum_{k: \, x < \nu_{k+1}} b_{x,k} \\
&\leq \frac{1}{n^{1/\k}} \sum_{|k|\leq Ln+1} \sum_{x < \nu_{k+1}} b_{x,k} + \frac{1}{n^{1/\k}} \sum_{k>Ln+1} \sum_{x\leq Ln} b_{x,k} \\
&\leq \frac{1}{n^{1/\k}} \sum_{|k|\leq Ln+1} \b_k + \frac{1}{n^{1/\k}} \sum_{k>Ln+1} \sum_{x\leq Ln} b_{x,k}.
\end{align*}
Under the measure $Q$, the first sum in the last line converges in distribution to a $\k$-stable random variable by Corollary \ref{cor:Wnweakconv} and  \eqref{bxksumsep} implies that the last sum converges to 0 in $Q$-probability, and from this \eqref{gsum} follows easily. 

We have shown that the initial configurations can be coupled in such a way so that the number of unmatched particles in the two systems is negligible on the hydrodynamic scale. It remains to show that for the matched particles, we can couple the random walks in the different environments so that the difference between the particles remains small enough (in fact we will show that the difference remains finite for each pair of coupled walks). 
We begin by introducing some notation. We will use $X_\cdot^{x,j}$ to denote the path of the $j$-th random walk started at location $x$ in the environment $\w$ and similarly $\tilde{X}_\cdot^{x,j}$ will denote the path of the $j$-th random walk started at $x$ in the environment $\tw$. 
For any $x \in \Z$ and $j \leq \gamma(x)$ we will attempt to couple $\tilde{X}_n^{x,j}$ with $X_n^{\nu_0+x,j}$ or $X_n^{\nu_1+x,j}$ depending on whether $x<0$ or $x\geq 0$, respectively. 
For convenience of notation, we will let
\[
 \bar{X}_n^{x,j} = \begin{cases} X_n^{\nu_0+x,j} &  \text{if } x< 0 \\ X_n^{\nu_1+x,j} &  \text{if } x\geq 0, \end{cases}
\]
so that our coupling will be for $\tilde{X}_n^{x,j}$ and $\bar{X}_n^{x,j}$. 
Given this notation, our goal will be to construct a coupling of the walks so that 
\begin{equation}\label{bXtXsystems}
 \lim_{n\ra\infty} E_{\mathcal{P}}\left[ P_{\w,\tw}^{\mu_u^n(\w),\mu_u^n(\tw)}\left( \sup_{t\leq T} \left| \sum_x \sum_{j=1}^{\gamma(x)} \phi\left(t, \frac{\bar{X}_{tn^{1/\k}}^{x,j}}{n} \right) - \phi\left(t, \frac{\tilde{X}_{tn^{1/\k}}^{x,j}}{n} \right) \right| \geq \e n^{1/\k} \right) \right] = 0, 
\end{equation}
for any $\e>0$. 

We now describe the coupling of $\tilde{X}_n^{x,j}$ and $\bar{X}_n^{x,j}$. We will couple the random walks so that on the $m$-th visit to corresponding sites, they both move in the same way. 
To make this precise, let $\mathbf{U} = \{U_{y,m}^{x,j}\}_{x,y\in\Z,m,j\geq 1}$ be an i.i.d.\ family of $U(0,1)$ random variables. Then, given the environment $\w$ and the family $\mathbf{U}$ of uniform random variables we construct the random walks in environment $\w$ as follows. 
\[
X_{n+1}^{x,j} = X_n^{x,j} + 2 \ind{U_{y,m}^{x,j} \leq \w_y} - 1 \qquad \text{ if } X_n^{x,j} = y \text{ and } \#\{i \leq n: X_i^{x,j} = y \} = m.
\]
To construct the coupled random walks in the environment $\tilde\w$ we will give a coupled family of uniform random variables $\tilde{U} = \{ \tilde{U}_{y,m}^{x,j} \}_{x,y \in \Z, m,j \geq 1}$ by letting 
\[
 \tilde{U}_{y,m}^{x,j} = U_{\tilde y, m}^{\tilde x,j}, \quad \text{where} \quad 
\tilde x 
=\begin{cases}
  \nu_0(\w) + x &  \text{if } x < 0 \\
  \nu_1(\w) + x &  \text{if } x \geq 0,
\end{cases}
\quad\text{and}\quad
\tilde y 
=\begin{cases}
  \nu_0(\w) + y &  \text{if } y < 0 \\
  \nu_1(\w) + y &  \text{if } y \geq 0
 \end{cases}.
\]
Then, given the environment $\tilde\w$ and the family $\mathbf{\tilde{U}}$ of uniform random variables, the random walks in environment $\tilde\w$ are constructed by letting
\[
\tilde X_{n+1}^{x,j} = \tilde X_n^{x,j} + 2 \ind{\tilde U_{y,m}^{x,j} \leq \tilde\w_y} - 1 \qquad \text{ if } \tilde X_n^{x,j} = y \text{ and } \#\{i\leq n: \tilde X_i^{x,j} = y \} = m.
\]
Given the above coupling of the random walks, it is easy to see that the steps of the random walks are identical up until the first time the random walk in environment $\w$ reaches the interval $[\nu_0,\nu_1)$ (or equivalently when the random walk in $\tilde\w$ crosses the edge between $-1$ and $0$ for the first time) and that upon each exit of this interval to the right the steps of the random walk in $\w$ match the steps of the random walk in $\tilde\w$ during excursions to the right of $0$. 
Thus, the difference in position of the two walks is bounded by
their initial difference $|\bar{X}_0^{x,j} - \tilde{X}_0^{x,j}| \leq \nu_1-\nu_0$ plus
 the amount of time the walk in $\w$ spends to the left of $\nu_1$ after first reaching $[\nu_0,\nu_1)$ and the amount of time the random walk in $\tilde\w$ spends to the left of the origin after first reaching the origin. 
Note that since the random walks are transient to the right, for any $x \in \Z$ and $j\geq 1$ these quantities are stochastically dominated by the random variable $\mathfrak{D}$ defined by 
\begin{equation}\label{Vdef}
 \mathfrak D = \nu_1 - \nu_0 + \sum_{n=0}^\infty \ind{ X_n^{\nu_0} < \nu_1 } + \sum_{n=0}^\infty \ind{\tilde{X}_n^0 < 0}. 
\end{equation}
Therefore, we can expand the measure $P_{\w,\tw}^{\mu_u^n(\w),\mu_u^n(\tw)}$ to include an i.i.d.\ family of random variables $\{\mathfrak{D}_{x,j}\}_{x\in \Z,\, j\geq 1}$ all with the same distribution a $\mathfrak{D}$ in \eqref{Vdef} and such that $\sup_n |\bar{X}_n^{x,j} - \tilde{X}_n^{x,j}| \leq \mathfrak{D}_{x,j}$ for every $x \in \Z$ and $j\geq 1$. 
Thus, we can conclude that 
\begin{align}
& P_{\w,\tw}^{\mu_u^n(\w),\mu_u^n(\tw)}\left( \sup_{t\leq T} \left| \sum_x \sum_{j=1}^{\gamma(x)} \phi\left(t, \frac{\bar{X}_{tn^{1/\k}}^{x,j}}{n} \right) - \phi\left(t, \frac{\tilde{X}_{tn^{1/\k}}^{x,j}}{n} \right) \right| \geq \e n^{1/\k} \right) \nonumber \\
&\leq P_{\w,\tw}^{\mu_u^n(\w),\mu_u^n(\tw)}\left( \sum_x \sum_{j=1}^{\gamma(x)} \Delta\left(\phi; \frac{\mathfrak{D}_{x,j}}{n} \right) \geq \e n^{1/\k} \right) \nonumber \\
&\leq \frac{\|u\|_\infty E_{\w,\w'}\left[ \Delta\left(\phi; \frac{\mathfrak{D}}{n} \right) \right]}{\e n^{1/\k}} \sum_{x=-Ln}^{Ln} g_{\tw}(x). \label{Dgsum}
\end{align}
Since $\Delta\left(\phi; \frac{\mathfrak{D}}{n} \right)$ is bounded and converges to 0 almost surely under the averaged measure $E_{\mathcal{P}}[P_{\w,\tw}(\cdot)]$, it is easy to see that 
$E_{\w,\w'}\left[ \Delta\left(\phi; \frac{\mathfrak{D}}{n} \right) \right]$ converges to 0 in $\mathcal{P}$-probability. 
Together with \eqref{gsum} this implies that \eqref{Dgsum} converges to 0 in $\mathcal{P}$-probability, and this in turn implies \eqref{bXtXsystems}. 
\end{proof}

\appendix

\section{A PDE characterization \texorpdfstring{of $u_W$}{}}\label{app:uWunique}

As mentioned in Remark \ref{rem:uWPDE}, hydrodynamic limits are often described via a solution of some partial differential equation. 
However, in our proof of the hydrodynamic limits in Theorem \ref{th:RWREhdl} and \ref{th:dthdl} we defined the function $u_W(t,x)$ probabilistically instead of as a solution to some PDE. 
We showed certain differentiability properties of this function $u_W(t,x)$ in Proposition \ref{prop:duWdt}, but it is not clear that these differentiability properties uniquely characterize the function $u_W(t,x)$. 
Indeed, if we only require that \eqref{duWdteq1} holds (differentiable in $t$ with time derivative equal to the negative of the \emph{left} spatial derivative with respect to $\s_W$), then it is easily seen that 
the function $v(t,x) \equiv u(x)$ for all $t\geq 0$ satisfies \eqref{duWdteq1}. However, the function $v(t,x)$ is different from $u_W(t,x)$ since $v$ does not satisfy the second differentiability property \eqref{duWdteq2} (time derivative equal to the negative of the \emph{right} spatial derivative with respect to $\s_W$). 
We suspect that the differentiability properties \eqref{duWdteq1} and \eqref{duWdteq2} together do uniquely characterize the function $u_W(t,x)$, but we are currently not able to prove this. 
Instead, we will give a slightly different characterization of the function $u_W$ in the case when $u$ is a function of bounded variation.

\begin{defn}
 A function $f$ is of bounded variation on $[a,b]$ if 
 \[
  V_{[a,b]}(f) = \sup_{a=x_0<x_1<\ldots<x_n=b} \sum_{i=1}^n |f(x_i) - f(x_{i-1})| < \infty. 
 \]
(The supremum in the above definition is taken over all finite partitions of $[a,b]$.)
The total variation of a function on all of $\R$ is 
\[
 V(f) = \lim_{\substack{a\ra -\infty\\b\ra\infty}} V_{[a,b]}(f). 
\]
The collection of all functions of bounded variation on $\R$ will be denoted by 
\[
 \mathrm{BV}(\R) = \{f: V(f) < \infty\}. 
\]
\end{defn}

\begin{lem}\label{lem:uWprops}
 If $W \in \mathcal{T}'$ and $u \in \mathcal{C}_0^+ \cap \mathrm{BV}(\R)$, then $u_W(t,x)$ is the unique function with the following properties. 
 \begin{itemize}
  \item $u_W(0,x) = u(x)$ for all $x\in \R$. 
  \item $\sup_{t\geq 0, \, x \in \R} |u_W(s,x)| < \infty.$
  \item For any fixed $t>0$, the function $x \mapsto u_W(t,x)$ is right-continuous with left limits and also of bounded variation on $\R$. 
Moreover, the measure $u_W(t,dx)$ is absolutely continuous with respect to $\s_W(dx)$. 
  \item For any fixed $x_k \in \mathcal{J}_W$, $u_W(t,x_k)$ is differentiable with respect to $t$, 
with 
\begin{equation}\label{duWdsW}
 - \frac{\del}{\del t} u_W(t,x_k) =
\begin{cases}
 0 & \text{if } t = 0 \\
\frac{du_W(t,\cdot)}{d\s_W}(x_k) & \text{if } t > 0,
\end{cases}
\quad 
\end{equation}
and such that 
\begin{equation} \label{intdudt}
\sup_{t\geq 0} \int_\R \left| \frac{\del}{\del t} u_W(t,x) \right| \, \s_W(dx) < \infty. 
\end{equation}
\end{itemize}
\end{lem}

\begin{rem}
 Note that \eqref{duWdsW} is equivalent to the statement that 
 \[
  u_W(t,b) - u_W(t,a) = \int_{(a,b]}  -\frac{\del}{\del t}u_W(t,x) \, \s_W(dx), \quad \forall t>0 \text{ and } a < b. 
 \]
\end{rem}

\begin{proof}
 We begin by proving that the function $u_W(t,x)$ has the properties claimed in the statement of the lemma. 
The first two properties are immediate consequences of the definition \eqref{uWdef}. 
 We've already shown in Section \ref{sec:uW} that $u_W(t,x)$ is right-continuous with left limits in $x$ and is differentiable in $t$ whenever $x \in \mathcal{J}_W$ with 
 \begin{equation}\label{pointdiff}
  \frac{\del}{\del t} u_W(t,x_k) = \frac{u_W(t,x_k-) - u_W(t,x_k)}{y_k}, \quad \forall x_k \in \mathcal{J}_W. 
 \end{equation}
Note that if we can show that $u_W(t,dx)$ is absolutely continuous with respect to $\s_W$, then \eqref{duWdsW} clearly follows from \eqref{pointdiff}. 
Thus, it remains to show \eqref{intdudt} and that $u_W(t,dx)$ is absolutely continuous with respect to $\s_W$. 
To prove \eqref{intdudt} we will show that 
 \begin{equation}\label{uWbv}
\int_\R \left| \frac{\del}{\del t} u_W(t,x) \right| \s_W(dx) = 
 \sum_k |u_W(t,x_k) - u_W(t,x_k-)| \leq V(u) < \infty, \quad \forall t\geq 0. 
 \end{equation}
The first equality follows from \eqref{pointdiff}. To prove the inequality in \eqref{uWbv}
recall from Lemma \ref{lem:uWc} that $u_W(t,x_k-) = u_W^\circ(t,x_k) = \Ev[ u(Z_W^\circ(t;x_k)) ]$ and thus
\begin{align*}
\sum_k |u_W(t,x_k) - u_W(t,x_k-)|
 &\leq \sum_k \Ev\left| u(Z_W^*(t;x_k)) - u(Z_W^\circ(t;x_k)) \right| \\
 &= \Ev\left[ \sum_k \left| u(Z_W^*(t;x_k)) - u(Z_W^\circ(t;x_k)) \right| \right],
\end{align*}
and since 
$Z_W^*(t;x_\ell) \leq Z_W^\circ(t;x_k) \leq Z_W^*(t;x_k)$ for all $x_\ell < x_k$, 
the sum inside the expectation in the last line is always bounded by $V(u)$.  
We claim that a further consequence of \eqref{uWbv} is that
\begin{equation}\label{rndsum}
  u_W(t,b) - u_W(t,a) 
  = \sum_{x_k \in (a,b]} \left( u_W(t,x_k) - u_W(t,x_k-) \right)
  \quad \forall t>0, \text{ and } a < b. 
\end{equation}
To see this, we first note that 
\begin{align*}
 u_W(t,b) - u_W(t,a) 
 &= \Ev\left[ u(Z_W^*(t;b)) - u(Z_W^*(t;a)) \right] \\
 &= \Ev\left[ \sum_{x_k \in (a,b]} u(Z_W^*(t;x_k)) - u(Z_W^\circ(t;x_k)) \right],
\end{align*}
where the last equality above is justified by the fact that $u$ is a function of bounded variation  and
that the half-open intervals $\left( Z_W^\circ(t;x_k),Z_W^*(t;x_k) \right] $ with $x_k \in (a,b]$ are disjoint and cover all but countably many points in $\left( Z_W^*(t;a),Z_W^*(t;b) \right]$. 
Due to \eqref{uWbv}, we can interchange the expectation and summation in the last line above to obtain \eqref{rndsum}. 
Combining \eqref{uWbv} and \eqref{rndsum}, we see that $u_W(t,\cdot)$ is of bounded variation with $V(u_W(t,\cdot)) \leq V(u)$ for all $t > 0$. 
Since $u_W(t,\cdot)$ is a function of bounded variation, the finite Borel measure $u_W(t,dx)$ is well defined. \eqref{rndsum} shows that this measure agrees with a measure concentrated on $\mathcal{J}_W$ for all subsets of the form $(a,b]$. Since the half-open intervals $(a,b]$ uniquely determine the Borel measures we can conclude that $u_W(t,dx)$ is absolutely continuous with respect to $\s_W(dx)$. 

Having shown that $u_W(t,x)$ satisfies all of the claimed properties in the statement of the Lemma, we now turn to the proof of the uniqueness. 
To this end, we first note that for any $g \in \mathcal{C}_0^+ \cap \text{BV}(\R)$ and any $t>0$ we can define a function $f$ on $[0,t]\times \R$ by 
\begin{equation}\label{fdef}
 f(s,x) = \Ev\left[ g( Z_W(t-s; x) ) \right], \qquad s \in [0,t], \, x \in \R. 
\end{equation}
Similarly to the above argument that $u_W$ satisfies the properties stated in the lemma, it can be shown that the function $f$ has the following properties. 
\begin{itemize}
 \item $f(t,x) = g(x)$ for all $x\in \R$. 
 \item $f$ is uniformly bounded. That is, $\sup_{x\in \R, \, s \in [0,t]} |f(s,x)| < \infty$. 
 \item For any $s\in [0,t)$ the function $x\mapsto f(s,x)$ is left-continuous with right limits and of bounded variation on $\R$. Moreover, the measure $f(s,dx)$ is absolutely continuous with respect to $\s_W(dx)$.  
 \item For any fixed $x_k \in \mathcal{J}_W$, $f(s,x_k)$ is differentiable with respect to $s \in [0,t]$, with 
\begin{equation}\label{dfds}
 -\frac{\del}{\del s} f(s,x_k) =
\begin{cases}
 \frac{df(s,\cdot)}{d\s_W}(x_k) & \text{if } s \in [0,t) \\
 0 & \text{if } s = t, 
\end{cases}
\end{equation}
and such that 
\begin{equation}\label{intdfdt}
 \sup_{s \in [0,t]} \int_\R \left| \frac{\del}{\del s} f(s,x) \right| \s_W(dx) < \infty. 
\end{equation}
\end{itemize}

Now, suppose that $u(t,x)$ is another function satisfying all of the properties in the statement of Lemma \ref{lem:uWprops}. 
Let $g \in \mathcal{C}_0^+ \cap \text{BV}(\R)$ and $t>0$ be fixed and let $f$ be defined as in \eqref{fdef}. Then,
\begin{align*}
& \int_\R u(t,x) g(x) \s_W(dx) = \int_\R u(t,x) f(t,x) \s_W(dx) \\
&= \int_\R u(0,x) f(0,x) \s_W(dx) + \int_\R \int_0^t \frac{\del}{\del s} \left( u(s,x)f(s,x) \right) ds \, \s_W(dx) \\
&= \int_\R u(x) f(0,x) \s_W(dx) + \int_\R \int_0^t \left( f(s,x) \frac{\del}{\del s} u(s,x) + u(s,x) \frac{\del}{\del s} f(s,x) \right) ds \, \s_W(dx) \\
&= \int_\R u(x) f(0,x) \s_W(dx) +  \int_0^t \int_\R \left( f(s,x) \frac{\del}{\del s} u(s,x) + u(s,x) \frac{\del}{\del s} f(s,x) \right) \s_W(dx) \, ds,  
\end{align*}
where the application of Fubini's Theorem in the last equality is justified by \eqref{intdudt}, \eqref{intdfdt} and the boundedness of $u$ and $f$. 
To simplify the double integral in the last line above, note that for any $s \in (0,t)$ it follows from \eqref{duWdsW} and \eqref{dfds} that 
\begin{align*}
& \int_\R \left( f(s,x) \frac{\del}{\del s} u(s,x) + u(s,x) \frac{\del}{\del s} f(s,x) \right) \s_W(dx) \\
&\qquad = \int_\R f(s,x) u(s,dx) + \int_\R u(s,x) f(s,dx) = 0,
\end{align*}
where the last equality follows from integration by parts since $u$ and $f$ are both of bounded variation, $u(s,\cdot)$ is right continuous and $f(s,\cdot)$ is left continuous. We have shown that 
\[
 \int_\R u(t,x) g(x) \s_W(dx) = \int_\R u(x) f(0,x) \s_W(dx), \qquad \forall g \in \mathcal{C}_0^+ \cap \text{BV}(\R), \, t > 0. 
\]
Since this is also true for the function $u_W(t,x)$, we can conclude that $u(t,x_k) = u_W(t,x_k)$ for all $t>0$ and $x_k \in \mathcal{J}_W$. However, since $W \in \mathcal{T}'$ then $\mathcal{J}_W$ is dense in $\R$ and since $u$ and $u_W$ are both right continuous in $x$ it follows that $u(t,x) = u_W(t,x)$ for all $t\geq 0$ and all $x\in \R$. 
\end{proof}

\section{Weak convergence of the random trap environments}\label{app:Wnweakconv}

In this appendix, we will give the proof of the Poissonian limit of the trap structure as stated in Corollary \ref{cor:Wnweakconv} as well as the proofs of Corollaries \ref{cor:bstable} and \ref{cor:tBstable}. 
All of these Corollaries are a consequence of Lemma \ref{lem:barWnconv}. 


\subsection{Poissonian limit for the trap structure}

\begin{proof}[Proof of Corollary \ref{cor:Wnweakconv}]
We begin by showing that $W_n$ converges in distribution to $W$.
Since Lemma \ref{lem:barWnconv} implies that $\overline{W}_n$ converges in distribution to $W$, 
we claim that it will be enough to show that 
\begin{equation}\label{WnbWn}
  \lim_{n\ra\infty} Q\left( \left| \langle \phi, W_n \rangle - \langle \phi, \overline{W}_n \rangle \right| \geq \d \right) = 0, \quad \forall \d>0, \, \phi \in \mathcal{C}_0^+(\R \times (0,\infty] ). 
\end{equation}
To see that \eqref{WnbWn} is sufficient, note that $\mathcal{M}_p$ is a Polish space under a metric $d_{\text{vague}}$ which is given by 
\[
 d_{\text{vague}}(M,M') = \sum_{k=1}^\infty 2^{-k} \left( 1- e^{-|\langle \phi_k, M\rangle - \langle \phi_k, M' \rangle|} \right), \quad \forall M,M' \in \mathcal{M}_p,
\]
where $\{\phi_k\}_{k\geq 1}$ is a certain fixed sequence of continuous functions with compact support. 
Clearly \eqref{WnbWn} then implies that 
\[
 \lim_{n\ra\infty} Q\left( d_{\text{vague}}( W_n, \overline{W}_n ) \geq \d \right) = 0, \quad \forall \d>0, 
\]
from which it follows that $W_n$ has the same limiting distribution as $\overline{W}_n$. 

To prove \eqref{WnbWn}, first note that 
\[
   \left| \langle \phi, W_n \rangle - \langle \phi, \overline{W}_n \rangle \right| \leq  \sum_k \left| \phi\left( \frac{\nu_k}{n},\frac{\b_k}{n^{1/\k}} \right) - \phi\left( \frac{k \bar\nu}{n},\frac{\b_k}{n^{1/\k}} \right) \right|.
\]
Now, suppose that $\phi \in \mathcal{C}_0^+(\R \times (0,\infty] )$ has support contained in $[-L,L]\times [\e,\infty]$. 
Since $\phi$ is uniformly continuous, for any $\e'>0$ there exists a $\d'>0$ such that $|\phi(x,y)-\phi(x',y)| < \e'$ if 
$|x-x'|< \d'$. 
Therefore, we can conclude that 
\begin{align*}
 Q\left( \left| \langle \phi, W_n \rangle - \langle \phi, \overline{W}_n \rangle \right| \geq \d \right) 
&\leq Q\left( \max_{|k|\leq Ln} |\nu_k - k \bar\nu| \geq \d' n \right) + Q\left( \sum_{k=-Ln}^{Ln} \ind{\b_k \geq \e n^{1/\k}} \geq \frac{\d}{\e'} \right) \\
&\leq Q\left( \max_{|k|\leq Ln} |\nu_k - k \bar\nu| \geq \d' n \right) + \frac{(2Ln+1)\e'}{\d}Q( \b_1 \geq \e n^{1/\k} )
\end{align*}
As noted in \eqref{unifnuk} above, the first probability on the right vanishes as $n\ra\infty$ for any $\d'>0$. 
The tail decay of $\b_1$ implies that the second term on the right
vanishes as first $n\ra\infty$ and then $\e'\ra 0$. 
This completes the proof of \eqref{WnbWn}, and thus we can conclude that $W_n$ converges in distribution to $W$.

Next, recalling the notation 
$M^{(\e)} = M(\cdot \cap \R \times [\e,\infty])$
for any point process $M \in \mathcal{M}_p(\R \times \R_+)$, 
we claim that the mapping $M\mapsto \s_{M^{(\e)}}$ from $\mathcal{M}_p (\R \times \R_+) \ra D_\R^J$ is continuous 
on the set 
\[
 \mathcal{A}_\e = \{ M \in \mathcal{M}_p :\, M(\R \times \{\e,\infty\}) = 0 \}. 
\]
Indeed, let $M_n \ra M \in \mathcal{A}_\e$ and fix $L<\infty$ with the property that $M(\{-L,L\}\times (0,\infty]) = 0$ (since $M$ has only countably many atoms, there are only countably many $L$ for which this does not hold). Then for $n$ sufficiently large we can enumerate the point processes $M_n$ and $M$ so that
\[
 M_n(\cdot \cap [-L,L]\times [\e,\infty] ) = \sum_{k=1}^K \d_{(x_k^n,y_k^n)}, \quad\text{and}\quad M(\cdot \cap [-L,L]\times [\e,\infty] ) = \sum_{k=1}^K \d_{(x_k,y_k)},
\]
and 
\[
 \lim_{n\ra\infty} \max_{k\leq K} |x_k^n-x_k| \vee |y_k^n-y_k| = 0. 
\]
(Of course, in the above notation $K$ must be given by $K = M([-L,L]\times [\e,\infty]) < \infty$.) 
It follows that for $n$ sufficiently large we have that 
\[
 d_{[-L,L]}^{J_1}( \s_{M_n^{(\e)}}, \s_{M^{(\e)}} ) 
\leq \max \left\{ \max_{k\leq K} |y_k^n - y_k|,\, \sum_{k=1}^K |x_k^n - y_k^n|  \right\}.
\]
From this we can conclude that $d_{[-L,L]}^{J_1}( \s_{M_n^{(\e)}}, \s_{M^{(\e)}} ) \ra 0$ as $n\ra\infty$ for almost every $L<\infty$. This is enough to conclude that $\s_{M_n^{(\e)}}$ converges to $\s_{M^{(\e)}}$ in $D_\R^J$. 

Since the Poisson point process $W$ satisfies $\Pv( W \in \mathcal{A}_\e) = 1$, the continuous mapping theorem and the fact that $W_n$ converges in distribution to $W$ then imply that $(W_n,\s_{W_n^{(\e)}})$ converges in distribution to $(W,\s_{W^{(\e)}})$. 
Since
\[
 \sup_{|x|\leq L} |\s_{W^{(\e)}}(x) - \s_W(x)| \leq \int_{-L}^L \int_0^\e y \, W(dx \, dy) \underset{\e\ra 0}{\longrightarrow} 0,
\]
it follows that $\s_{W^{(\e)}} \ra \s_W$ in $D_\R^J$, and thus to conclude that $(W_n,\s_{W_n})$ converges in distribution to $(W,\s_W)$ we need only to show that 
\[
 \lim_{\e\ra 0} \limsup_{n\ra\infty} Q\left( \sup_{|x|\leq L} \left| \s_{W_n^{(\e)}}(x) - \s_{W_n}(x) \right| \geq \d \right) = 0, \quad \forall L< \infty, \, \d > 0.  
\]
Since  
\[
 \sup_{|x|\leq L} \left|\s_{W_n^{(\e)}}(x) - \s_{W_n}(x) \right| \leq \int_{-L}^L \int_0^\e y \, W_n(dx \, dy)
\leq \frac{1}{n^{1/\k}} \sum_{k=-Ln}^{Ln} \b_k \ind{\b_k \leq \e n^{1/\k}}, 
\]
this will follow from 
\begin{equation}\label{betatsum}
 \lim_{\e \ra 0} \limsup_{n\ra\infty} Q\left( \sum_{|k|\leq L n} \b_k \ind{\b_k < \e n^{1/\k}} \geq \d n^{1/\k} \right) = 0, \quad \forall L<\infty, \d>0. 
\end{equation}
However, 
\begin{align*}
 \limsup_{n\ra\infty} Q\left( \sum_{|k|\leq L n} \b_k \ind{\b_k < \e n^{1/\k}} \geq \d n^{1/\k} \right)
&\leq \limsup_{n\ra\infty} \frac{2Ln+1}{\d n^{1/\k}} E_Q\left[ \b_1 \ind{\b_1 < \e n^{1/\k}} \right]\\
&= \frac{2 L C_1 \k}{\d(1-\k)} \e^{1-\k},
\end{align*}
where the last equality follows from the tail decay of $\b_1$ in \eqref{bnutails}.
Since $\k \in (0,1)$, this is enough to prove \eqref{betatsum}. 
\end{proof}

\subsection{Stable limits in Corollaries \ref{cor:bstable} and \ref{cor:tBstable}}

\begin{proof}[Proof of Corollary \ref{cor:bstable}]
 It is known that if $W = \sum_k \d_{(x_k,y_k)}$ is a Poisson point process with intensity measure $\l y^{-\k-1} \, dx \, dy$, then $\langle y \ind{|x|\leq a}, W \rangle = \sum_k y_k \ind{|x_k|\leq a}$ is a $\k$-stable random variable for any $a>0$. 
Since 
\[
 \frac{1}{n^{1/\k}} \sum_{|k|\leq n} \b_k = \langle y \ind{|x|\leq \bar\nu}, \overline{W}_n \rangle,
\]
we would like to use Lemma \ref{lem:barWnconv} to conclude that $n^{-1/\k}\sum_{|k|\leq n} \b_k$ also converges in distribution to  $\langle y \ind{|x|\leq \bar\nu}, W \rangle $. 
Unfortunately, the mapping $M \mapsto \langle y \ind{|x|\leq \bar\nu}, M \rangle$ from $\mathcal{M}_p \ra \R$ is not a continuous mapping. 
However, the mapping 
\[
 M \mapsto \langle y \ind{|x|\leq \bar\nu, \, y \geq \e}, M \rangle
\]
is continuous on the set $\{M \in \mathcal{M}_p: \, M(\R \times \{\e\}) = M(\{-\bar\nu, \bar\nu \}\times (0,\infty]) = 0 \} $. 
Since for any $\e>0$ the Poisson point process $W$ belongs to this set with probability 1, we can conclude from Lemma \ref{lem:barWnconv} that 
\[
 \frac{1}{n^{1/\k}} \sum_{|k|\leq n} \b_k \ind{\b_k \geq \e n^{1/\k}} 
\Longrightarrow
\langle y \ind{|x|\leq \bar\nu, \, y \geq \e}, W \rangle, \quad \text{as } n\ra\infty,
\]
for any $\e>0$. 
The proof of the corollary then follows from \eqref{betatsum} and the fact that $$\langle y \ind{|x|\leq \bar\nu, \, y \geq \e}, W \rangle \underset{\e\ra 0}{\longrightarrow} \langle y \ind{|x|\leq \bar\nu}, W \rangle,$$ with probability one.
\end{proof}

\begin{proof}[Proof of Corollary \ref{cor:tBstable}]
Let $S_n(t) = n^{-1/\k} \tau_{\mathfrak{B}}([0,nt))$ be the hitting time process for the directed traps. 
We claim that both limiting distributions in the statement of the corollary will follow if we can show that the hitting time process
$\{S_n(t)\}_{t\geq 0}$ converges in distribution to a $\k$-stable subordinator $\{S(t)\}_{t\geq 0}$ on the space $D_{\R_+}^{J}$. 
First, since $\nu_n/n \ra \bar\nu$, $Q$-a.s., it follows that $n^{-1/\k} \tau_\mathfrak{B}([0,\nu_n))$ converges in distribution to $S(\bar\nu)$, which is a $\k$-stable random variable. 
Next, let $D_{u,\uparrow}^+ \subset D_{\R_+}$ be the set of non-decreasing c\`adl\`ag functions $x(t)$ with $x(0)\geq 0$ and $x(t) \ra \infty$ as $t\ra\infty$ and let $\mathfrak{I}:D_{u,\uparrow}^+ \mapsto D_{u,\uparrow}^+$ be the space-time inversion operator given by 
\[
 \mathfrak{I}x(t) = \sup\{ s\geq 0: \, x(s) \leq t \}, \quad t \geq 0, x \in D_{u,\uparrow}^+. 
\]
Note that the construction of the directed trap process $Z_{\mathfrak{B}}$ is such that $n^{-1} Z_\mathfrak{B}(t n^{1/\k}) = \mathfrak{I}S_n(t)$.
Since it is known that the operator $\mathfrak{I}$ is continuous on the subset $D_{u,\uparrow\uparrow}^+ \subset D_{u,\uparrow}^+$ of functions that are \emph{strictly} increasing \cite[Corollary 13.6.4]{wSPL}, 
the continuous mapping theorem then implies that
$\{ t \mapsto n^{-1} Z_\mathfrak{B}(t n^{1/\k}) \}$ converges in distribution to $\{\mathfrak{I} S(t) \}_{t\geq 0}$. 
Since the inverse of a $\k$-stable subordinator $\mathfrak{I}S(t)$ is almost surely a continuous process, we can conclude that this convergence in distribution is with respect to the uniform topology on $D_{\R_+}$.

It remains to prove the claimed convergence of the hitting time process $S_n$. 
Recall that the crossing time $\tau_{\mathcal{B}}([0,nt))$ for the directed trap process in the trap environment $\mathfrak{B}$ is given by 
\[
 \tau_{\mathfrak{B}}([0,nt)) = \sum_{k} \b_k \zeta_k \ind{\nu_k \in [0,nt)}, 
\]
where $\{\zeta_k\}_k$ is an i.i.d.\ sequence of Exp(1) random variables that is independent of $\mathfrak{B}$. 
We then construct the point process
\[
 \Gamma_n = \sum_{k \in \Z} \d_{(\frac{\nu_k}{n}, \frac{\b_k \zeta_k}{n^{1/\k}} )}
\]
which is obtained by multiplying the $y$-coordinate of the atoms of $W_n$ by the independent random variables $\zeta_k$. It then follows from Corollary \ref{cor:Wnweakconv} that $\Gamma_n$ converges in distribution to a Poisson point process $\Gamma$ with intensity measure $\l' y^{-\k-1} \, dx \, dy$, where $\l' = \l \Gamma(\k+1)$ with $\l$ as in Corollary \ref{cor:Wnweakconv}.
Now, note that $t\mapsto S(t) =  \langle y\ind{x \in [0,t)}, \Gamma \rangle$ is a $\k$-stable subordinator and
\[
 S_n(t) = \frac{1}{n^{1/\k}} \tau_{\mathfrak{B}}([0,nt)) = \frac{1}{n^{1/\k}} \sum_{k} \b_k \zeta_k \ind{\nu_k \in [0,nt)} = \langle y\ind{x \in [0,t)}, \Gamma_n \rangle.
\]
The mapping $M\mapsto \{ \langle y \ind{x \in [0,t)}, M \rangle \}_{t\geq 0}$ is not a continuous mapping from $\mathcal{M}_p$ to $D_{\R_+}^J$, 
but for any $\e>0$ the mapping $M\mapsto \{ \langle y \ind{x \in [0,t), \, y\geq \e}, M \rangle \}_{t\geq 0}$
is continuous on the set $\{M \in \mathcal{M}_p: M(\R\times \{\e\}) = 0\}$. 
Then, similarly to the proof of Corollary \ref{cor:bstable}, we will be able to deduce the convergence of $S_n$ to $S$ if we can show that 
\begin{equation}\label{smallbz}
 \lim_{\e \ra 0} \limsup_{n\ra\infty} E_Q\left[ P_\w \left( \sum_{k=0}^{nT} \b_k \zeta_k \ind{\b_k \zeta_k < \e n^{1/\k} }\geq \d n^{1/\k} \right) \right] = 0, \quad \forall T<\infty, \, \d>0. 
\end{equation}
To see this, first note that the tail decay of $\b_1$ in \eqref{bnutails} and the fact that $\zeta_1$ is independent of $\b_1$ can be used to show that 
\[
 E_Q\left[ P_\w(\b_1 \zeta_1 \geq x) \right] \sim C_1 \Gamma(\k+1) x^{-\k}, \qquad \text{as } x\ra\infty.
\]
From this tail decay asymptotics, the proof of \eqref{smallbz} follows in the same way as the proof of \eqref{betatsum} above. 
This completes the proof of the convergence of the hitting times process $S_n(t)$, and thus also the proof of the Corollary. 
\end{proof}

\bibliographystyle{alpha}
\bibliography{RWRE}


\end{document}